\documentclass[reqno]{amsart}
 
\usepackage{amssymb,amsfonts,amsmath, amsthm}
\usepackage{lipsum} 
\usepackage[all,arc]{xy}
\usepackage{enumerate}
\usepackage{mathrsfs}
\usepackage[usenames, dvipsnames]{xcolor}
\definecolor{cerulean}{RGB}{0,123,167}
\usepackage[colorlinks = true, linkcolor = cerulean, citecolor = Thistle]{hyperref}
\usepackage{tikz}
\usepackage{tikz-cd}
\usepackage[super]{nth}
\usetikzlibrary{arrows}
\usepackage{caption}

\usepackage[left=.75 in, right=.75 in,top=.75 in, bottom=.75 in]{geometry}

\usepackage{multicol}
\usepackage{esint}
\usepackage{mathtools}
\usepackage{import}
\usepackage{xifthen}
\usepackage{pdfpages}
\usepackage{transparent}
\usepackage{upgreek}
\usepackage{pgfplots}
\pgfplotsset{compat=1.17} 
\usepackage{bbm}

\newcommand{\C}{\mathbb{C}}

\newcommand{\F}{\mathbb{F}}	

\newcommand{\R}{\mathbb{R}}
\newcommand{\Q}{\mathbb{Q}}
\newcommand{\N}{\mathbb{N}}
\newcommand{\Z}{\mathbb{Z}}
\newcommand{\K}{\mathbb{K}}

\newcommand{\ve}{\varepsilon}
\newcommand{\p}{\partial}

\newcommand{\abs}[1]{\left\vert#1\right\vert}
\newcommand{\tfloor}[1]{\left\lfloor#1\right\rfloor}
\newcommand{\nab}{\nabla}
\providecommand{\norm}[1]{\left\Vert#1\right\Vert}
\newcommand{\weakar}{\rightharpoonup} 

\newcommand{\triplenorm}[1]{{\left\vert\kern-0.25ex\left\vert\kern-0.25ex\left\vert #1 
    \right\vert\kern-0.25ex\right\vert\kern-0.25ex\right\vert}}

\DeclareMathOperator{\Tan}{tan}

\newcommand{\Hzerotan}{{}_{\Tan}H^1}

\newcommand{\Hzerostan}[1]{{}_{\Tan}H^{#1}}
\newcommand{\alphaHzeros}[1]{{}_{\alpha-\Tan}H^{#1}}
\newcommand{\Hzerostanigtan}{{}_{\Tan}H^1_\sigma}

\renewcommand{\Re}{\operatorname{Re}}
\renewcommand{\Im}{\operatorname{Im}}

\numberwithin{equation}{section}

\DeclareMathOperator{\esssup}{esssup}

\DeclareMathOperator{\Tr}{Tr}
\DeclareMathOperator{\diverge}{div}

\DeclareMathOperator{\Range}{Ran}
\DeclareMathOperator{\sym}{sym}

\newtheorem{thm}{Theorem}[section]
\newtheorem{cor}[thm]{Corollary}
\newtheorem{prop}[thm]{Proposition}
\newtheorem{lem}[thm]{Lemma}

\theoremstyle{definition}
\newtheorem*{thm*}{Theorem}
\newtheorem*{def*}{Definition}
\newtheorem*{prop*}{Proposition}
\newtheorem{defn}[thm]{Definition}

\theoremstyle{remark}

\newcommand{\m}{\mathrm}

\newcommand{\lv}{\lVert}
\newcommand{\rv}{\rVert}

\newcommand{\lf}{\lfloor}
\newcommand{\rf}{\rfloor}

\newcommand{\bpm}{\begin{pmatrix}}
\newcommand{\epm}{\end{pmatrix}}

\newcommand{\loc}{\m{loc}}
\renewcommand{\bar}{\overline}

\renewcommand{\le}{\leqslant}
\renewcommand{\ge}{\geqslant}

\newcommand{\tnorm}[1]{\lv#1\rv}

\title[Traveling wave solutions to the free boundary incompressible Navier-Stokes equations]{Traveling wave solutions to the free boundary incompressible Navier-Stokes equations with Navier boundary conditions}

\author{Junichi Koganemaru}
\address{
Department of Mathematical Sciences\\
Carnegie Mellon University\\
Pittsburgh, PA 15213, USA
}
\email[J. Koganemaru]{jkoganem@andrew.cmu.edu}

\author{Ian Tice}
\address{
Department of Mathematical Sciences\\
Carnegie Mellon University\\
Pittsburgh, PA 15213, USA
}
\email[I. Tice]{iantice@andrew.cmu.edu}
\thanks{I. Tice was supported by an NSF Grant (DMS \#2204912). }



\begin{document}

\begin{abstract}
In this paper we study traveling wave solutions to the free boundary incompressible Navier-Stokes system with generalized Navier-slip conditions. The fluid is assumed to occupy a horizontally infinite strip-like domain that is bounded below by a flat rigid surface and above by a moving surface. We assume that the fluid is acted upon by a bulk force and a surface stress that are stationary in a coordinate system moving parallel to the fluid bottom, and a uniform gravitational force that is perpendicular to the flat rigid surface. We construct our solutions via an implicit function argument, and show that as the slip parameter shrinks to zero, the Navier-slip solutions converge to solutions to the no-slip problem obtained previously. 
\end{abstract}

\maketitle

\section{Introduction}

The construction of traveling wave solutions to the inviscid, incompressible equations of fluid dynamics is a classical subject in mathematics with a rich history. In comparison, progress on the corresponding viscous problems began quite recently: the series of papers \cite{tice, leonitice, noahtice, noahtice3} developed a well-posedness theory for the free surface Navier-Stokes equations modeling incompressible fluids in a horizontally infinite strip-like domain of finite depth, subject to sources of applied force and stress.  In each of these papers, the fluid is assumed to obey the standard no-slip boundary condition at its lower boundary with a flat, rigid floor. The purpose of this paper is to continue the study of this type of problem by incorporating the more general Navier-slip condition, which allows the fluid to slip along the bottom boundary, and show that a generic well-posedness theory persists.  The slip boundary condition, first proposed by Navier \cite{Navier} in 1832, asserts that the tangential fluid velocity at the fluid bottom is proportional to the tangential stress experienced by the fluid.  The ratio of the tangential stress to the tangential fluid velocity is referred to as the slip parameter.  We will prove that not only are traveling wave solutions also generic under the Navier-slip conditions, but that  one recovers the no-slip solutions in the limit as the characteristic slip parameter goes to zero.

\subsection{Problem formulation}\label{sec: initial description}

In this paper we consider a single layer of viscous, incompressible fluid evolving in a horizontally infinite strip-like domain, bounded below by a flat, rigid surface and above by a free moving surface that can be described by the graph of a continuous function, in dimensions $n \ge 2$. Even though the only physically relevant cases are when $n=2,3$, the analysis in this paper can be applied more generally to higher dimensions as well.  Since our primary interest is the construction of traveling wave solutions, we will skip the somewhat lengthy discussion of the formulation of the fully dynamic problem and the subsequent reformulation under a traveling wave ansatz and jump straight to the traveling wave problem; these omitted details can be found in the introduction of \cite{tice}. The equations for a traveling wave solution to the free boundary Navier-Stokes system are
\begin{align}\label{eq: main unflattened}
		\begin{cases}
			\diverge S(q,v) - \gamma  e_1  \cdot \nabla v + v \cdot \nabla v  + \mathfrak{g} (\nabla' \eta, 0)= \mathfrak{f}, & \text{in} \; \Omega_{b + \eta} \\
			\diverge v = 0 , & \text{in} \; \Omega_{b + \eta} \\
			-\gamma  \p_1 \eta + \nabla' \eta \cdot v' = v_n  ,  & \text{on} \; \Sigma_{b + \eta} \\	
			S(q,v) \mathcal{N} =  (-\sigma \mathcal{H}(\eta) I  + \mathcal{T})\mathcal{N}   ,  & \text{on} \; \Sigma_{b + \eta}	\\
			-\alpha( S (q,v)  e_n)' = [A(v)]', & \text{on} \; \Sigma_0 \\
		v_n= 0, & \text{on} \; \Sigma_0.
	\end{cases}
\end{align}

We now explain all the terms appearing in the system \eqref{eq: main unflattened}.  The fluid occupies the unknown domain $\Omega_{b +\eta} = \{ x = (x',x_n) \in \R^{n}  \mid 0 < x_n < b +\eta(x') \}$, where $\eta : \R^{n-1}  \to (-b, \infty)$ is the unknown free surface function and $b >0$ is the equilibrium depth of the fluid.  The graph  $\Sigma_{b+\eta} =  \{x = (x',x_n) \in \R^{n}  \mid x_n = b+\eta(x')\}$ is the unknown upper boundary of the fluid, while the trivial graph $\Sigma_0 = \{ x = (x', x_n) \in \R^{n}  \mid x_n = 0 \}$ denotes its fixed, rigid lower boundary.  See Figure \ref{fig:fig_1} for a graphical depiction of the fluid domain.

\begin{figure}[!ht]
	\includegraphics[scale=0.5]{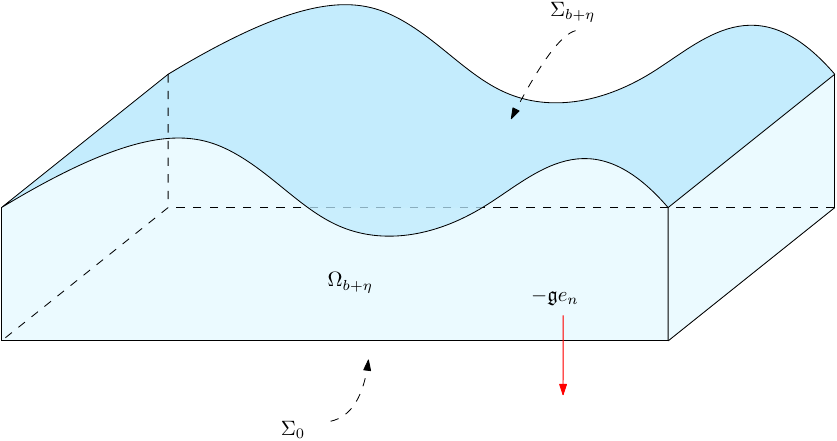}
	\caption{A sample portion of the unknown fluid domain in dimension $n=3$}
	\label{fig:fig_1}
\end{figure}

The fluid's velocity field and pressure are denoted here by $v : \Omega_{b+\eta} \to \R^n$ and $q : \Omega_{b+\eta} \to \R$, and together they determine the viscous stress tensor 
\begin{equation}
 S(q,v) = q I_{n \times n} - \mu \mathbb{D} v  = q I_{n \times n} - \mu (\nabla v + (\nabla v)^T) \in \R^{n \times n}
\end{equation}
with the viscosity coefficient $\mu >0$.  We emphasize, though, that the pressure $q$ is not really the fundamental fluid mechanical pressure, but rather a ``good'' pressure unknown obtained by subtracting off a variant of the hydrostatic pressure (see \cite{tice} for details). The parameter $\mathfrak{g} >0$ is the strength of the gravitational field, and the term $\mathfrak{g}(\nab' \eta,0)$ corresponds to the gravitational force the fluid experiences, after the aforementioned reformulation of the pressure unknown.  Without loss of generality, we henceforth assume the convenient normalization $\mu = \mathfrak{g} =1$.

The parameter $\gamma \in \R$ is the traveling wave speed, and its specific appearance in \eqref{eq: main unflattened} corresponds to solutions to the dynamic problem that are stationary in a coordinate frame moving with velocity $\gamma e_1$.  The applied bulk force $\mathfrak{f} : \Omega_{b+\eta} \to \R^n$ and the applied surface stress $\mathcal{T} : \Sigma_{b+\eta} \to  \R^{n\times n}_{\sym}$ are given data that are responsible for inducing the motion of the fluid.  The term $\mathcal{N} = (-\nabla' \eta, 1)$ denotes the non-unit normal vector field to $\Sigma_{b+\eta}$,  while the term $- \sigma \mathcal{H}(\zeta)$ corresponds to surface tension on $\Sigma_{b+\eta}$, with $\sigma >0$ denoting the coefficient of surface tension and $\mathcal{H}(\eta) = \diverge' ( \nabla' \eta / \sqrt{1+ \abs{\nabla' \eta}^2 })$ denoting the mean-curvature operator.   

The system \eqref{eq: main unflattened} is obtained from the incompressible Navier-Stokes system.  The first two equations in \eqref{eq: main unflattened} correspond to the balance of momentum and conservation of mass.  The third equation is the kinematic boundary condition describing the evolution of the free surface.  The fourth equation is called the dynamic boundary condition, as it encodes the balance of forces on the free surface.  The fifth and sixth equations constitute a general nonlinear version of the Navier-slip condition, which we now elaborate on.  The sixth equation is called the no-penetration condition, and it requires that the fluid is not able to detach from or pass through $\Sigma_0$.  Unlike in the case of the no-slip boundary condition, the fluid is allowed to have a nontrivial tangential component on $\Sigma_0$, which is described as ``slip.''   However, slip comes at a price: it generates a tangential stress on the fluid that opposes the motion, which one should think of as being analogous to the way that air resistance is modeled in standard Newtonian point-particle mechanics.  The precise form we impose in the fifth equation is (using the sixth)
\begin{equation}\label{eq:slip_expansion}
 [A(v)]' = -\alpha( S (q,v)  e_n)'  =  -\alpha (q e_n - \mathbb{D}v e_n)' = \alpha (\mathbb{D} v e_n)' =\alpha( \nabla' v_n + \p_n v') = \alpha \p_n v'
\end{equation}
for a given smooth ``slip parameter'' $\alpha >0$ and  ``slip function'' $A : \R^n \to \R^n$ satisfying (for technical reasons we will discuss later)
\begin{align}\label{eq: A monotone}
	A(0) = 0, \; A(w) \cdot w > 0 \text{ for } w \in \R^{n} \setminus \{0\}, \text{ and } A (w) \cdot w \ge \theta \abs{w}^2 \text{ for } w \in B(0,\delta) \setminus \{0\},
\end{align}
where $\delta, \theta > 0$ are fixed constants.  One should then think of \eqref{eq:slip_expansion} as a one-parameter family (indexed by $\alpha$) of nonlinear Robin boundary conditions with the extreme case $\alpha =0$ recovering the no-slip condition since then \eqref{eq:slip_expansion} requires $0 = A(v)'$ on $\Sigma_0$, which together with the condition $v_n=0$ and \eqref{eq: A monotone} implies that $v=0$ on $\Sigma_0$.  The most common form of the Navier-slip condition in the literature is in a linear form, in which $A$ is a linear map (often just the identity);  we have included the nonlinear form for the sake of generality, and our analysis certainly handles the standard linear case.

\subsection{Previous work}

The Navier-slip condition was first proposed by Navier \cite{Navier} and it is now used to model a wide range of physical phenomena, including liquid-solid contact lines (we refer to Dussan's survey \cite{dussan}) and flows through irregular surfaces (see, for instance, the work of G\'erard-Varet and Masmoudi \cite{masmoudi}).  It also plays a crucial role in the analysis of collisions in fluid-solid systems: see, for example, \cite{feireisl,gv_hillairet_2010, gv_hillairet_2014,gv_hillairet_wang,hillairet,hillairet_takahashi,starovoitov}. The slip phenomenon has also been empirically observed in recent experiments; we refer to the survey of Neto et al. \cite{neto} and the references therein for a review of these results. 

The well-posedness of the Navier-Stokes system with Navier-slip boundary conditions has been investigated by several authors. Solonnikov-\v{S}\v{c}adilov \cite{solonnikovslip} studied the 3D linearized stationary Navier-Stokes system and proved the existence of weak solutions as well as their regularity. Beir\~{a}o da Veiga \cite{viega} studied the stationary problem on the half space and proved strong regularity up to the boundary. Ferreira \cite{Ferreira} studied the inhomogeneous system on bounded space-time domains and proved the existence of weak solutions. Masmoudi-Rousset \cite{rousset} proved  uniform in time bounds with respect to the viscosity parameter.   Kelliher \cite{kelliher} studied the 2D  equations on bounded domains  and proved that 2D Navier-slip solutions with sufficiently smooth initial velocities converge to the no-slip solutions as the slip parameter goes to zero. Murata-Shibata \cite{murata} studied the compressible variant with slip boundary conditions on bounded domains and proved a global in time unique existence theorem for small data. 

The dynamical stability of Navier-slip solutions has also been studied by various authors. Li-Pan-Zhang \cite{li} studied the stability of steady state solutions to the 3D incompressible problem on bounded domains.  Ding-Lin \cite{ding} studied the stability of the Couette flow in 2D, and Li-Zhang \cite{li-zhang} studied the stability of Couette flow in 3D, and separately they proved that the Couette flow is asymptotically stable under small perturbations with various conditions on the slip parameter and viscosity.

The well-posedness of the traveling wave problem for the free boundary Navier-Stokes system first appeared in the recent work of Leoni-Tice \cite{leonitice}.  This work was extended to periodic and tilted fluid configurations by Koganemaru-Tice \cite{tice}. Stevenson-Tice \cite{noahtice} studied multi-later configurations \cite{noahtice}, the vanishing wave speed limit \cite{noahtice3}, and the compressible  traveling wave problem \cite{noahtice2}.  Similar well-posedness result for the traveling wave formulation of the Muskat problem were obtained by Nguyen-Tice \cite{nguyentice}.

\subsection{Reformulation in a fixed domain}\label{sec:flatten}

The fluid domain $\Omega_{b+\eta}$ is one of the unknowns in  \eqref{eq: main unflattened}, so it is convenient to recast the system in a fixed domain.  We choose the equilibrium domain $\Omega := \Omega_b = \R^{n-1} \times (0,b)$ for this, and write $\Sigma_b = \{x \in \R^n \mid x_n =b\}$ for the flat upper boundary.  
The reformulation is achieved by introducing  the flattening map $\mathfrak{F}: \R^n \to \R^n$, associated to any $\eta \in C^1_b(\R^{n-1}; \R)$ satisfying $\eta > -b$, defined via 
\begin{align}\label{eq:flattening}
	\mathfrak{F}(x',x_n) = (x' , x_n + \eta(x') \varphi(x_n)),
\end{align}
where $\varphi \in C_b^\infty( \R ; \R)$ is some fixed function  that it is a monotone and satisfies $\varphi = 0$ on $(-\infty,b/4]$ and $\varphi = 1$ on $[3b/4,\infty)$.  By construction, we have that $\mathfrak{F} (\overline{\Omega_b}) =\overline{\Omega_{b+\eta}}$,  $\mathfrak{F}(\Sigma_b) = \Sigma_{b+\eta}$, and $\mathfrak{F}  = I$ in $\R^{n-1} \times (-\infty,b/4)$, which in particular means that $\mathfrak{F}$ is the identity on $\Sigma_0$.  Moreover, it's easy to see that if $\norm{\eta}_{C^0_b}$ is sufficiently small then $\mathfrak{F}$ is a diffeomorphism.  

We compute 
\begin{align}\label{eq:A}
\nabla \mathfrak{F}(x) = \begin{pmatrix}
I_{(n-1) \times (n-1)} & 0_{ (n-1) \times 1 } \\
\nabla' \eta(x')^T \varphi(x_n) & 1 + \eta(x') \varphi'(x_n)
\end{pmatrix}
\text{ and } 
(\nabla \mathfrak{F})^{-\intercal}(x) =  \begin{pmatrix} 
	I_{(n-1) \times (n-1) } & -\frac{ \nabla' \eta (x')\varphi(x_n) }{1 + \eta(x') \varphi'(x_n)} \\
	0_{1 \times (n-1)} &  \frac{1}{1 + \eta(x') \varphi'(x_n)}.
\end{pmatrix}.
\end{align}
We then define $\mathcal{A}: \R^n \to \R^{n\times n}$ via  $\mathcal{A}(x) = (\nabla \mathfrak{F})^{-\intercal}(x)$ and $\mathcal{J},\mathcal{K} : \R^n \to \R$ via
\begin{align}\label{eq:JandK}
\mathcal{J}(x) = \det \nabla \mathfrak{F}(x) = 1+ \eta(x') \varphi'(x_n)
\text{ and }
\mathcal{K}(x) = \frac{1}{J(x)} =\frac{1}{ 1+ \eta(x') \varphi'(x_n)}.
\end{align}
Then we define the $\mathcal{A}$-dependent differential operators: $(\nabla_{\mathcal{A}} f)_i = \sum_{j=1}^n \mathcal{A}_{ij} \p_j f,$ $(X \cdot \nabla_\mathcal{A} u)_i = \sum_{j,k=1}^n X_j \mathcal{A}_{jk} \p_k u_i$, $\diverge_{\mathcal{A}} X = \sum_{i,j = 1}^n \mathcal{A}_{ij} \p_j X_i$, $(\mathbb{D}_{\mathcal{A}} u)_{ij} = \sum_{k=1}^n \mathcal{A}_{ik} \p_k u_j + \mathcal{A}_{jk} \p_k u_i$, $S _{\mathcal{A}}(p,u) = pI - \mu \mathbb{D}_\mathcal{A} u$, $\diverge_{\mathcal{A}} S_{\mathcal{A}}(p,u) = \nabla_{\mathcal{A}} p - \mu \Delta_{\mathcal{A}} u - \mu \nabla_{\mathcal{A}} \diverge_{\mathcal{A}} u$, $(\Delta_\mathcal{A} u)_i = \sum_{j,k,m=1}^n \mathcal{A}_{j,k} \p_k (\mathcal{A}_{jm} \p_m u_i)$. 

Next, we introduce the new unknowns $u: \Omega \to \R^n$, $p: \Omega \to \R,$ and  $f :\Omega \to \R^n$ via $u = v \circ \mathfrak{F}$, $p = q \circ \mathfrak{F}$, and $f = \mathfrak{f} \circ \mathfrak{F}$.  This yields the reformulation of \eqref{eq: main unflattened}:
\begin{align} \label{eq: main flattened}
	\begin{cases}	
		\diverge_{\mathcal{A}} S_{\mathcal{A}}(p,u) - \gamma e_1 \cdot \nabla_{\mathcal{A}} u + u \cdot \nabla_{\mathcal{A}} u + (\nabla'\eta,0)  = \mathfrak{f} \circ \mathfrak{F}
		, &  \text{in} \; \Omega \\
	\diverge_\mathcal{A} u  = 0, & \text{in} \; \Omega \\
	-\gamma \p_1 \eta - u\cdot \mathcal{N} = 0, & \text{on} \; \Sigma_b \\
	S_\mathcal{A}(p, u) \mathcal{N} = \left[ - \sigma \mathcal{H}(\eta) I + \mathcal{T}\circ \mathfrak{F} \right] \mathcal{N}, & \text{on} \; \Sigma_b \\
		\alpha[S_\mathcal{A} (p,u) \nu]' = [A(u)]', & \text{on} \; \Sigma_0 \\
	u_n = 0, & \text{on} \; \Sigma_0.
	\end{cases}
\end{align}

\subsection{Main results and discussion}

In this subsection we state the main results obtained in this paper.  The first result establishes the existence and uniqueness of solutions to the flattened problem \eqref{eq: main flattened};   the spaces $C^k_b, C^k_0$ appearing in the statement are defined in Section~\ref{sec: notation} and the space $\mathcal{X}^s$ is defined in Definition~\ref{defn:Xsforsol}.

\begin{thm}[Proved later in Section~\ref{sec: main flattened}]\label{thm:main1}
	Suppose  $\N \ni s \ge 1 + \tfloor{n/2}$ and that either $\sigma > 0$ and $n \ge 2$ or else $\sigma = 0$ and $n = 2$.   Further suppose that $A : \R^n \to \R^n$ is smooth and obeys \eqref{eq: A monotone}. Then there exists open sets 
\begin{align} \
		\mathcal{U}^s \subset \R^+ \times (\R \setminus \{0\}) \times H^{s+3}(\R^{n} ; \R^{n \times n}_{\sym}) \times H^{s+\frac{1}{2}}(\R^{n-1} ; \R^{n\times n}_{\sym}) \times H^{s+2}(\R^{n}; \R^n) \times H^s(\R^{n-1} ; \R^n) 
\end{align}
and $\mathcal{O}^s \subset \mathcal{X}^s$ such that the following hold.
\begin{enumerate}
	\item $(0,0,0) \in \mathcal{O}^s$, and for every $(u,p,\eta) \in \mathcal{O}^s$ we have that $u \in C_0^{s + 1 - \lf n/2 \rf} (\Omega ; \R^n)$, $p \in C^{s - \lf n/2 \rf}_0 (\Omega; \R)$,  $\eta \in C^{s + 1 - \lf (n-1)/2 \rf}_0 (\R^{n-1}; \R)$ with $\max_{\R^{n-1}} \abs{\eta} \le \frac{b}{2}$, and the flattening map $\mathfrak{F}$ is a  $C^{s + 1  - \lf (n-1)/2 \rf}$ diffeomorphism. 
	\item We have $\R^+ \times (\R \setminus \{0\}) \times \{ 0\} \times \{ 0 \} \times \{ 0 \} \times \{0 \} \subset \mathcal{U}^s$.
	\item For each $(\alpha, \gamma, \mathcal{T}, T, \mathfrak{f}, f) \in \mathcal{U}^s$, there exists a unique $(u, p, \eta) \in \mathcal{O}^s$ classically solving
\begin{align} \label{eq: thm 1 main flattened}
	\begin{cases}
		\diverge_{\mathcal{A}} S_{\mathcal{A}}(q,v) - \gamma  e_1  \cdot \nabla_{\mathcal{A}} v + v \cdot \nabla_{\mathcal{A}} v  + (\nabla' \eta, 0)= \mathfrak{f} \circ \mathfrak{F} + L_{\Omega_b} f, & \text{in} \; \Omega \\
		\diverge_{\mathcal{A}} v = 0 , & \text{in} \; \Omega \\
		-\gamma  \p_1 \eta + \nabla' \eta \cdot v' = v_n  ,  & \text{on} \; \Sigma_{b} \\	
		S_{\mathcal{A}}(q,v) \mathcal{N} =  (-\sigma \mathcal{H}(\eta) I  + \mathcal{T}\circ \mathfrak{F} \vert_{\Sigma_b} + S_b T )\mathcal{N}   ,  & \text{on} \; \Sigma_{b}	\\
		\alpha[S_{\mathcal{A}} (q,v)  \nu]' = [A(v)]', & \text{on} \; \Sigma_0 \\
	v_n= 0, & \text{on} \; \Sigma_0,
\end{cases}
\end{align}	
	where $L_{\Omega_b}, S_b$ are defined via \eqref{eq: L_Omega and S_b}.
	\item The map $\mathcal{U}^s \ni (\alpha, \gamma, \mathcal{T}, T, \mathfrak{f}, f) \mapsto (u,p,\eta) \in \mathcal{O}^s$ is $C^1$ and locally Lipschitz. 
\end{enumerate}
\end{thm}

The theorem is proved by way of the implicit function theorem by adapting the strategies employed for the corresponding no-slip problem; we refer to Section 1.5 of \cite{leonitice} and Section 1.7 in \cite{tice} for a high-level summary of this plan.  We emphasize, though, that while there is a serious overlap in the strategies, there are interesting technical problems introduced by the Navier-slip condition that must be dealt with along the way.  We further note that by following the approach in \cite{leonitice}, solutions to the unflattened system \eqref{eq: main unflattened} may be obtained from this theorem by employing the inverse of the flattening map, $\mathfrak{F}^{-1}$.  We omit the details here for the sake of brevity.  

Our second result, which is the principal novelty of this paper, establishes that if the slip map $A$ is linear then  in the limit $\alpha \to 0$ we can recover the no-slip solution to the incompressible Navier-stokes system obtained in \cite{tice,leonitice}.

\begin{thm}[Proved later in Section~\ref{sec: alpha to zero}]\label{thm: alpha to zero}
Suppose that  $\N \ni s \ge 1 + \tfloor{n/2}$ and that either $\sigma > 0$ and $n \ge 2$ or else $\sigma = 0$ and $n=2$. Further suppose that $A(\cdot) = \beta \cdot$ where $\beta \in \R^{n \times n}$ is positive definite. Then there exist open sets 
\begin{align} \label{eq: set U^s}
		\mathcal{U}^s \subset \R^+ \times (\R \setminus \{0\}) \times H^{s+3}(\R^{n} ; \R^{n \times n}_{\sym}) \times H^{s+\frac{1}{2}}(\R^{n-1} ; \R^{n\times n}_{\sym}) \times H^{s+2}(\R^{n}; \R^n) \times H^s(\R^{n-1} ; \R^n) 
\end{align}
and $\mathcal{O}^s \subset \mathcal{X}^s$ such that for each $\gamma_* \in \R \setminus \{0\}$, there exists an open set $V(\gamma_*)$ such that for all $\alpha_* \in (0,1)$ the following hold. 
\begin{enumerate}
	 \item The open sets $\mathcal{U}^s, \mathcal{O}^s$ satisfy the first two items of Theorem~\ref{thm:main1}.
	 \item $(\alpha_*, \gamma_*, 0 ,0,0, 0) \in (0,1) \times V(\gamma_*) \subset \mathcal{U}^s$.
	 \item For every $(\mathcal{T},  T, \mathfrak{f}, f)$ such that $(\gamma_*, \mathcal{T}, T, \mathfrak{f}, f) \in V(\gamma_*)$, there exists a unique $(u_{\alpha_*}, p_{\alpha_*}, \eta_{\alpha_*}) \in \mathcal{O}^s$ classically solving \eqref{eq: thm 1 main flattened}. Furthermore, $(u_{\alpha_*}, p_{\alpha_*}, \eta_{\alpha_*})$ converges weakly to $(u_0, p_0, \eta_0)$ in $H^{s+2}(\Omega; \R^n) \times H^{s+1}(\Omega; \R) \times X^{s+5/2}(\R^{n-1} ; \R)$ as $\alpha_* \to 0$, where 
	 \begin{equation*}
	 (u_0, p_0, \eta_0) \in C_b^{s + 1 - \lf n/2 \rf} (\Omega ; \R^n) \times \in C^{s - \lf n/2 \rf}_b (\Omega; \R) \times C^{s + 1 - \lf (n-1)/2 \rf}_0 (\R^{n-1}; \R)
	 \end{equation*}
 is the unique solution to 
	 \begin{align}\label{eq: noslip}
		 \begin{cases}
			 \diverge_{\mathcal{A}} S_{\mathcal{A}}(q,v) - \gamma  e_1  \cdot \nabla_{\mathcal{A}} v + v \cdot \nabla_{\mathcal{A}} v  + (\nabla' \eta, 0)= \mathfrak{f} \circ \mathfrak{F} + L_{\Omega_b} f, & \text{in} \; \Omega_{b} \\
			 \diverge_{\mathcal{A}} v = 0 , & \text{in} \; \Omega_{b} \\
			 -\gamma  \p_1 \eta + \nabla' \eta \cdot v' = v_n  ,  & \text{on} \; \Sigma_{b} \\	
			 S_{\mathcal{A}}(q,v) \mathcal{N} =  (-\sigma \mathcal{H}(\eta) I  + \mathcal{T}\circ \mathfrak{F} \vert_{\Sigma_b} + S_b T )\mathcal{N}   ,  & \text{on} \; \Sigma_{b}	\\
		 v= 0, & \text{on} \; \Sigma_0.
	 \end{cases}
	 \end{align}
	
\end{enumerate}
\end{thm}

We now turn to a brief discussion of our strategy for proving this theorem.  There are essentially two key difficulties that must be dealt with.  The first comes from the fact that we want to fix the stress-force tuple $(\mathcal{T},  T, \mathfrak{f}, f)$ and produce a family of solutions $(u_\alpha,p_\alpha,\eta_\alpha)$ to \eqref{eq: thm 1 main flattened}, parameterized by $\alpha \in (0,1)$.  This is certainly plausible within the context of Theorem \ref{thm:main1}, but there is nothing within the statement of that result that can guarantee that the tuple remains within the open set of data that yields solutions.  Indeed, in principle the open set could shrink dramatically as $\alpha \to 0$, making it impossible to employ a fixed data tuple in the limiting argument.  Provided this problem can be dispatched, we then arrive at the second: we need to establish $\alpha$-independent estimates for the solutions $(u_\alpha,p_\alpha,\eta_\alpha)$ in order to invoke weak compactness results.  
 
We resolve both of these problems by combining a careful analysis of the linearization of  \eqref{eq: thm 1 main flattened} with some nonlinear tricks.  In the linear analysis we achieve $\alpha$-independent estimates by focusing on the linearization \eqref{eq: linear w/o eta} with $l=0$.  This is only reasonable insofar as we can encode $l=0$ in the nonlinear problem, which means that the fifth equation in \eqref{eq: thm 1 main flattened} must already be linear.  To enforce this we require that $A$ itself is linear and that the matrix $\mathcal{A}$ is the identity in a neighborhood of $\Sigma_0$.  The latter condition is the motivation for the introduction of the cutoff $\varphi$ in the definition of the flattening map $\mathfrak{F}$; unfortunately, its presence here requires us to retool many previously established results.  

In order to enforce the linear slip condition in an implicit function theorem argument we then need to build this condition into the domain of the nonlinear map.  For any fixed value of $\alpha$ this is easy, but we need to do this for $\alpha \in (0,1)$, which means the linear subspace obeying the $\alpha$-slip condition changes as a function of $\alpha$.  This then requires us to develop a special version of the implicit function theorem capable of handling maps $f_\alpha: X \times Y_\alpha \to Z$  defined over a one-parameter family of Banach spaces. We prove this variant of the implicit function theorem in Appendix~\ref{sec: parameter IFT} and demonstrate that with uniform control over the derivatives of the nonlinear solution operator with respect to the parameter $\alpha$, we may also deduce uniform control over the norms of solutions obtained via the parameter-dependent implicit function theorem.  This is a stronger mandate than that from the standard implicit function theorem, so we must then verify these conditions in our linear analysis.  This turns out to be doable but somewhat tricky because it requires determining the asymptotics of an implicit Fourier multiplier as a function of $\alpha$.

\subsection{Notational conventions and outline of the article}\label{sec: notation}

We will frequently use the ``horizontal'' Fourier transform for functions on  $\Omega = \R^{n-1} \times (0,b)$, defined by $\hat{f}(\xi, x_n) = \int_{\R^{n-1}} f(x', x_n) e^{-2\pi i x' \cdot \xi} \; dx'$. For $k \in \N$, a non-empty open set $U \subseteq \R^d$, and a finite-dimensional inner product space $V$, we write $H^k(U;V)$ for the usual  $L^2$-based Sobolev space; when $U = \R^d$ we extend to $H^s(\R^d;V)$ for $s \in \R$ in the usual way.  For $s \ge 0$ and a function $f \in L^2(\R^d;V)$, we write $f \in \dot{H}^{-1}$ to mean that the $\dot{H}^{-1}$-seminorm of $f$ defined via $[f]_{\dot{H}^{-1}}^2 = \int_{\R^d} \abs{\xi}^{-2}|\hat{f}(\xi)|^2  d\xi$ is finite. For $k \in \N$, a real Banach space $V$, and a nonempty open set $U \subseteq \R^{d}$ for $d \ge 1$, we define the space $C^k_b(U;V)$ of $k-$times continuous differentiable maps from $U \to V$ with all derivatives bounded. We also define the space $C_0^k(\R^d;V) \subset C^k_b(\R^d;V)$ to be the closed subspace of $f$ such that $\lim_{|x| \to \infty} \p^\alpha f(x) = 0$ for all $\abs{\alpha} \le k$.

\section{The $\gamma$-Stokes equations with stress boundary conditions}\label{sec: gammaStokes}

In this section our goal is to study the solvability of the linear problem 
\begin{align} \label{eq: linear w/o eta}
	\begin{cases}
		\diverge S(p,u) - \gamma \p_1 u = f, & \text{in} \; \Omega \\
		\diverge u = g, & \text{in} \; \Omega \\
		S(p,u) e_n = k, & \text{on} \; \Sigma_b \\
		\left[ \alpha S (p,u) e_n +  \beta u \right]' = l, & \text{on} \; \Sigma_0 \\
		u_n = 0, & \text{on} \; \Sigma_0
	\end{cases}
\end{align}
with a given data tuple $(f,g,k,l)$ and parameters $\alpha \in (0,\infty)$, $\gamma \in \R$, and $\beta \in \R^{n \times n}$. Due to techniques we will employ later, it will be convenient to have access to a well-posedness theory over both the reals and the complex numbers.  As such, throughout this section and the next we let $\F \in \{ \R , \C\}$ denote either field, and we develop a well-posedness theory generically over $\F$.  Recall that when $\F = \C$ and $X$ is a complex Hilbert space, the Riesz map is a linear isomorphism from $X$ to $X^{\overline{*}}$, where the latter denotes the anti-linear functionals on $X$.  We will use this notation a few times throughout this section.

We begin our analysis by fixing some notation.  

\begin{defn}\label{defn:OPB}
Let $\Omega$ be defined as per Section~\ref{sec: initial description}.  For $\R \ni s > 1/2$, we define the spaces ${}_{\text{tan}}H^s(\Omega ; \F)= \{ u \in H^s(\Omega ;\F) : u_n \rvert_{\Sigma_0}= 0  \}$ and     ${}_{\text{tan}}H_\sigma^1(\Omega ; \F)= \{ u \in {}_{\text{tan}}H^s(\Omega;\F) : \diverge u = 0  \}$.
%
%
		We equip these spaces with the standard $H^s$-norm, and note that since these spaces are closed subspaces of the Hilbert space $H^s(\Omega ;\F)$, they inherit the natural Hilbert structure.  If in addition $\R \ni s > 3/2$ and $\alpha \in \R$, we define the space ${}_{\alpha-\text{tan}}H^s(\Omega ; \F)= \{ u \in H^s(\Omega ;\F) : u_n \rvert_{\Sigma_0}= 0, [-\alpha \mathbb{D}u e_n + \beta u]' \rvert_{\Sigma_0} = 0  \}$,
		which is a closed subspace of $ {}_{\text{tan}}H^s(\Omega ; \F)$ and thus inherits the natural Hilbert structure from $H^s(\Omega ;\F)$ as well. 
	\end{defn}

In order to produce weak solutions to the system \eqref{eq: linear w/o eta} we will first need some functional analytic tools in ${}_{\text{tan}}H^s(\Omega ; \F)$.  We begin with a version of Korn's inequality.

\begin{lem}\label{lem: korn}
We have that $\norm{u}_{H^1(\Omega)} \lesssim \norm{\mathbb{D}u}_{L^2(\Omega)} + \norm{\Tr_{\Sigma_0} u}_{L^2(\Sigma_0)}$  for $u \in H^1(\Omega ; \F^n)$.  Consequently, 
\begin{align}\label{eq: equivalent norm on H1}
	 \norm{u}_{\Hzerotan(\Omega)} = \sqrt{ \norm{\mathbb{D}u}_{L^2(\Omega)}^2 + \norm{\Tr_{\Sigma_0} u}_{L^2(\Omega)}^2 },
\end{align}
generates the standard $H^1$ topology on the space  ${}_{\text{tan}}H^s(\Omega ; \F)$. 

\end{lem}
\begin{proof}
The second assertion follows from the first bound and standard trace theory. To prove the first, it suffices to prove the result when $\F = \R$, as the case $\F=\C$ can then be recovered by applying the real result to the real and imaginary parts of $u$.  Assume $\F = \R$.  

Consider a rectangle  $Q = \{ x' \in \R^{n-1} : \abs{x' }_\infty < 1 \} \times (0,b)$.  The standard Korn inequality in Lipschitz domains (see, for instance, Lemma IV.7.6 in  \cite{boyer}) shows that $\norm{u}_{H^1(Q)} \lesssim \norm{u}_{L^2(Q)} + \norm{\mathbb{D}u}_{L^2(Q)}$.  We claim that 
\begin{equation}\label{eq: reduced korn}
\norm{u}_{L^2(Q)} \lesssim  \norm{\mathbb{D}u}_{L^2(Q)} + \norm{\Tr_{\p Q_0} u}_{L^2(\p Q_0)}, 
\end{equation}
where $\p Q_0 = \p Q \cap \{x_n =0\}$.  Indeed, if not then we can produce a sequence $\{ u_k \}_{k=1}^\infty \subset H^1(Q ; \R^n)$ such that $\norm{u_k}_{L^2(Q)} =1$, $\norm{\mathbb{D} u_k}_{L^2(Q)} < 1/k$,  and $\norm{\Tr_{\p Q_0} u_k}_{L^2(\p Q_0)} < 1/k$. Then, by compactness, there exists $u \in \Hzerotan(Q ; \R^n)$ with $\norm{u}_{L^2(Q)} = 1$ such that up to passing to a subsequence, $\mathbb{D} u_k \to \mathbb{D} u = 0$ and $\Tr_{\p Q_0}  u_k \to \Tr_{\p Q_0} u = 0$ as $k \to \infty$.   Since $\mathbb{D}u = 0$, we then have that $u(x) = a + Bx$ for a constant $ a \in \R^n$ and $B$ skew-symmetric, but since $\Tr_{\p Q_0} u = 0$ we deduce that $a = 0$ and $B = 0$.  Thus $u = 0$, and we contradict the identity $\norm{u}_{L^2(Q)}=1$, proving the claim.

With \eqref{eq: reduced korn} in hand, we write $\Omega$ as a countable almost disjoint union (null intersections along the boundary) of rectangles of the form  $Q_{l} = \{ x \in \Omega : \abs{x' - l}_\infty \le 1 \}$ for $l \in \Z^{n-1}$.  Since each $Q_\ell$ is a translation of the rectangle $Q$ from above, \eqref{eq: reduced korn} allows us to bound 
\begin{equation}
\norm{u}_{L^2(\Omega)}^2 = \sum_{\ell \in \Z^{n-1}} \norm{u}_{L^2(Q_l)}^2 \lesssim  \sum_{\ell \in \Z^{n-1}}\left( \norm{\mathbb{D}u}_{L^2(Q_l)}^2 + \norm{\Tr_{ \Sigma_0} u}_{L^2((\p Q_\ell)_0)}^2 \right) = \norm{\mathbb{D}u}_{L^2(\Omega)}^2 + \norm{\Tr_{\Sigma_0} u}_{L^2(\Sigma_0)}^2.
\end{equation}
This is the desired bound.
\end{proof}

The next result provides a right inverse to the divergence operator.

\begin{lem}\label{lem:divinverse}
	There exists a linear and continuous mapping $\Pi : L^2(\Omega ; \F) \to {}_0H^{1}(\Omega; \F^n)$ such that $\diverge \Pi g = g$ for all $g \in L^2(\Omega; \F)$. In particular, for all $g \in L^2(\Omega ; \F)$ we have $\norm{\Pi g}_{\Hzerotan(\Omega)} \lesssim_{n,b} \norm{g}_{L^2(\Omega)}$.
\end{lem}
\begin{proof}
This follows from Lemma 2.1 in \cite{leonitice} and Lemma 2.2 in \cite{noahtice}.
\end{proof}

We next prove a Helmholtz decomposition of $\Hzerotan(\Omega;\F^n)$.

\begin{lem}\label{lem:orthodecomp}
	Define the bounded linear operator $Q: L^2(\Omega; \F) \to \Hzerotan(\Omega, \F^n)$ via  
		 \begin{align} \label{eq:G}
			 \int_\Omega p \diverge \overline{v} = \left( Qp, v \right)_{\Hzerotan(\Omega;\F)} \; \text{for all} \; v \in \Hzerotan(\Omega; \F^n).
		 \end{align}
	Then $Q$ has closed range, and $(\Range Q)^\perp = \Hzerostanigtan(\Omega; \F^n)$. Consequently, we have the orthogonal decomposition
		\begin{align} \label{eq: helmholtz}
			\Hzerotan(\Omega; \F^n) = \Hzerostanigtan(\Omega; \F^n) \oplus_{\Hzerotan} \Range Q.
		\end{align}
\end{lem}
\begin{proof}
	We first show that $Q$ has closed range. To do so, we first note that for all $p \in L^2(\Omega; \F)$ we have the bound $\norm{Qp}_{\Hzerotan(\Omega)} \lesssim_{n,p} \norm{p}_{L^2(\Omega)}$. 	On the other hand, by Lemma~\ref{lem:divinverse} there exists a $v_0 \in \Hzerotan(\Omega; \F^n)$ such that $\diverge \overline{v_0} = p$ and $\norm{v_0}_{\Hzerotan} \lesssim_{n,b} \norm{p}_{L^2}$. Therefore by the Cauchy-Schwartz inequality, 
	\begin{align} 
		\norm{p}_{L^2}^2 = \int_\Omega p \diverge\overline{v_0}= \left( Qp, v_0 \right)_{\Hzerotan(\Omega)} \le \norm{Qp}_{\Hzerotan(\Omega)} \norm{v_0}_{\Hzerotan(\Omega)} \lesssim_{n,b} \norm{Qp}_{\Hzerotan(\Omega)} \norm{p}_{L^2(\Omega)}.
	\end{align}
	This implies that $\norm{p} \lesssim_{n,b} \norm{Qp}_{\Hzerotan}$, thus we have the equivalence of norms $\norm{Qp}_{\Hzerotan} \asymp \norm{p}_{L^2}$ for all $p \in L^2(\Omega; \F)$. This immediately implies that $Q$ has closed range, and so $\Hzerotan(\Omega; \F^n) = \Range Q \oplus_{\Hzerotan} \left( \Range Q \right)^{\perp}$. It remains to show that $\left( \Range Q \right)^\perp = \Hzerostanigtan(\Omega; \F^n)$. If $v \in \left( \Range Q \right)^\perp$, then  $\left( Qp , v \right)_{\Hzerotan} = (p, \diverge v)_{L^2} = 0$ for all $p \in L^2(\Omega; \F)$. Thus we must have $\diverge \overline{v} = 0$ $\mathcal{L}^n$-a.e., which implies that $v \in \Hzerostanigtan(\Omega; \F^n)$. If $v \in \Hzerostanigtan(\Omega; \F^n)$, then $\left( Qp, v \right) = \int_\Omega p \diverge \overline{v} = 0$ for any $p \in L^2(\Omega; \F)$, which implies that $v \in \left( \Range Q \right)^\perp$. This shows that $(\Range Q)^\perp = \Hzerostanigtan(\Omega, \F^n)$ as desired. Since the range of $Q$ is closed, the Helmholtz decomposition \eqref{eq: helmholtz} follows. 
\end{proof}

This gives us an immediate corollary.

\begin{cor}\label{cor:pressure}
 Let $\Lambda_1 \in \left( \Hzerotan(\Omega;\F^n) \right)^{\overline{*}}$ be such that  $\langle \Lambda_1, v \rangle = 0$ for all $v \in \Hzerostanigtan(\Omega; \F^n)$. Then there exists unique $p \in L^2(\Omega;\F)$ such that $\langle \Lambda_1, v \rangle = \int_\Omega p \diverge \overline{v}$ for all $v \in \Hzerotan(\Omega; \F^n)$.  Moreover, we have the estimate $\norm{p}_{L^2} \lesssim_{n,b} \norm{\Lambda_1}_{\left( \Hzerotan \right)^{\overline{*}} }$.
\end{cor}

\begin{proof}
First we suppose that $\F = \R$ and let $\Lambda \in \left( \Hzerotan(\Omega;\R^n) \right)^{*}$ be such that it vanishes on solenoidal fields. By the Riesz representation theorem, there exists $w \in \Hzerotan(\Omega; \R^n)$ such that  $\langle \Lambda, v \rangle = \left(  w,v\right)_{\Hzerotan}$ for all$v \in \Hzerotan(\Omega; \R^n)$ and $\norm{w}_{\Hzerotan} = \norm{\Lambda}_{\left( \Hzerotan \right)^{\overline{*}}}$. Then for all $v \in \Hzerostanigtan(\Omega; \R^n)$, we have $(w, v)_{\Hzerotan} = \langle \Lambda, v \rangle = 0$, thus $w \in \left( \Hzerostanigtan(\Omega;\R^n) \right)^\perp$. By Lemma~\ref{lem:orthodecomp}, we have $w \in \Range Q$, therefore there exists a $p \in L^2(\Omega; \R)$ such that $Qp = w$. So we have $\langle \Lambda, v \rangle = \left( Qp,v \right)_{\Hzerotan} = \int_\Omega p \diverge \overline{v}$ for all $v \in \Hzerotan(\Omega; \R^n)$, with the estimate 
	\begin{align} 
		\norm{p}_{L^2} \lesssim_{n,b} \norm{Qp}_{\Hzerotan} = \norm{w}_{\Hzerotan} = \norm{\Lambda}_{\left( \Hzerotan \right)^{\overline{*}}}.
	\end{align}
	Moreover, $p \in L^2(\Omega; \R)$ is unique since $Q$ is surjective.

	Now we consider the case when $\F = \C$. If we have an antilinear functional $\Lambda \in \left( \Hzerotan(\Omega;\C^n) \right)^{\overline{*}}$ vanishing on solenoidal fields, we can define the $\R$-linear functionals $\Lambda_{\Re}, \Lambda_{\Im} \in \left( \Hzerotan(\Omega; \R^n) \right)^{\overline{*}}$ via $\langle \Lambda_{\Re}, v \rangle = \Re \langle F, v \rangle$ and $ \langle \Lambda_{\Im}, v \rangle = \Re \langle F, iv \rangle$ 	for any $v \in \Hzerotan(\Omega; \R^n)$. Note that if $v  \in \Hzerostanigtan(\Omega; \R^n)$, then $\langle \Lambda, v \rangle = 0$ by assumption, so it follows that $\Lambda_{\Re}, \Lambda_{\Im}$ vanishes on real-valued solenoidal fields. Thus when $\F = \R$, there exist unique $q,r \in L^2(\Omega; \R)$ such that for all $v,w \in \Hzerotan(\Omega; \R^n)$, 
	\begin{align} 
		\Re [ \langle \Lambda, v + i w \rangle ] &= \langle \Lambda_{\Re}, v \rangle + \langle \Lambda_{\Im}, w \rangle = \int_\Omega q \diverge v + r \diverge w = \Re \left[ \int_\Omega \left( q+ir \right) \overline{\diverge \left( v + i w \right) }\right]. 	\end{align}
	Now define $p \in L^2(\Omega; \C)$ via $p = q + ir$, and for any $u \in \Hzerotan(\Omega; \C^n)$ we write it as $u = v+ iw$. Then 
	\begin{multline} 
		\langle \Lambda, u \rangle = \Re [\langle \Lambda, v + i w \rangle] + i \Im [ \langle \Lambda, v + i w \rangle ] 
					    = \Re [\langle \Lambda, v + i w \rangle] + i \Re [-i \langle \Lambda, v + iw \rangle] 
						\\ = \Re [\langle \Lambda, v + i w \rangle] + i \Re [ \langle \Lambda, -w + iv  \rangle] 
					   = \Re \left[ \int_\Omega \left( q+ir \right) \overline{\diverge \left( v + i w \right) }\right] 	+ i \Re \left[ \int_\Omega \left( q+ir \right) \overline{\diverge \left( -w+iv \right) }\right] \\
					   = \Re \left[ \int_\Omega \left( q+ir \right) \overline{\diverge \left( v + i w \right) }\right] 	+ i \Im \left[ \int_\Omega \left( q+ir \right) \overline{\diverge \left( v + iw \right) }\right] 
					   = \int_{\Omega} p \diverge \overline{u}.
	\end{multline}
\end{proof}

With these preliminary results in hand, we now turn to the question of  weak solvability of \eqref{eq: linear w/o eta}.  We first set some notation.

\begin{defn}
Let $\R \ni s \ge 0$, $\R \ni \alpha > 0, \gamma \in \R$, and $\beta \in \R^{n \times n}$ be positive definite.  We  define the map $\mathfrak{L}_{\alpha,\beta,\gamma} : \Hzerostan{s+3/2}(\Omega ; \F^n) \times H^{s}(\Omega ; \F) \to (\Hzerotan(\Omega;\F^n))^{\bar{*}}$  via 
\begin{align}\label{eq: bilinear map B}
			  \left\langle \mathfrak{L}_{\alpha,\beta,\gamma}(u,p) , v \right\rangle_{(\Hzerotan)^{\bar{*}}, \Hzerotan} = \int_\Omega \frac{\mu}{2} \mathbb{D} u : \mathbb{D} \overline{v} -  p \diverge \overline{v} - \gamma \p_1 u \cdot \overline{v} +  \frac{1}{\alpha} \int_{\Sigma_0} \beta u \cdot \overline{v}.
\end{align}
Given $F \in (\Hzerotan(\Omega;\F^n))^{\overline{*}}$ and $g \in L^2(\Omega; \F)$, we say that $u \in \Hzerotan(\Omega; \F^n)$ and $p \in L^2(\Omega; \F)$ are weak solutions to \eqref{eq: linear w/o eta} if $\diverge u = g$ and 
			 \begin{align}\label{eq: abst weakform}
				 \left\langle \mathfrak{L}_{\alpha,\beta,\gamma}(u,p) , v \right\rangle_{(\Hzerotan)^{\bar{*}}, \Hzerotan} = \left\langle F, v \right\rangle_{(\Hzerotan)^{\bar{*}}, \Hzerotan}.
			\end{align}
\end{defn}

This notation allows us to efficiently state our weak well-posedness result.

\begin{thm}\label{thm:weaksol}
	Let $\R \ni \alpha > 0$, $\gamma \in \R$, and $\beta \in \R^{n \times n}$ be positive definite.  Define $ \chi_{\alpha,\beta, \gamma}: \Hzerotan(\Omega;\F^n) \times L^2(\Omega;\F) \to (\Hzerotan(\Omega;\F^n))^{\overline{*}} \times L^2(\Omega;\F)$ via $\chi_{\alpha,\beta, \gamma} (u,p) = (\mathfrak{L}_{\alpha,\beta,\gamma}(u,p), \diverge u)$,  where $\mathfrak{L}_{\alpha,\beta,\gamma}$ is defined in \eqref{eq: bilinear map B}. Then $ \chi_{\alpha,\beta, \gamma}$ is an isomorphism. 
\end{thm}
\begin{proof}
	We first define the map $B_\alpha : \Hzerotan(\Omega; \F^n ) \times \Hzerotan(\Omega; \F^n ) \to \F$ via 
	\begin{align}
		B_\alpha(u,v) = \int_\Omega \frac{\mu}{2} \mathbb{D} u : \mathbb{D} \overline{v} - \gamma \p_1 u \cdot \overline{v}  +  \frac{1}{\alpha} \int_{\Sigma_0} \beta u \cdot \overline{v},
	\end{align} 
    which is clearly well-defined and continuous. Note that if $u \in \Hzerotan(\Omega; \F^n)$, then integration by parts shows that 
	\begin{align} \label{eq: by parts calculation}
		\int_\Omega \p_1 u \cdot \overline{u} = - \int_\Omega u \cdot \overline{\p_1 u} = -\overline{\int_\Omega \p_1 u \cdot \overline{u}} \implies 	\Re \int_\Omega \p_1 u \cdot \overline{u} = 0. 
	\end{align}
	Thus by the Korn's inequality from Lemma~\ref{lem: korn}, and using $\alpha > 0$ and  the fact that $\beta$ is positive definite, we have 
	\begin{align} \label{eq:Bcoercive}
		\abs{B_\alpha(u,u) } \ge \Re B_\alpha(u,u) = \frac{1}{2} \int_\Omega \abs{\mathbb{D}u}^2 + \frac{1}{\alpha} \int_{\Sigma_0} \beta u \cdot \overline{u}  \gtrsim \norm{u}^2_{\Hzerotan},
	\end{align}
	which shows that $B_\alpha$ is $\Hzerotan$-coercive. Since $\Hzerostanigtan(\Omega; \F^n)$ is a closed subspace of $\Hzerotan(\Omega; \F^n)$, $B_\alpha$ is a well-defined, continuous, coercive functional that is bilinear when $\F = \R$ and sesquilinear when $\F = \C$. 
	
	Let $(F,g) \in (\Hzerotan(\Omega ; \F^n))^{\overline{*}} \times L^2(\Omega ; \F)$ and define the functional $\Lambda_\alpha \in \left( \Hzerotan(\Omega; \F^n \right)^{\overline{*}}$  via $\langle \Lambda_\alpha, v \rangle =  -B_\alpha(\Pi g,v) + \langle F, v \rangle_{(\Hzerotan)^{\overline{*}}}$, where $\Pi: L^2(\Omega;\F) \to {}_0H^1(\Omega; \F^n)$ is the right inverse of the divergence operator introduced in Lemma~\ref{lem:divinverse}. By applying the standard Lax-Milgram theorem when $\F = \R$ and the anti-dual Lax Milgram theorem (see, for instance, Theorem A.5 of \cite{noahtice}) when $\F = \C$, there exists a unique $u \in \Hzerostanigtan(\Omega; \F^n)$ such that  $B_\alpha(u,v)  = \langle \Lambda_\alpha, v \rangle$  for all $v \in \Hzerostanigtan(\Omega; \F^n)$,  obeying the estimate 
	\begin{align}\label{eq: weak u est}
		\norm{u}_{\Hzerotan} \lesssim \norm{\Lambda}_{(\Hzerotan(\Omega; \F^n))^{\overline{*}}}  \lesssim_{\alpha, n,b} \norm{F}_{\left( \Hzerotan \right)^{\overline{*}}} + \norm{g}_{L^2}.
	\end{align}
	Furthermore, by Corollary~\ref{cor:pressure} there exists a unique $p \in L^2(\Omega; \F)$ such that 
	\begin{align} 
		B_\alpha(u,v)  = -B_\alpha(\Pi g,v) + \langle F, v \rangle_{(\Hzerotan)^{\overline{*}}, \Hzerotan} + \int_\Omega p \diverge \overline{v}
	\end{align}
	for all $v \in \Hzerotan(\Omega; \F^n)$. This shows that $ \chi_{\alpha,\beta, \gamma}(u + \Pi g, p) = (F,g)$, so $ \chi_{\alpha,\beta, \gamma}$ is surjective. 

	On the other hand, if $(u,p) \in \Hzerotan(\Omega;\F^n) \times L^2(\Omega;\F)$ such that $ \chi_{\alpha,\beta, \gamma}(u,p) = (F,g)$, first we can use the Helmholtz decomposition \eqref{eq: helmholtz} to write $u = w + \Pi g$. Then if we use $v = \Pi p$ in the definition of the map $\mathfrak{L}_{\alpha,\beta,\gamma}$ from \eqref{eq: bilinear map B}, we arrive at the estimate
	\begin{align} \label{eq: weak p est}
		\norm{p}_{L^2} \lesssim_{n,b} \norm{\Lambda}_{\left( \Hzerotan \right)^{\overline{*}}} \lesssim_{\alpha, n,b} \norm{u}_{\Hzerotan} + \norm{F}_{\left( \Hzerotan \right)^{\overline{*}}}.
	\end{align}
	The injectivity of $ \chi_{\alpha,\beta, \gamma}$ then follows from the estimates \eqref{eq: weak u est} and \eqref{eq: weak p est}. 
\end{proof}

Next we combine the weak isomorphism with standard elliptic regularity to arrive at our well-posedness result.

\begin{thm}\label{thm:classicalexistence}
Let $s \ge 0$ and assume $\beta \in \R^{n \times n}$ is positive definite. For any $\gamma \in \R$, we define the bounded linear operator $\Phi_{\alpha,\beta,\gamma} : \Hzerostan{s+2}(\Omega ; \F^n) \times H^{s+1}(\Omega;\F) \to H^{s} (\Omega; \F^n) \times H^{s+1}(\Omega;\F) \times H^{s+\frac{1}{2}}(\Sigma_b ; \F^n) \times H^{s+\frac{1}{2}}(\Sigma_0 ; \F^{n-1})$ via $\Phi_{\alpha,\beta,\gamma}(u,p) = \left( \diverge S(p,u) - \gamma \p_1 u, \diverge u, S(p,u) e_n, \left[\alpha S(p,u)e_n + \beta u \right]' \right)$.  Then $\Phi_{\alpha,\beta,\gamma}$ is an isomorphism for all $\gamma \in \R$. 
\end{thm}

\begin{proof}
This follows from Theorem~\ref{thm:weaksol} and the regularity theory for elliptic systems (see, for instance, \cite{ADN_1964}).
\end{proof}

Next we prove an important result that will be essential in the analysis to follow. We show that the weak solution map $ \chi_{\alpha,\beta, \gamma}$ and the strong solution map $\Phi_{\alpha,\beta,\gamma}$ commute with tangential multipliers, as defined in Definition~\ref{defn: tangential m}.

\begin{thm}\label{thm: chi commutes with M_omega}
	Let $s \ge 0$ and suppose $\omega \in L^\infty(\R^{n-1}; \C)$. Consider the tangential multiplier $M_\omega$ defined via Definition~\ref{defn: tangential m}. If $(F,g) \in (\Hzerotan(\Omega;\F^n))^{\overline{*}} \times L^2(\Omega;\F)$ and $(u,p) = \chi_{\alpha,\beta, \gamma}^{-1}(F,g)$, then $(M_\omega u, M_\omega p) = \chi_{\alpha,\beta, \gamma}^{-1}(M_\omega F, M_\omega g)$. Furthermore, if $f \in H^s(\Omega ; \F^n)$, $g \in H^{s+\frac{1}{2}}(\Sigma_b; \F^n)$, $k \in H^{s+\frac{1}{2}}(\Sigma_b; \F^{n})$, $l \in H^{s+\frac{1}{2}}(\Sigma_0; \F^{n-1})$, and we set $(u,p) = \Phi_{\alpha,\beta,\gamma}^{-1}(f,g,k,l)$, then $(M_\omega u, M_\omega p) = \Phi_{\alpha,\beta,\gamma}^{-1}(M_\omega f,M_\omega g,M_\omega k,M_\omega l)$.
\end{thm}
\begin{proof}
	Let $\omega \in L^\infty(\R^{n-1}; \C), (F,g) \in (\Hzerotan(\Omega;\F^n))^{\overline{*}} \times L^2(\Omega;\F)$ and $(u,p) = \chi_{\alpha,\beta, \gamma}^{-1}(F,g)$. We first note that by the definition of $M_\omega$ on $L^2(\Omega; \F)$, the multiplier $M_\omega$ commutes with differential operators and therefore we immediately have $M_\omega g = M_\omega \diverge u = \diverge M_\omega u$. We then note that by the definition of $M_\omega$ on $(\Hzerotan(\Omega;\C^n))^{\overline{*}}$, we may compute for all $v \in \Hzerotan(\Omega;\C^n)$
	\begin{multline}\label{eq: commute F}
		 \left\langle M_\omega F, v \right\rangle_{(\Hzerotan)^{\overline{*}}, \Hzerotan} = \left\langle F, M_{\overline{\omega}} v \right\rangle_{(\Hzerotan)^{\overline{*}}, \Hzerotan} = \int_\Omega \frac{\mu}{2} \mathbb{D}M_\omega u : \mathbb{D} \overline{v}- M_\omega p \diverge \overline{v}  - \gamma  \p_1 M_\omega u \cdot \overline{v} +  \frac{1}{\alpha} \int_{\Sigma_0} \beta M_\omega u \cdot \overline{v} \\
		 = \langle \mathfrak{L}_{\alpha,\beta,\gamma}(M_\omega u, M_\omega p), v \rangle.
	\end{multline}
	Combining these then shows that $(M_\omega u, M_\omega p) = \chi_{\beta,\gamma} ^{-1}(M_\omega F, M_\omega g)$. Next we note that again that by the definition of $M_\omega$ on $H^s(U; \F^k)$ and $L^2(\Sigma; \F^k)$ for $\Sigma \in \{\Sigma_b, \Sigma_0\}$ and $k \ge 1$, the tangential multiplier $M_\omega$ commutes with differential operators and therefore 
	\begin{multline}\label{eq: commute f k l}
		 M_\omega f = M_\omega (\diverge S(p,u) - \gamma \p_1 u) = \diverge M_\omega S(p,u) - \gamma \p_1 M_\omega u = \diverge S(M_\omega p, M_\omega u) - \gamma \p_1 M_\omega u, \; \\ M_\omega k = M_\omega S(p,u) e_n = S(M_\omega p, M_\omega u) e_n, \; M_\omega l = M_\omega \left[
			\alpha S(p,u)e_n + \beta u \right]' = \left[
				\alpha S(M_\omega p, M_\omega u)e_n + \beta M_\omega u \right]'.
	\end{multline}
	Thus, $(M_\omega u, M_\omega p) = \Phi_{\alpha,\beta,\gamma} ^{-1} (M_\omega f,M_\omega g,M_\omega k,M_\omega l)$.
\end{proof}

We conclude this section deriving an $\alpha$-independent estimate for the operator from Theorem \ref{thm:classicalexistence}, assuming that $l = 0$ and $\alpha \in (0,1)$.  We only focus on the case when $\F = \R$, as we only consider real-valued solutions later. 

\begin{prop}\label{prop:ind of alpha}
Suppose  $\gamma \in \R$, and $\beta \in \R^{n \times n}$ is positive definite.  Let  $\R \ni s \ge 0$, $f \in H^s(\Omega;\R^n)$, $g \in H^{s+1}(\Omega ; \R)$, and $k \in H^{s+\frac{1}{2}}(\Sigma_b ; \R^n)$.  Then there exists a constant $C>0$ such that if $\alpha \in (0,1)$ and  $u \in \Hzerostan{s+2}(\Omega;\R^n)$ and $p \in H^{s+1}(\Omega ; \R)$ satisfy \eqref{eq: linear w/o eta} with $l = 0$, then
\begin{align}\label{eq: regularity solution est w/o alpha}
	\norm{u}_{\Hzerostan{s+2}} + \norm{p}_{H^{s+1}} \le C \left( \norm{f}_{H^s} + \norm{g}_{H^{s+1}} + \norm{k}_{H^{s+\frac{1}{2}}(\Sigma_b)} \right).
   \end{align}
\end{prop}
\begin{proof}

Throughout the proof we will use the operators $\mathfrak{J}^t_M$ defined in Lemma \ref{lem: bessel}.  Suppose $(u,p) \in \Hzerostan{s+2}(\Omega ; \R^n) \times H^{s+1}(\Omega ; \R)$ is the solution to \eqref{eq: linear w/o eta} with $f \in H^s(\Omega; \R^n), g \in H^{s+1} (\Omega; \R)$, $k \in H^{s+1/2}(\Sigma_b; \R^n)$ and $l = 0$. Then we note by Theorem~\ref{thm: chi commutes with M_omega}, for any $M > 0$ the tuple $(\mathfrak{J}^{s+1}_M u, \mathfrak{J}^{s+1}_M p)$ $\in \Hzerostan{1}(\Omega;\R^n)$  is the solution to \eqref{eq: linear w/o eta} with data $\mathfrak{J}^{s+1}_M f \in H^s(\Omega; \R^n), \mathfrak{J}^{s+1}_Mg \in H^{s+1}(\Omega; \R)$, $\mathfrak{J}^{s+1}_M k$ $\in H^{s+1/2}(\Sigma_b; \R^n)$ and $l = 0$.

We may then use $\mathfrak{J}^{s+1}_M u \in \Hzerostan{1}(\Omega;\R^n)$ as a test function in the weak formulation \eqref{eq: abst weakform} to obtain
\begin{align}
	 \int_\Omega \frac{\mu}{2} \abs{\mathbb{D} \mathfrak{J}^{s+1}_M u}^2 + \frac{1}{\alpha} \int_{\Sigma_0} \beta \mathfrak{J}^{s+1}_M u : \mathfrak{J}^{s+1}_M u = \int_\Omega \mathfrak{J}^{s}_M f : \mathfrak{J}^{s+2}_M u - \int_{\Sigma_b} \mathfrak{J}^{s+1}_M k : \mathfrak{J}^{s+1}_M u + \int_\Omega \mathfrak{J}^{s+1}_M p \mathfrak{J}^{s+1}_M g.
\end{align}
Since $\alpha \in (0,1)$, by Lemmas \ref{lem: korn} and \ref{lem: bessel} and trace theory, we have that there exists constants $c_1, c_2$ independent of $\alpha$ and $M$ such that 
\begin{multline}\label{eq: absorb ineq 1}
	 c_1 \norm{\mathfrak{J}^{s+1}_M u}_{\Hzerotan}^2 \le \norm{\mathfrak{J}^{s}_M f}_{L^2}  \norm{\mathfrak{J}^{s+2}_M u}_{L^2} + \norm{\mathfrak{J}^{s+1}_M k}_{H^{-1/2}(\Sigma_b)} \norm{\mathfrak{J}^{s+1}_M u}_{H^{1/2}(\Sigma_b)}  + \norm{\mathfrak{J}^{s+1}_M p}_{L^2} \norm{\mathfrak{J}^{s+1}_M g}_{L^2}  \\
	 \le c_2 \left( \norm{\mathfrak{J}^{s}_M f}_{L^2}^2 + \norm{\mathfrak{J}^{s+1}_M k}_{H^{-1/2}(\Sigma_b)}^2 \right)  + \norm{\mathfrak{J}^{s+1}_M p}_{L^2} \norm{\mathfrak{J}^{s+1}_M g}_{L^2}  +\frac{c_1}{2} \norm{\mathfrak{J}^{s+1}_M u}_{H^1(\Omega)}^2.
\end{multline}
By absorbing the last term on the right side of \eqref{eq: absorb ineq 1} and again using Lemma~\ref{lem: bessel}, we have for any $\ve > 0$,
\begin{multline}\label{eq: absorbtion est}
	\norm{\mathfrak{J}^{s+1}_M u}_{\Hzerotan} \lesssim \norm{\mathfrak{J}^{s}_M f}_{L^2} + \frac{1}{\ve} \norm{\mathfrak{J}^{s+1}_M g}_{L^2}  + \norm{\mathfrak{J}^{s+1}_M k}_{H^{-1/2}(\Sigma_b)} + \ve \norm{\mathfrak{J}^{s+1}_M p}_{L^2} \\
	\lesssim \norm{ f}_{H^s} + \frac{1}{\ve} \norm{g}_{H^{s+1}}  + \norm{k}_{H^{s+1/2}(\Sigma_b)} + \ve \norm{\mathfrak{J}^{s+1}_M p}_{L^2}.
\end{multline}

Next, we seek to derive a priori estimates on the pressure.  By Lemma~\ref{lem:divinverse} there exists $v_0 \in {}_0 H^1(\Omega; \F^n)$ such that $\diverge v_0 = \mathfrak{J}^{s+1}_M p$ and $\norm{v_0}_{{}_0 H^1} \lesssim_{n,b} \norm{\mathfrak{J}^{s+1}_M p}_{L^2}$.  Using $v_0$ in the weak formulation \eqref{eq: abst weakform} with the same data, we find that there exists a constant $\mathfrak{C}  = \mathfrak{C} (\mu, \gamma, s,b) >  0$ independent of $\alpha$ and $M$ such that  
\begin{multline}\label{eq: p u f k est}
	 \norm{\mathfrak{J}^{s+1}_M p}_{L^2}^2  = \int_{\Omega}\frac{\mu}{2} \mathbb{D} \mathfrak{J}^{s+1}_M u : \mathbb{D} v_0 - \gamma \p_1 \mathfrak{J}^{s+1}_M u \cdot v_0 - \mathfrak{J}^{s}_M f: \mathfrak{J}^1 v_0 + \int_{\Sigma_b} \mathfrak{J}^{s+1}_M k : v_0 \\
	 \le \frac{\mu}{2} \norm{\mathfrak{J}^{s+1}_M u}_{\Hzerotan} \norm{v_0}_{{}_0H^1} + \abs{\gamma} \norm{\mathfrak{J}^{s+1}_M u}_{\Hzerotan} \norm{v_0}_{L^2}  +  \norm{\mathfrak{J}^{s}_M f}_{L^2} \norm{ v_0}_{{}_0 H^1}   + \norm{\mathfrak{J}^{s+1}_M k}_{H^{-1/2}(\Sigma_b)} \norm{v_0}_{{}_0 H^1}  \\
	 \le  \mathfrak{C}\left( \norm{\mathfrak{J}^{s+1}_M u}_{\Hzerotan}^2 + \norm{f}_{H^s}^2 + \norm{k}_{H^{s+1/2}(\Sigma_b)}^2 \right) +  \frac{1}{2} \norm{\mathfrak{J}^{s+1}_M p}_{L^2}^2.
\end{multline}
Thus by another absorption argument we find that
\begin{align}\label{eq: absorbtion est p}
	\norm{\mathfrak{J}^{s+1}_M p}_{L^2} \lesssim \norm{\mathfrak{J}^{s+1}_M u}_{\Hzerotan} + \norm{f}_{H^s} + \norm{k}_{H^{s+1/2}(\Sigma_b)},
\end{align}
where the universal constant is independent of $\alpha$ and $M$. By combining \eqref{eq: absorbtion est} and \eqref{eq: absorbtion est p} we may then choose $\ve > 0$ sufficiently small so that 
\begin{align}\label{eq: p M est}
	\norm{\mathfrak{J}^{s+1}_M u}_{\Hzerotan}  + \norm{\mathfrak{J}^{s+1}_M p}_{L^2} \lesssim \norm{f}_{H^s} + \norm{g}_{H^{s+1}}+ \norm{k}_{H^{s+1/2}(\Sigma_b)}.
\end{align}
Since the universal constant in \eqref{eq: p M est} is independent of $M$, we may apply the monotone convergence theorem to conclude that 
\begin{align}\label{eq: p L^2 est}
	\norm{\mathfrak{J}^{s+1} u}_{\Hzerotan} + \norm{\mathfrak{J}^{s+1} p}_{L^2} \lesssim \norm{f}_{H^s} + \norm{g}_{H^{s+1}}+ \norm{k}_{H^{s+1/2}(\Sigma_b)}.
\end{align}
Standard elliptic regularity results (see \cite{a-d-n}, for instance) then show that 
\begin{align}\label{eq: ellitpic est}
	 \norm{u}_{H^{s+2}} + \norm{\nabla p}_{H^s} \lesssim \norm{f}_{H^s} + \norm{g}_{H^{s+1}} + \norm{\Tr_{\Sigma_b} u}_{H^{s+3/2}(\Sigma_b)} +  \norm{\Tr_{\Sigma_0} u}_{H^{s+3/2}(\Sigma_0)}.  
\end{align}
Lemma~\ref{lem: bessel}, the identity \eqref{eq: trace multiplier commute}, and trace theory show that $\Sigma \in \{\Sigma_b, \Sigma_0\}$, 
\begin{align}\label{eq: trace est}
	 \norm{\Tr_{\Sigma} u}_{H^{s+3/2}(\Sigma)} = \norm{ \mathfrak{J}^{s+1} \Tr_{\Sigma} u}_{H^{1/2}(\Sigma)} = \norm{\Tr_{\Sigma} \mathfrak{J}^{s+1} u}_{H^{1/2}(\Sigma)} \lesssim \norm{\mathfrak{J}^{s+1} u}_{H^1(\Omega)}.
\end{align}
Thus by combining \eqref{eq: p L^2 est}, \eqref{eq: ellitpic est}, and \eqref{eq: trace est} we find that  
\begin{align}\label{eq: u f k est}
	 \norm{u}_{H^{s+2}} + \norm{p}_{H^{s+1}} \lesssim  \norm{f}_{H^s} + \norm{g}_{H^{s+1}} +  \norm{k}_{H^{s+1/2}(\Sigma_b)},
\end{align}
where the universal constant in \eqref{eq: u f k est} is uniform over $\alpha \in (0,1)$. 
\end{proof}

Combining Theorem~\ref{thm:classicalexistence} and Proposition~\ref{prop:ind of alpha} gives us the following corollary.

\begin{cor}\label{thm: alpha classicalexistence}
	Let $s \ge 0$ and assume $\beta \in \R^{n \times n}$ is positive definite. For any $\gamma \in \R$ and $\alpha \in (0,1)$, we define the bounded linear operator $\Theta_{\alpha,\beta,\gamma} : \alphaHzeros{s+2}(\Omega ; \F^n) \times H^{s+1}(\Omega;\F) \to H^{s} (\Omega; \F^n) \times H^{s+1}(\Omega;\F) \times H^{s+\frac{1}{2}}(\Sigma_b ; \F^n)$
		via $\Theta_{\alpha,\beta,\gamma}(u,p) = \left( \diverge S(p,u) - \gamma \p_1 u, \diverge u, S(p,u) e_n\right)$. 	Then $\Theta_{\alpha,\beta,\gamma}$ is an isomorphism for all $\gamma \in \R$. Furthermore, there exists a constant $M > 0$ such that 
		\begin{align}\label{eq: alpha ind Theta}
			 \sup_{\alpha \in (0,1)} \norm{\Theta_{\alpha,\beta,\gamma}}_{\mathcal{L}(\alphaHzeros{s+2} \times  H^{s+1}; H^{s} \times H^{s+1} \times H^{s+\frac{1}{2}}(\Sigma_b))} \le M. 
		\end{align} 
	\end{cor}
	
	\begin{proof}
	The fact that $\Theta_{\alpha,\beta,\gamma}$ is an isomorphism follows immediately from the definition of the space $\alphaHzeros{s+2}(\Omega ; \F^n)$ recorded as the second item of Definition~\ref{defn:OPB} and Theorem~\ref{thm:classicalexistence}, and \eqref{eq: alpha ind Theta} follows immediately from Proposition~\ref{prop:ind of alpha}.
	\end{proof}

\section{The overdetermined $\gamma$-Stokes problem}\label{sec: overdetermined}

Our goal in this section is to extend the linear analysis of the system \eqref{eq: linear w/o eta} to the overdetermined variant 
\begin{align} \label{eq: overdetermined}
	\begin{cases}
		\diverge S(p,u) - \gamma \p_1 u = f, & \text{in} \; \Omega \\
		\diverge u = g, & \text{in} \; \Omega \\
		u_n = h, & \text{on} \; \Sigma_b \\
		S(p,u) e_n = k, & \text{on} \; \Sigma_b \\
		[\alpha S (p,u) e_n + \beta u]' = l, & \text{on} \; \Sigma_0 \\
		u_n = 0, & \text{on} \; \Sigma_0
	\end{cases}
\end{align}
obtained from \eqref{eq: linear w/o eta} by appending the equation $u_n =h$ on $\Sigma_b$.

\subsection{The specified divergence problem and the divergence-trace compatibility condition}

In this subsection we establish results concerning the specified divergence problem 
\begin{align} \label{eq:spdiv}
	\begin{cases}
		\diverge u = g, & \text{in} \; \Omega \\
		u_n = h, & \text{on} \; \Sigma_b \\
		u_n = 0, & \text{on} \; \Sigma_0,
	\end{cases}
\end{align}
over $\F \in \{\R, \C\}$. The system \eqref{eq:spdiv} is overdetermined in the sense that a non-trivial compatibility condition needs to be satisfied by the data $g$ and $h$. We record this condition below. 

\begin{lem}\label{lem:divtrace}
	Let $u \in \Hzerotan(\Omega ; \F^n)$ and let $g = \diverge u \in L^2(\Omega; \F)$ and $h = u_n \rvert_{\Sigma_b} \in H^{\frac{1}{2}}(\Sigma_b; \F)$. Then 
	\begin{align} \label{eq: divtraceestimate}
		h(\cdot) - \int_{0}^b g(\cdot, x_n) \; dx_n \in \dot{H}^{-1} (\R^{n-1}; \F)
	\text{ and }
	\left[ h - \int_0^b g(\cdot, x_n) \; dx_n \right]_{\dot{H}^{-1}} \le 2\pi \sqrt{b} \norm{u}_{L^2}. 
	\end{align}
\end{lem}
\begin{proof}
Theorem 3.1 in \cite{leonitice} establishes this for $u$ that entirely vanish on $\Sigma_0$. However, an inspection of the proof there shows that it really only requires $u_n = 0$ on $\Sigma_0$, so the same argument proves the result for $u \in\Hzerotan(\Omega;\F^n)$.
\end{proof}

The next result constructs a right inverse to \eqref{eq:spdiv}.

\begin{prop}\label{prop:Q}
	Consider the Hilbert space $\mathcal{H}(\Omega; \F) = \{ (g,h) \in L^2(\Omega; \F) \times H^{\frac{1}{2}}(\Omega ; \F) : \norm{(g,h)}_{\mathcal{H}} < \infty \},$ where $\norm{(g,h)}_{\mathcal{H}}$ is defined via $\norm{(g,h)}_{\mathcal{H}}^2 = \norm{g}_{L^2}^2 + \norm{h}_{H^{\frac{1}{2}}}^2 + \left[ h - \int_{0}^b g(\cdot, x_n) \; dx_n \right]^2_{\dot{H}^{-1}}$. 	There exists a bounded linear operator $G: \mathcal{H}(\Omega; \F) \to {}_0H^{1}(\Omega; \F^n)$ such that $u = G(g,h)$ satisfies \eqref{eq:spdiv}.
\end{prop}
\begin{proof}
This is Proposition 2.4 in \cite{noahtice}. 
\end{proof}

\subsection{Adjoint problem analysis}\label{sec:asymptotics}

Now we are ready to study the $\R$-solvability of the system \eqref{eq: overdetermined}.  We first record its formal adjoint, the underdetermined problem  (here in homogeneous form)
\begin{align} 
	\begin{cases}\label{eq:adjoint}
		\diverge S(q,v) + \gamma \p_1 v = 0, &\text{in} \; \Omega \\
	\diverge v = 0, &\text{in} \; \Omega\\
	(S(q,v) e_n )' = 0, &\text{on} \; \Sigma_b \\
	[\alpha S(q,v)e_n' + \beta^T v]' = 0, &\text{on} \; \Sigma_{0} \\
	v_n = 0, &\text{on} \; \Sigma_{0}.	
	\end{cases}
	\end{align}
We note in particular that since $\beta w \cdot w = \beta^T w \cdot w$ for all $w \in \R^n$, $\beta^T$ is positive definite whenever $\beta$ is.  As a consequence, we can augment the third equation with the extra condition $S(p,u)e_n \cdot e_n = \psi$ for arbitrary $\psi \in H^{s+1/2}(\Sigma_b)$ in order to parameterize the solution space via the  isomorphism $\Phi_{\alpha, \beta^T, -\gamma}$ from Theorem~\ref{thm:classicalexistence}.

Throughout the rest of this subsection we aim to develop the asymptotics of some special functions associated to the map  $\Phi_{\alpha, \beta^T, -\gamma}$ from Theorem~\ref{thm:classicalexistence}, which we call the normal stress to solution map.  First we define symbols of the pseudodifferential operator associated to this map.

\begin{defn}
	Let $\gamma \in \R$, $\beta \in \R^{n \times n}$ be positive definite, and $s \in [-1,\infty)$. We define the \emph{normal stress to velocity} and the \emph{normal stress to pressure} maps to be the bounded linear maps $\mathcal{U}_{\alpha,\beta,\gamma} : H^{s+\frac{1}{2}}(\Sigma_b; \F) \to H^{s+2}\left(\Omega;\F^n \right)$ and $\mathcal{P}_{\alpha,\beta,\gamma} : H^{s+\frac{1}{2} }(\Sigma_b ; \F) \to H^{s+1}(\Omega; \F)$ defined via $(\mathcal{U}_{\alpha,\beta,\gamma}(\psi), \mathcal{P}_{\alpha,\beta,\gamma} (\psi) ) = \Phi^{-1}_{\alpha, \beta^T, -\gamma}(0,0, \psi e_n, 0),$ 	where $\Phi_{\alpha, \beta^T, -\gamma}$ is the isomorphism from Theorem~\ref{thm:classicalexistence}. In other words, $(\mathcal{U}_{\alpha,\beta,\gamma}(\psi), \mathcal{P}_{\alpha,\beta,\gamma}(\psi))$ is the unique solution to the adjoint problem \eqref{eq:adjoint} for a given $\psi \in  H^{s+\frac{1}{2} }(\Sigma_b ; \F)$.
\end{defn}

\begin{thm}\label{thm: V Q symbols}
	There exist bounded, measurable functions $V_{\alpha,\beta}, Q_{\alpha,\beta} : \R \times [0,b] \times \R \to \C$ such that $\overline{V_{\alpha,\beta}(\xi,x_n,\gamma)} = V_{\alpha,\beta}(-\xi,x_n,\gamma), \overline{Q_{\alpha,\beta}(\xi,x_n,\gamma)} = Q_{\alpha,\beta}(-\xi,x_n,\gamma)$ for $\xi \in \R^{n-1}$ a.e., and for $s \in [-1,\infty)$ and all $\psi \in H^{s+\frac{1}{2}}(\Sigma_b ; \R)$, we have $\widehat{\mathcal{U}_{\alpha,\beta,\gamma}(\psi)}(\xi,x_n) = V_{\alpha,\beta}(\xi, x_n, -\gamma) \hat{\psi}(\xi)$  and $\widehat{\mathcal{P}_{\alpha,\beta,\gamma}(\psi)} (\xi,x_n) = Q_{\alpha,\beta}(\xi, x_n, -\gamma) \hat{\psi}(\xi)$. 	Moreover, for all $\alpha \in \R$ there exists a constant $c > 0$ such that for a.e. $\xi \in \R^{n-1}$, we have 
	\begin{align}\label{eq:VQestimates}
		\abs{V_{\alpha,\beta}(\xi,x_n, \gamma)} \le c (1+\abs{\xi}^2)^{-\frac{1}{2}}.
	\end{align}
\end{thm}
\begin{proof}
	We note that for fixed $x_n \in [0,b]$, the map $\psi \mapsto \mathcal{U}_{\alpha,\beta,\gamma}(\psi)(\cdot, x_n)$ is a bounded linear translation-invariant map between $H^{s+1/2}(\R^d; \F)$ and $H^{s+3/2}(\R^d; \F^n)$ and the map $\psi \mapsto \mathcal{P}_{\alpha,\beta,\gamma}(\psi)(\cdot, x_n)$ is a bounded linear translation-invariant map between $H^{s+1/2}(\R^d; \F)$ and $H^{s+1/2}(\R^d; \F)$ by Theorem~\ref{thm:classicalexistence}. Thus the existence of $V_{\alpha,\beta}$, $Q_{\alpha,\beta}$, and the estimate \eqref{eq:VQestimates} is guaranteed by Proposition~\ref{prop: tangential m}. Since $\psi$ is assumed to be real-valued, $\mathcal{U}_{\alpha, \gamma}, \mathcal{P}_{\alpha, \gamma}$ are also real-valued, thus it follows that $\overline{V_{\alpha,\beta}(\xi,x_n,\gamma)} = V_{\alpha,\beta}(-\xi,x_n,\gamma), \overline{Q_{\alpha,\beta}(\xi,x_n,\gamma)} = Q_{\alpha,\beta}(-\xi,x_n,\gamma)$ for a.e. $\xi \in \R^{n-1}$. The estimate \eqref{eq:VQestimates} follows from trace theory and the estimate \eqref{eq: esssup multiplier est} recorded as a part of Proposition~\ref{prop: tangential m}.
\end{proof}

We then define $m_{\alpha,\beta}: \R \times \R \to \C$ via 
\begin{align} \label{eq:dtonmap}
	m_{\alpha,\beta}(\xi,\gamma) = V_{\alpha,\beta}(\xi,b,\gamma) \cdot e_n,
\end{align}
which can be viewed as the symbol of the normal stress to Dirichlet pseudodifferential operator $\psi \mapsto u_n \rvert_{\Sigma_b}$.

Recall that by Theorem~\ref{thm:weaksol} and trace theory, we have the equivalence $\|\psi\|_{H^{-\frac{1}{2}}} \asymp \|u\|_{_{0}H^{1}} + \|p\|_{L^2}$. The next theorem shows that if we weaken the control of $\psi$ at low frequencies on the Fourier side, we then have a norm equivalence without the pressure term. Now note in particular that the constant appearing in \eqref{eq: energy equiv} can be made to be uniform in the parameter $\alpha$ if $\alpha \in (0,1)$.  First we need some notation.

\begin{defn}
Let $\F \in \{\R, \C\}, \R \ni \alpha > 0$.  For $s \ge -1$, we define  $\mathscr{O} : H^{s+\frac{1}{2}}(\Sigma_b;\F^n)  \times H^{s+\frac{1}{2}}(\Sigma_0;\F^{n-1}) \to (\Hzerotan(\Omega;\F^n))^{\bar{*}}$ by the action on $v \in \Hzerotan(\Omega; \F^n)$ via 
	  \begin{align}\label{eq: map O}
	  \left\langle \mathscr{O}(k, l ), v \right\rangle_{(\Hzerotan)^{\bar{*}}, \Hzerotan} = \left\langle k,  v \rvert_{\Sigma_b} \right\rangle_{H^{-\frac{1}{2}}, H^{\frac{1}{2}}} -  \frac{1}{\alpha} \left\langle l,  v' \rvert_{\Sigma_0} \right\rangle_{H^{-\frac{1}{2}}, H^{\frac{1}{2}}},
	  \end{align}
	  where $\left\langle k,  v \rvert_{\Sigma_b} \right\rangle_{H^{-\frac{1}{2}}, H^{\frac{1}{2}}}$ denotes the dual paring between $k \in H^{-\frac{1}{2}}(\Sigma_b; \F^n) = ( H^{\frac{1}{2}}(\Sigma_b;\F^n) )^{\overline{*}}$ and $v \rvert_{\Sigma_b} \in H^{1/2} (\Sigma_b ; \F^n)$, and similarly $\left\langle l,  v' \rvert_{\Sigma_0} \right\rangle_{H^{-\frac{1}{2}}, H^{\frac{1}{2}}}$ denotes the dual paring between $l \in( H^{\frac{1}{2}}(\Sigma_0;\F^{n-1}) )^{\overline{*}}$ and $v' \rvert_{\Sigma_b} \in H^{1/2} (\Sigma_0 ; \F^{n-1})$.  Clearly, $\mathscr{O}$ is bounded and linear.
 
\end{defn}

We can now state and prove the previously mentioned result.

\begin{thm}\label{thm:energyequiv}
 Suppose $\R \ni \alpha > 0, \beta \in \R^{n \times n}$ is positive definite, and $\gamma \in \R$. Let $\psi \in H^{-1/2}(\Sigma_b ; \F)$ and consider $(u,p) = \chi_{\beta^T, -\gamma}^{-1}(\mathscr{O}(\psi e_n, 0),0)$. The following hold.

 \begin{enumerate}
	 \item There exists a constant $c > 0$ such that 
	 \begin{align} \label{eq: energy equiv}
		 c^{-1} \|u\|_{\Hzerotan} \le \left( \int_{\R^{n-1}} \min \{ \abs{\xi}^2 , \abs{\xi}^{-1} \} \abs{\hat{\psi}(\xi)}^2 \; d\xi \right)^{\frac{1}{2}} \le c \|u\|_{\Hzerotan} .
	 \end{align}
	 \item Furthermore, there exists a constant $c > 0$ depending on physical parameters and $\gamma$ such that \eqref{eq: energy equiv} holds for all $\alpha \in (0,1)$. In other words, the constant $c$ can be chosen to be independent of $\alpha$ if $\alpha \in (0,1)$.
 \end{enumerate}	
\end{thm}

\begin{proof}
First we note that by the weak form of the system \eqref{eq: abst weakform} and the definition of the map $\mathscr{O}$ via \eqref{eq: map O}, we have 
\begin{align}\label{eq: energy weak form 1}
	\int_\Omega \frac{\mu}{2} \mathbb{D}u : \mathbb{D} \overline{v }- p\diverge \overline{v} + \gamma \p_1 u \cdot \overline{v} + \frac{1}{\alpha} \int_{\Sigma_0} \beta u' \cdot \overline{v'} = -\langle \psi e_n , v \rvert_{\Sigma b} \rangle_{H^{-1/2}(\Sigma_b),H^{1/2}(\Sigma_b)}
\end{align}
for all $v \in \Hzerotan(\Omega; \F^n)$. By letting $v = u \in \Hzerostanigtan(\Omega; \F^n)$, taking the real part of \eqref{eq: energy weak form 1}, using \eqref{eq: by parts calculation}, Lemma~\ref{lem: korn}, the fact that $\alpha > 0$, \eqref{eq: beta coercive C} and the anti-dual representation of Sobolev spaces (see Proposition A.6 of \cite{noahtice}) we have
\begin{multline} \label{eq: energy equiv est 1}
	\|u\|_{\Hzerotan}^2 \lesssim_\alpha \Re\left( \int_\Omega  \frac{\mu}{2} \abs{\mathbb{D}u}^2 +  \frac{1}{\alpha}\int_{\Sigma_0} \beta u'\cdot \overline{u'} \right)  \\ = -\Re \langle \psi, \Tr_{\Sigma_b} u \cdot e_n \rangle_{H^{-1/2},H^{1/2}} = -\Re \int_{\R^{n-1}} \hat{\psi}(\xi) \overline{\widehat{\Tr_{\Sigma_b} u \cdot e_n}(\xi)} \; d\xi \\
			  \le \left(  \int_{\R^{n-1}} \min \{\abs{\xi}^2, \abs{\xi}^{-1} \} \abs{\hat{\psi}(\xi)}^2\; d\xi \right)^{\frac{1}{2}} \left( \int_{\R^{n-1}} \max \{\abs{\xi}^{-2}, \abs{\xi}^{1} \} \abs{ \widehat{\Tr_{\Sigma_b} u \cdot e_n}(\xi)}^2 \; d\xi  \right)^{\frac{1}{2}}.
\end{multline}
In particular, we note that if $\alpha \in (0,1)$, then $\alpha^{-1} > 1$ and thus we may choose the constant on the left hand side of \eqref{eq: energy equiv est 1} to be independent of $\alpha$. 
By using the divergence-trace compatibility estimate \eqref{eq: divtraceestimate}, we have 
\begin{align} 
	\int_{\R^{n-1}} \max \{\abs{\xi}^{-2}, \abs{\xi}^{1} \} \abs{ \widehat{\Tr_{\Sigma_b} u \cdot e_n}(\xi) }^2\; d\xi \le \|\Tr_{\Sigma_b} u \cdot e_n \|_{\dot{H}^{-1} \cap H^{\frac{1}{2}}  } \lesssim \|u\|_{\Hzerotan}.
\end{align}
This gives us the left hand side of \eqref{eq: energy equiv}. 

For the right hand side, we first define $\phi \in H^{1/2}(\Sigma_b ; \F) \cap \dot{H}^{-1}(\Sigma_b ; \F)$ via $\hat{\phi}(\xi) = \min \{ \abs{\xi}^2, \abs{\xi}^{-1} \} \hat{\psi}(\xi),$ where 
\begin{align} 
	\|\phi \|^2_{H^{\frac{1}{2}} \cap \dot{H}^{-1}} \le 2 \int_{\R^{n-1}} \max \{\abs{\xi}^{-2}, \abs{\xi}^{1} \} \abs{\min \{\abs{\xi}^2, \abs{\xi}^{-1} \} \hat{\psi}(\xi)}^2 \; d\xi = 2 \int_{\R^{n-1}}  \min \{\abs{\xi}^2, \abs{\xi}^{-1} \} \abs{ \hat{\psi}(\xi)}^2 \; d\xi.
\end{align}
Note that for $\abs{\xi} \le 1$, $\abs{\xi}^2 \lesssim (1+\abs{\xi}^2)^{-1/2}$ and for $\abs{\xi} \ge 1$, $\abs{\xi}^{-1} \lesssim (1+\abs{\xi}^2)^{-1/2}$. So we find that $\|\phi \|^2_{H^{\frac{1}{2}} \cap \dot{H}^{-1}} \lesssim \norm{\phi}_{H^{-1/2}}$, and therefore we can apply Proposition~\ref{prop:Q} and consider $w = G(0, \phi) \in {}_0 H_{\sigma}^1(\Omega; \F)$, for which we have the estimate 
\begin{align} \label{est:w}
\|w\|_{\Hzerotan}^2 \lesssim \norm{G(0,\phi)}_{\mathcal{H}(\Omega)} = \|\phi\|^2_{\dot{H}^{-1} \cap H^{1/2}} \lesssim  \int_{\R^{n-1}}  \min \{\abs{\xi}^2, \abs{\xi}^{-1} \} \abs{ \hat{\psi}(\xi)}^2 \; d\xi. 
\end{align}
Now using $w \in {}_0H^1_\sigma(\Omega ; \F^n)$ in the weak formulation of the adjoint problem \eqref{eq:adjoint} gives us 
\begin{align}
     \left\langle \psi, w_n \rvert_{\Sigma_b}\right\rangle_{H^{-1/2} (\Sigma_b), H^{1/2}(\Sigma_b)} = - \int_\Omega \frac{\mu}{2} \mathbb{D}u : \mathbb{D} \overline{w} - \gamma \p_1 u \cdot \overline{w}
\end{align}
and using the anti-dual representation of Sobolev spaces we have 
\begin{align}
    \left\langle \psi, w_n \rvert_{\Sigma_b}\right\rangle_{H^{-1/2}(\Sigma_b), H^{1/2}(\Sigma_b)} = \int_{\R^{n-1}} \hat{\psi} \overline{\hat{\phi}} \; d\xi = \int_{\R^{n-1}}\min \{\abs{\xi}^2, \abs{\xi}^{-1} \} \abs{ \hat{\psi}(\xi)}^2 \; d\xi,
\end{align}
and
\begin{align} \label{est:rhs}
	\int_{\R^{n-1}}  \min \{\abs{\xi}^2, \abs{\xi}^{-1} \} \abs{ \hat{\psi}(\xi)}^2 \; d\xi \le \abs{\int_\Omega \frac{\mu}{2} \mathbb{D}u : \mathbb{D} \overline{w} + \gamma \p_1 u \cdot \overline{w}} \lesssim_{\mu, \gamma,\beta} \|u\|_{\Hzerotan} \norm{w}_{\Hzerotan}.
\end{align}
Combining \eqref{est:rhs} with \eqref{est:w} gives us the desired inequality \eqref{eq: energy equiv}. We note that since the constant appearing in \eqref{est:rhs} does not depend on $\alpha$, the second item follows. 
\end{proof}

Utilizing the energy equivalence established above, we are now ready to establish some key estimates for $V_{\alpha,\beta}$, $Q_{\alpha,\beta}$, $m_{\alpha,\beta}$ as defined in Theorem~\ref{thm: V Q symbols} and \eqref{eq:dtonmap}. 

\begin{thm}\label{thm:VandQ}
	Let $\R \ni \alpha > 0, \gamma \in \R$, and $\beta \in \R^{n \times n}$ be positive definite. The following hold.
	\begin{enumerate}
		 \item There exists a constant $c > 0$  such that for all a.e. $\xi \in \R^{n-1}$, we have 
		 \begin{align} \label{eq: V estimates}
			 \int_0^b \abs{V_{\alpha,\beta}(\xi, x_n, -\gamma)}^2 \; dx_n \le c \min \{ \abs{\xi}^2, \abs{\xi}^{-2}\}, \quad \abs{V_{\alpha,\beta}(\xi, 0, -\gamma)}^2 \le c \min \{ \abs{\xi}^2, \abs{\xi}^{-2} \}. 
		 \end{align}
		 and 
		 \begin{align} \label{eq: Q estimates}
			  \int_0^b \abs{Q_{\alpha,\beta}(\xi, x_n, -\gamma) - 1}^2 \; dx_n \le c \abs{\xi}^2.
		 \end{align}
		 \item There exists a constant $c > 0$ depending on physical parameters and $\gamma$ such that \eqref{eq: V estimates} and \eqref{eq: Q estimates} hold for all $\alpha \in (0,1)$. In other words, the constant $c$ can be chosen to be independent of $\alpha$ if $\alpha \in (0,1)$.
	\end{enumerate}

\end{thm}

\begin{proof}
	To prove the first item, we note that by Parseval's theorem, Tonelli's theorem and Theorem~\ref{thm:energyequiv}, we have 
	\begin{align} \label{eq: Vpsi est}
		\int_0^b \int_{\R^{n-1}} \abs{  V_{\alpha,\beta}(\xi,x_n, -\gamma) \hat{\psi}(\xi) }^2 \; d\xi dx_n = \|u\|_{L^2}^2 \lesssim \|u\|_{\Hzerotan}^2 \lesssim_\alpha 	\int_{\R^{n-1}}  \min \{\abs{\xi}^2, \abs{\xi}^{-1} \} \abs{ \hat{\psi}(\xi)}^2 \; d\xi,
	\end{align}
	for all $\psi \in H^{-\frac{1}{2}}(\Sigma_b;\F)$. Let $\varphi \in L^1(\R^{n-1}; \R)$ such that $\varphi(\xi) \ge 0$ a.e. with compact support. Define $\phi \in \bigcap_{s \in \R} H^s (\R^{n-1} ; \C)$ via $\phi = \mathscr{F}^{-1}[ \sqrt{\varphi}]$, and we take $\psi = \phi$ in the inequality above. This gives us
	\begin{align}\label{eq: V alpha est 1}
		\int_{\R^{n-1}} \left( \int_0^b \abs{  V_{\alpha,\beta}(\xi,x_n, -\gamma) }^2 \; dx_n \right)  \varphi(\xi) \; d\xi \lesssim_\alpha 	\int_{\R^{n-1}}  \min \{ \abs{\xi}^2 , \abs{\xi}^{-1}\} \varphi(\xi) \; d\xi.
	\end{align}
	Since this holds for all $\varphi \in L^1(\R^{n-1} ; \R)$, for all $\alpha > 0$ there exists a constant $C > 0$ such that
	\begin{align} 
	\int_0^b \abs{V_{\alpha,\beta}(\xi, x_n, -\gamma)}^2 \; dx_n &\le C \min \{ \abs{\xi}^2, \abs{\xi}^{-1}\} \; \text{for a.e.} \; \xi \in \R^{n-1}.
\end{align}
		Combining this with the estimate \eqref{eq:VQestimates} gives us the first estimate in \eqref{eq: V estimates}. Furthermore, we note that by the second item of Theorem~\ref{thm:energyequiv}, the constants appearing on the right hand side of \eqref{eq: Vpsi est} and \eqref{eq: V alpha est 1} can be chosen to be independent of $\alpha$ if $\alpha \in (0,1)$.

		We note that since
		\begin{align}
		 \int_{\R^{n-1}} \abs{  V_{\alpha,\beta}(\xi,0, -\gamma) \hat{\psi}(\xi) }^2 \; d\xi dx_n = \norm{\Tr_{\Sigma_0} u}_{L^2}^2 \lesssim \norm{u}_{\Hzerotan}^2,  
		\end{align}
	    applying the exact same argument as above gives us the second estimate in \eqref{eq: V estimates}. 
	

    To prove \eqref{eq: Q estimates}, we note that $\hat{p}- \hat{\psi} = (Q_{\alpha,\beta}-1)\hat{\psi}$ and recall that the weak formulation of the system requires 
	\begin{align} 
		 \int_\Omega \frac{\mu}{2} \mathbb{D}u : \mathbb{D}\overline{v} - p \diverge \overline{v} + \gamma \p_1 u \cdot \overline{v} +  \frac{1}{\alpha}\int_{\Sigma_0} \beta u' \cdot \overline{v'} =- \langle \psi e_n, v \rvert_{\Sigma_b} \rangle_{H^{-1/2},H^{1/2}}= -\int_{\R^{n-1}} \psi(x') \overline{v_n}  \; dx',
	\end{align}
	for all $v \in \Hzerotan(\Omega; \F^n)$. Let $v = \Pi( p - \psi(x') ) \in {}_0H^1(\Omega; \F^n)$ where $\Pi : L^2(\Omega; \F) \to {}_0H^1(\Omega; \F^n)$ is the right inverse to the divergence operator appearing in Lemma~\ref{lem:divinverse}. Then by testing $v$ in the weak formulation we find that 
	\begin{align} 
		\int_\Omega \abs{p-\psi}^2  = \int_\Omega \frac{\mu}{2} \mathbb{D} u : \mathbb{D} \overline{v} &+ \gamma \p_1 u \cdot \overline{v} +\int_{\R^{n-1}} \psi(x') \left[ \overline{v_n} - \int_0^b \diverge \overline{v} \; dx_n \right] dx'.
\end{align}
	By applying Cauchy-Schwartz and using the continuity of the trace operator, we have
	\begin{align} \label{eq: p - psi est}
		\int_\Omega \abs{p-\psi}^2  \; dx \lesssim_{\mu,\gamma,\beta} \|u\|_{\Hzerotan} \|v\|_{\Hzerotan} + [\psi]_{\dot{H}^1} \|v\|_{\Hzerotan} + [\psi]_{\dot{H}^1} \left[\int_0^b \diverge \overline{v} \; dx_n\right]_{\dot{H}^{-1}}.
	\end{align}
	Note that 
	\begin{align} \label{eq: p - psi est 2}
		\left[\int_0^b \diverge \overline{v}\; dx_n\right]_{\dot{H}^{-1}} =\int_0^b \int_{\R^{n-1}} \abs{\xi}^{-2} \abs{ 2 \pi i \xi \cdot v }^2 d\xi dx_n \lesssim \int_0^b \int_{\R^{n-1}} \abs{v(\xi,x_n)}^2 \; d\xi dx_n = \|v\|_{L^2}. 
	\end{align}
	Furthermore, since $\hat{\psi} = \sqrt{\varphi}$ has compact support, using the left hand side of the energy equivalence \eqref{eq: energy equiv} we have 
	\begin{align} \label{eq: p - psi est 3}
		\|u\|_{\Hzerotan}^2 \lesssim_\alpha \int_{\R^{n-1}} \abs{\xi}^2 \abs{ \hat{\psi}(\xi) }^2 = [\psi]_{\dot{H}^{1}}^2.
	\end{align}
	Combining the estimates \eqref{eq: p - psi est}, \eqref{eq: p - psi est 2}, \eqref{eq: p - psi est 3} give us 
	\begin{align} 
		\int_\Omega \abs{p(x)-\psi(x')}^2  \; dx \lesssim [\psi]_{\dot{H}^1} \|v\|_{\Hzerotan} \lesssim [\psi]_{\dot{H}^1} \| p - \psi\|_{L^2}.	
	\end{align}
	Then 
\begin{align} 
	\int_0^b \int_{\R^{n-1}} \abs{ (Q_
	\alpha(\xi,x_n,-\gamma) - 1) \hat{\psi}(\xi)}^2 \; d\xi dx_n = \int_\Omega \abs{p(x)-\psi(x') }^2 \; dx  \lesssim \int_{\R^{n-1}} \abs{\xi}^2 \abs{\hat{\psi}(\xi)}^2 \; d\xi.
\end{align}
Following the same argument as before, we arrive at the desired estimate. To prove the second item, we also note that the constants appearing in \eqref{eq: p - psi est}, \eqref{eq: p - psi est 2} does not depend on $\alpha$ and by the second item of Theorem~\ref{thm:energyequiv}, the constant appearing on the right hand side of \eqref{eq: p - psi est 3} can be chosen to be uniform in $\alpha$ if $\alpha \in (0,1)$. The second item then follows.

\end{proof}
We also need the asymptotics of $m_{\alpha,\beta}(\xi,\gamma)$. 

\begin{lem}\label{lem:aymptotics m}
    Let $\R \ni \alpha > 0$ and $\gamma \in \R$. The following hold.
	\begin{enumerate}
		\item For a.e. $\xi \in \R^{n-1}$, $\Re \overline{m_{\alpha,\beta}(\xi,-\gamma)}$ is strictly negative and there exists a constant $C > 0$ for which 
		\begin{align}\label{eq:realm}
			\min\{\abs{\xi}^2, \abs{\xi}^{-1} \} \le -C  \Re \overline{m_{\alpha,\beta}(\xi,-\gamma)}.
		\end{align}
		\item There exists a constant $c > 0$ such that for a.e. $\xi \in \R^{n-1}$, we have 
        \begin{align}\label{eq:masymp}
        c^{-1}\min \{ \abs{\xi}^2 , \abs{\xi}^{-1} \} \le \abs{m_{\alpha,\beta}(\xi,-\gamma)} &\le c \min \{ \abs{\xi}^2 , \abs{\xi}^{-1} \}.   
    \end{align}
		\item There exists constants $C, c > 0$ such that \eqref{eq:realm} and \eqref{eq:masymp} hold for all $\alpha \in (0,1)$. In other words, the constants can be chosen to be independent of $\alpha$ if $\alpha \in (0,1)$. 
	\end{enumerate}
	
\end{lem}
\begin{proof}
    First we prove the second item and the right hand side of \eqref{eq:masymp}. Note that by the divergence-trace compatibility condition and the energy equivalence \eqref{eq: energy equiv}, we have 
	\begin{align} 
		\int_{\R^{n-1}} \abs{\xi}^{-2} \abs{m_{\alpha,\beta}(\xi,-\gamma) \hat{\psi}(\xi) }^2 \;d\xi = \| \Tr u \cdot e_n\|_{\dot{H}^{-1}}^2 \lesssim \|u\|_{L^2}^2 \lesssim_\alpha \int_{\R^{n-1}}  \min \{\abs{\xi}^2, \abs{\xi}^{-1} \} \abs{ \hat{\psi}(\xi)}^2.
	\end{align}
	 Setting $\psi = \phi = \mathscr{F}^{-1}[ \sqrt{\varphi}] \in  \bigcap_{s \in \R} H^s (\R^{n-1} ; \C)$ as in the proof for Theorem~\ref{thm:VandQ}, we have 
	\begin{align} 
		\int_{\R^{n-1}} \abs{\xi}^{-2} \abs{m_{\alpha,\beta}(\xi,-\gamma) }^2 \varphi(\xi) \; d\xi \lesssim_\alpha	\int_{\R^{n-1}}  \min \{\abs{\xi}^2, \abs{\xi}^{-1} \} \varphi(\xi) \; d\xi.	
	\end{align}
	Repeating the same argument as in the proof for Theorem~\ref{thm:VandQ}, we can conclude that $\abs{m_{\alpha,\beta}(\xi, -\gamma)} \lesssim_\alpha \min \{ \abs{\xi}^2 , \abs{\xi}^{\frac{1}{2}} \}$. Combining this with the estimate \eqref{eq:VQestimates}, we reach the desired conclusion that $\abs{m_{\alpha,\beta}(\xi, -\gamma)} \lesssim_\alpha \min \{ \abs{\xi}^2 , \abs{\xi}^{-1}\}$. 
	
To prove the left side of \eqref{eq:masymp}, we let $(u,p) = \chi_{\beta^T, -\gamma}^{-1}(\mathscr{O}(\psi e_n, 0),0)$ be the unique weak solution to \eqref{eq:adjoint} and test $u \in \Hzerotan(\Omega; \F^n)$ in the weak formulation to find 
\begin{align}
     - \left\langle \phi, \Tr u \cdot e_n  \right\rangle_{H^{-1/2}(\Sigma_b), H^{1/2}(\Sigma_b)} = \int_\Omega \frac{\mu}{2} \abs{\mathbb{D}u}^2 - \gamma \p_1 u \cdot \overline{u} +  \frac{1}{\alpha} \int_{\Sigma_0} \beta u' \cdot \overline{u'}.
\end{align}
By taking the real part on both sides, using \eqref{eq: by parts calculation}, Lemma~\ref{lem: korn}, the fact that $\alpha > 0$, $\beta$ satisfies \eqref{eq: beta coercive C}, and the anti-dual representation of Sobolev spaces gives us  gives us 
\begin{multline}
        \int_{\R^{n-1}}  \min \{ \abs{\xi}^2 , \abs{\xi}^{-1}\} \varphi(\xi) \; d\xi \lesssim \|u\|_{\Hzerotan}^2 \lesssim_\alpha -\Re \langle \phi, u_n \rvert_{\Sigma_b} \rangle_{H^{-1/2}(\Sigma_b), H^{1/2}(\Sigma_b)} = -\Re \int_{\R^{n-1}} \widehat{\phi}(\xi) \overline{\widehat{\Tr u \cdot e_n}(\xi)} \; d\xi \\
         = - \Re \int_{\R^{n-1}} \overline {m_{\alpha,\beta}(\xi,-\gamma)} \abs{\hat{\phi}(\xi)}^2 \; d\xi = -\Re \int_{\R^{n-1}} \overline {m_{\alpha,\beta}(\xi,-\gamma)} \varphi(\xi) \; d\xi.
\end{multline}
Thus we have $\min \{ \abs{\xi}^2 , \abs{\xi}^{-1}\} \lesssim_\alpha - \Re \overline{m_{\alpha,\beta}(\xi,-\gamma)} \le \abs{m_{\alpha,\beta}(\xi,-\gamma)}$ for a.e. $\xi \in \R^{n-1}$. This proves the first item and also the left side of the inequality in the second. 

To prove the third item, we note that throughout the proof for the first and the second items, by the second item of Theorem~\ref{thm:energyequiv} and the fact that $\alpha^{-1} > 0$ if $\alpha \in (0,1)$, the constants in the estimates above can be chosen to be independent of $\alpha$ if $\alpha \in (0,1)$, therefore the third item follows.
\end{proof}

We conclude this subsection by recording the properties of an auxiliary function defined in terms of $m_{\alpha,\beta}$. 

\begin{lem}\label{lem:rho}
	Suppose $\R \ni \alpha >0$ and $\gamma \in \R \setminus \{0\}$, and define 
	\begin{align}\label{eq:rho}
		\rho_{\alpha,\beta,\gamma}(\xi) = 2\pi i \gamma \xi_1 + (1+4\pi^2  \abs{\xi}^2\sigma) \overline{m_{\alpha,\beta}(\xi,-\gamma)}.
   \end{align}
   Then the following hold.
   \begin{enumerate}
	    \item $\rho_{\alpha,\beta,\gamma}(\xi) = 0$ if and only if $\xi = 0$, and $\overline{\rho_{\alpha,\beta,\gamma}(\xi)} = \rho_{\alpha,\beta,\gamma}(-\xi)$ for all $\xi \in \R^{n-1}$. 
		\item For $\sigma > 0$, there exists a constant $C = C(\alpha, n,\gamma,\sigma, b)> 0$  such that for all $\xi \in \R^{n-1}$, we have 
		\begin{align}\label{eq: rho est 1}
		  C^{-1} \abs{ \rho_{\alpha,\beta,\gamma}(\xi)}^2 \le (\xi_1^2 + \abs{\xi}^4) \mathbbm{1}_{B(0,1)}(\xi) + (1+\abs{\xi}^2)  \mathbbm{1}_{B(0,1)^c}(\xi) \le C \abs{\rho_{\alpha,\beta,\gamma}(\xi)}^2. 
		\end{align}
		\item For $\sigma = 0$ and $n=2$, there exists a constant $C = C(\alpha,\gamma, b)> 0$ such that for all $\xi \in \R^{n-1}$, we have 
		\begin{align}\label{eq: rho est 2}
			C^{-1}\abs{ \rho_{\alpha,\beta,\gamma}(\xi)}^2 \le \abs{\xi}^2 \mathbbm{1}_{B(0,1)}(\xi) + (1 + \abs{\xi}^2) \mathbbm{1}_{B(0,1)^c}(\xi) \le C \abs{\rho_{\alpha,\beta,\gamma}(\xi)}^2. 
		  \end{align}
		  \item Furthermore, there exists a constant $C = C(n,\gamma,\sigma, b)> 0$ such that \eqref{eq: rho est 1} holds for all $\alpha \in (0,1)$ and a constant $C = C(\gamma,b)> 0$ such that \eqref{eq: rho est 2} holds for all $\alpha \in (0,1)$. In other words, the constants in \eqref{eq: rho est 1} and \eqref{eq: rho est 2} can be chosen to be independent of $\alpha$. 
   \end{enumerate}
\end{lem}

\begin{proof}
	To prove the first item, we note that the identity $\overline{\rho_{\alpha,\beta,\gamma}(\xi)} = \rho_{\alpha,\beta,\gamma}(-\xi)$ follows from Theorem~\ref{thm: V Q symbols}, therefore $\rho_{\alpha,\beta,\gamma}(0) = 0$. Furthermore, $\Re \rho_{\alpha,\beta,\gamma}(\xi) =  (1+4\pi^2 \abs{\xi}^2\sigma)  \Re \overline{m_{\alpha,\beta}(\xi,-\gamma)} < 0$ for $\xi \neq 0$ by the first item of Lemma~\ref{lem:aymptotics m}. Thus $\rho_{\alpha,\beta,\gamma}(\xi) = 0$ if and only if $\xi = 0$. This proves the first item. 
	
	Next we prove the second item, and we first prove the left hand side of \eqref{eq: rho est 1}. Recall that by Lemma~\ref{lem:aymptotics m}, $\abs{m_{\alpha,\beta}(\xi, -\gamma)}$ satisfies $
	 \abs{m_{\alpha,\beta}(\xi,-\gamma)} \asymp \min \{ \abs{\xi}^2 , \abs{\xi}^{-1} \}$.
	This implies that  
	\begin{align}
		 \abs{\rho_{\alpha,\beta,\gamma}(\xi)} \lesssim \abs{\xi_1} + \min \{ \abs{\xi}^2 , \abs{\xi}^{-1} \}  + \min \{ \abs{\xi}^4 , \abs{\xi} \}.
	\end{align}
Then it immediately follows that 
	\begin{multline}\label{eq:rho1}
		\abs{\rho_{\alpha,\beta,\gamma}(\xi)}^2 \lesssim  (\abs{\xi_1}^2 + \abs{\xi}^2 + \abs{\xi}^4) \mathbbm{1}_{B(0,1)}(\xi)  + (\abs{\xi_1}^2 + \abs{\xi}^2 + \abs{\xi}^4) \mathbbm{1}_{B(0,1)^c}(\xi) \\ \lesssim (\abs{\xi_1}^2 + \abs{\xi}^2) \mathbbm{1}_{B(0,1)}(\xi)+ \abs{\xi}^2\mathbbm{1}_{B(0,1)^c}(\xi).
	\end{multline}
	Next we prove the right hand side of \eqref{eq: rho est 1}. We first note that since $2\pi i \gamma \xi_1$ is purely imaginary and $1+4\pi^2 \abs{\xi}^2 \sigma$ is real, we have 
	\begin{align}\label{eq:realrho}
		 \Re \rho_{\alpha,\beta,\gamma}(\xi) = (1+4\pi^2 \abs{\xi}^2 \sigma) \Re m_{\alpha,\beta}(\xi,-\gamma), \quad \Im \rho_{\alpha,\beta,\gamma}(\xi) = 2 \pi \gamma \xi_1 + (1+4\pi^2 \abs{\xi}^2 \sigma) \Im m_{\alpha,\beta}(\xi,-\gamma).
	\end{align}
	Next we call that by \eqref{eq:realm}, we have $\abs{\xi}^2 \lesssim - \Re \overline{m_{\alpha,\beta}(\xi,-\gamma)}$ for a.e. $\abs{\xi} \le 1$ and $\abs{\xi}^{-1} \lesssim - \Re \overline{m_{\alpha,\beta}(\xi,-\gamma)}$ for a.e. $\abs{\xi} \ge 1$; by \eqref{eq:masymp}, $\abs{m_\alpha(\xi,-\gamma)} \asymp \abs{\xi}^2$ for a.e. $\abs{\xi} \le 1$ and $\abs{m_{\alpha,\beta}(\xi,-\gamma)} \asymp \abs{\xi}^{-1}$ for a.e. $\abs{\xi} \ge 1$. Then for a.e. $\abs{\xi} \le 1$, since $2\pi i \gamma \xi_1$ is purely imaginary and $m_{\alpha,\beta}(\xi,-\gamma) \overline{m_{\alpha,\beta}(\xi,-\gamma)}$ is real, we have  
	\begin{multline}\label{eq:rho2}
		 \abs{\xi_1} \abs{\xi}^2 \lesssim \abs{2 \pi \gamma \xi_1 \Re m_{\alpha,\beta}(\xi, -\gamma)} = \abs{\Im[ 2 \pi i \gamma \xi_1 m_{\alpha,\beta}(\xi,-\gamma) + (1+4\pi^2 \abs{\xi}^2 \sigma)m_{\alpha,\beta}(\xi,-\gamma) \overline{m_\alpha(\xi,-\gamma)}]}\\
		 \le \abs{\rho_{\alpha,_{\alpha,\beta}\beta,\gamma}(\xi) m_{\alpha,\beta}(\xi,-\gamma)} \lesssim \abs{\rho_{\alpha,\beta,\gamma}(\xi)} \abs{\xi}^2 \implies \abs{\xi_1} \lesssim \abs{\rho_{\alpha,\beta,\gamma}(\xi)},
	\end{multline}
and also by \eqref{eq:realrho},
\begin{align}\label{eq:rho3}
	 \abs{\xi}^2 \lesssim  \abs{(1 + 4\pi^2 \abs{\xi}^2 \sigma) \Re \overline{m_{\alpha,\beta}(\xi,-\gamma)}} = \abs{\Re \rho_{\alpha,\beta,\gamma}(\xi)} \lesssim \abs{\rho_{\alpha,\beta,\gamma}(\xi)}.
\end{align}
For a.e. $\abs{\xi} \ge 1$, we have 
\begin{align}\label{eq:rho4}
	 \abs{\xi} \lesssim \abs{\xi}^2 \abs{\Re m_{\alpha,\beta}(\xi,-\gamma)} \lesssim \abs{\Re \rho_{\alpha,\beta,\gamma}(\xi)} \lesssim \abs{\rho_{\alpha,\beta,\gamma}(\xi)}.
\end{align}
\eqref{eq: rho est 1} then follows by combining \eqref{eq:rho1}, \eqref{eq:rho2}, \eqref{eq:rho3}, and \eqref{eq:rho4}. We also note that by the third item of Lemma~\ref{lem:aymptotics m}, the constants appearing in the estimates above can be chosen to be independent of $\alpha$ if $\alpha \in (0,1)$. This proves the second item.

The third item follows from a similar set of arguments. 
\end{proof}

\subsection{Data compatibility and the associated isomorphism}

Now we are ready to discuss compatibility conditions associated to the solvability of \eqref{eq: linear w/o eta}. To do so we first define some spaces associated to the data. 

\begin{defn}\label{defn:dataY}
    Let $s \ge 0$.
	\begin{enumerate}
		 \item We define the Hilbert space 
		 \begin{multline}
			  \mathcal{Y}^s = \big\{ (f,g,h,k,l) \in H^s(\Omega ;\R^{n}) \times H^{s+1}(\Omega; \R )  \times  H^{s+3/2}(\Sigma_b ; \R)  \\ \times H^{s+1/2}(\Sigma_b ; \R^{n})  \times  H^{s+1/2}(\Sigma_0 ; \R^{n-1}) \mid \norm{(f,g,h,k,l)}_{\mathcal{Y}^s} < \infty \big \},
		 \end{multline}
		 where we equip $\mathcal{Y}^s$ with the norm defined via 
		 \begin{align}
			  \norm{(f,g,h,k,l)}_{\mathcal{Y}^s}^2 = \norm{f}_{H^s}^2 + \norm{g}_{H^{s+1}}^2 + \norm{h}_{H^{s+3/2}}^2  + \norm{k}_{H^{s+1/2}}^2 + \norm{l}_{H^{s+1/2}}^2  + \left[ h - \int_0^b g(\cdot, x_n) \; dx_n\right]_{\dot{H}^{-1}}^2.   
		 \end{align}
		 \item We define the Hilbert space 
		 \begin{multline}
			  \mathcal{Z}^s = \big\{ (f,g,h,k) \in H^s(\Omega ;\R^{n}) \times H^{s+1}(\Omega; \R )  \times  H^{s+3/2}(\Sigma_b ; \R) \times H^{s+1/2}(\Sigma_b ; \R^{n}) \mid \norm{(f,g,h,k)}_{\mathcal{Y}^s} < \infty \big \},
		 \end{multline}
		 where we equip $\mathcal{Z}^s$ with the norm defined via 
		 \begin{align}
			  \norm{(f,g,h,k)}_{\mathcal{Z}^s}^2 = \norm{f}_{H^s}^2 + \norm{g}_{H^{s+1}}^2 + \norm{h}_{H^{s+3/2}}^2  + \norm{k}_{H^{s+1/2}}^2 + \norm{l}_{H^{s+1/2}}^2  + \left[ h - \int_0^b g(\cdot, x_n) \; dx_n\right]_{\dot{H}^{-1}}^2.   
		 \end{align}
	\end{enumerate}
	
    \end{defn}

 Next we define the bilinear maps associated to the data spaces $\mathcal{Y}^s$ and $\mathcal{Z}^s$. 

\begin{defn}
Let $\R \ni \alpha > 0$, $\beta \in \R^{n \times n}$ be positive definite,  $\gamma \in \R$, and $\R \ni s \ge 0$. We define the bilinear map 
	 \begin{align}
		  \mathscr{B}_{\alpha,\beta,\gamma} : [H^s(\Omega ;\R^{n}) \times H^{s+1}(\Omega; \R )  \times H^{s+3/2}(\Sigma_b ; \R) \times H^{s+1/2}(\Sigma_b ; \R^n)   \times H^{s+1/2}(\Sigma_0 ; \R^{n-1}) ] \times [H^{s+1/2}(\Sigma_b ; \R) ] \to \R 
	 \end{align}
	 via 
	 \begin{align}\label{eq: bilinear B}
		  \mathscr{B}_{\alpha,\beta,\gamma} ( (f,g,h,k,l), \psi) = \int_{\Omega} (f \cdot v - gq)- \int_{\Sigma_b} (k \cdot v -  h\psi) + \frac{1}{\alpha} \int_{\Sigma_0} l \cdot v',
	 \end{align}
	 where $(v,q) = \Phi^{-1}_{\alpha, -\gamma, \beta^T}(0,0,\psi e_n,0)\in \Hzerostan{s+2}(\Omega;\R^n) \times H^{s+1}(\Omega ; \R)$ is the unique solution to the normal stress problem \eqref{eq:adjoint} guaranteed by Theorem~\ref{thm:classicalexistence}. Since $\Phi_{\alpha, \beta^T, -\gamma}^{-1}$ is an isomorphism, we have that $\mathscr{B}_{\alpha,\beta,\gamma}$ is continuous. We define the left kernel of $\mathscr{B}_{\alpha,\beta,\gamma}$ as
	 \begin{multline}
		 \overleftarrow{\ker} \mathscr{B}_{\alpha,\beta,\gamma} = \{ (f,g,h,k,l) \in H^s(\Omega ;\R^{n})  \times H^{s+1}(\Omega; \R )    \times  H^{s+3/2}(\Sigma_b ; \R)   \\  \times H^{s+1/2}(\Sigma_b ; \R^n) \times H^{s+1/2}(\Sigma_0 ; \R^{n-1})      \mid \mathscr{B}_{\alpha,\beta,\gamma} ( (f,g,h,k,l), \psi) =0 \; \forall \psi \in H^{s+1/2}(\Sigma_b ; \R) \}.
	 \end{multline}
	 Since $\overleftarrow{\ker} \mathscr{B}_{\alpha,\beta,\gamma}$ is a closed subspace of $H^s(\Omega ;\R^{n}) \times H^{s+1}(\Omega; \R ) \times H^{s+3/2}(\Sigma_b ; \R) \times H^{s+1/2}(\Sigma_b ; \R^n) \times H^{s+1/2}(\Sigma_0 ; \R^{n-1}) $ and $\mathcal{Y}^s \cap \overleftarrow{\ker} \mathscr{B}_{\alpha,\beta,\gamma}$ is a closed subspace of $\mathcal{Y}^s$ inheriting the topology of $\mathcal{Y}^s$, we may regard $\mathcal{Y}^s \cap \overleftarrow{\ker} \mathscr{B}_{\alpha,\beta,\gamma}$ as a Hilbert space equipped with the inner product coming from $\mathcal{Y}^s$. 
\end{defn}

Now we are ready to record the isomorphism associated to the overdetermined problem \eqref{eq: overdetermined}.

\begin{thm} \label{thm: overdeterminediso}
Let $\R \ni \alpha > 0$, $\beta \in \R^{n \times n}$ be positive definite, $\gamma \in \R$ and $\R \ni s \ge 0$. Consider the bounded linear map $\Psi_{\alpha,\beta,\gamma} : \Hzerostan{s+2}(\Omega ; \R^n) \times H^{s+1}(\Omega;\R) \to \mathcal{Y}^s \cap \overleftarrow{\ker} \mathscr{B}_{\alpha,\beta,\gamma}$ defined via $\Psi_{\alpha,\beta,\gamma}(u,p) = \mathcal{P}(\Phi_{\alpha,\beta,\gamma}(u,p), u_n)$, 	 where $\Phi_{\alpha,\beta,\gamma}$ is defined in Theorem \ref{thm:classicalexistence} and $\mathcal{P}:  H^s(\Omega ;\R^{n}) \times H^{s+1}(\Omega; \R )   \times H^{s+1/2}(\Sigma_b ; \R^{n})  \times  H^{s+1/2}(\Sigma_0 ; \R^{n-1}) \times  H^{s+3/2}(\Sigma_b ; \R)  \to \mathcal{Y}^s$ is a permutation map defined via $\mathcal{P}(f,g,k,l,h) = (f,g,h,k,l)$.  Then $\Psi_{\alpha,\beta,\gamma}$ is an isomorphism. 
\end{thm}
\begin{proof}
To prove the first item, we first show that the map $\Psi_{\alpha,\beta,\gamma}$ is well-defined. Let $(u,p) \in \Hzerostan{s+2}(\Omega ; \R^n) \times H^{s+1}(\Omega;\R)$. By Theorem~\ref{thm:classicalexistence}, Lemma~\ref{lem:divtrace} and trace theory, we have $\Psi_{\alpha,\beta,\gamma}(u,p) \in \mathcal{Y}^s$. To show that $\Psi_{\alpha,\beta,\gamma}(u,p) \in \overleftarrow{\ker} \mathscr{B}_{\alpha,\beta,\gamma}$, we let $(f,g,k,l,h) = \Psi_{\alpha,\beta,\gamma}(u,p)$, and for any $\psi \in H^{s+1/2}(\Sigma_b;\R)$ we let $(v,q) = \Phi^{-1}_{\alpha, \beta^T, -\gamma}(0,0,\psi e_n, 0)$. Then  
\begin{multline}
	\mathscr{B}_{\alpha,\beta,\gamma} ( (f,g,h,k,l), \psi) = \int_{\Omega} (f \cdot v - gq)- \int_{\Sigma_b} (k \cdot v -  h \psi) + \frac{1}{\alpha} \int_{\Sigma_0} l \cdot v' \\ = \int_\Omega ( \diverge S(p,u) -\gamma \p_1 u  ) \cdot v - (\diverge u)q - \int_{\Sigma_b}S(p,u) e_n \cdot v -  u \cdot \psi e_n + \frac{1}{\alpha} \int_{\Sigma_0} [\alpha S(p,u) e_n + \beta u]' \cdot v' \\
	= \int_{\Omega} u \cdot (\diverge S(v,q) + \gamma \p_1 v) + p \diverge v + \frac{1}{\alpha} \int_{\Sigma_0} u' \cdot [\alpha S(q,v) e_n + \beta^T v]' = 0.
\end{multline}
This shows that the map $\Psi_{\alpha,\beta,\gamma}(u,p)$ is well-defined, and it is clearly linear and bounded. 

The injectivity of $\Psi_{\alpha,\beta,\gamma}$ follows from Theorem~\ref{thm:classicalexistence}. To prove that $\Psi_{\alpha,\beta,\gamma}$ is surjective, for any $(f,g,h,k,l) \in \mathcal{Y}^s \cap \overleftarrow{\ker} \mathscr{B}_{\alpha,\beta,\gamma}$ we let $(u,p) = \Phi_{\alpha, \beta, \gamma}^{-1}(f,g,k,l)$, and for any $\psi \in H^{s+1/2}(\Sigma_b;\R)$ we let $(v,q) = \Phi^{-1}_{\alpha, \beta^T, -\gamma}(0,0,\psi e_n,0)$. Since $(f,g,h,k,l) \in  \overleftarrow{\ker} \mathscr{B}_{\alpha,\beta,\gamma}$, we then have 
\begin{align}
	- \int_{\Sigma_b} u_n \psi = \int_\Omega (f \cdot v - g q) - \int_{\Sigma_b} k \cdot v + \frac{1}{\alpha} \int_{\Sigma_0} l \cdot v' = - \int_{\Sigma_0} h \psi,
\end{align}
and so  $u_n = h$ on $\Sigma_b$. This shows that $\Psi_{\alpha, \beta, \gamma}$ is surjective, and the desired conclusion follows. 

To prove the second item we follow a similar set of arguments as above, where we use the isomorphism $\Theta_{\alpha,\beta,\gamma}$ in place of $\Phi_{\alpha, \beta, \gamma}$, Corollary~\ref{thm: alpha classicalexistence} in place of Theorem~\ref{thm:classicalexistence},the bilinear map $\mathscr{B}_\gamma$ in place of  $\mathscr{B}_{\alpha,\beta,\gamma}$ and the Hilbert space $\mathcal{Z}^s$ in place of $\mathcal{Y}^s$.  The fact that the operator norm of $\Psi_{\alpha, \beta, \gamma}$ is independent of $\alpha$ for $\alpha \in (0,1)$ follows from Proposition~\ref{prop:ind of alpha}.
\end{proof}

Next we would like to introduce a quantitative way of measuring how close a data tuple $(f,g,h,k,l)$ is to being compatible. To do so we introduce the linear map $\Lambda_{\alpha,\beta,\gamma}: H^s(\Omega ;\R^{n}) \times H^{s+1}(\Omega; \R ) \times H^{s+3/2}(\Sigma_b ; \R) \times H^{s+1/2}(\Sigma_b ; \R^n) \times H^{s+1/2}(\Sigma_0 ; \R^{n-1}) \to L^2(\Sigma_b; \R)$ induced by the bilinear map $\mathscr{B}_{\alpha,\beta,\gamma} $. The induced linear map $\Lambda_{\alpha,\beta,\gamma}$ is defined via 
\begin{align}\label{eq: Lambda gamma}
	 \left\langle \Lambda_{\alpha,\beta,\gamma}(f,g,h,k,l) , \psi \right\rangle_{L^2} = \mathscr{B}_{\alpha,\beta,\gamma}((f,g,h,k,l), \psi),
\end{align} 
where we use the canonical injection $i : H^{s+1/2}(\Sigma_b ; \F) \hookrightarrow L^2(\Sigma_  b; \F)$ to identify $\psi$ with an element of $L^2$. First we show that $\Lambda_{\alpha,\beta,\gamma}(f,g,h,k,l)$ commutes with tangential multipliers defined in Definition~\ref{defn: tangential m}.

\begin{prop}\label{prop:lambda commutes}
	Suppose $\omega \in L^\infty(\R^{n-1}; \C)$ and consider the tangential multiplier $M_\omega$ defined in Definition~\ref{defn: tangential m}. Then 
	\begin{align}
		 M_\omega \Lambda_{\alpha,\beta,\gamma}(f,g,h,k,l) = \Lambda_{\alpha,\beta,\gamma}(M_\omega f,M_\omega g,M_\omega h,M_\omega k,M_\omega l).
	\end{align}
\end{prop}
\begin{proof}
For a given $\psi \in H^{s+1/2}(\Sigma_b ; \R)$ we define $(v,q) = \Phi^{-1}_{\alpha, \beta^T, -\gamma}(0,0,\psi e_n, 0)$. Then by Theorem~\ref{thm: chi commutes with M_omega},  we have 
\begin{multline}
	\left\langle  M_\omega \Lambda_{\alpha,\beta,\gamma}(f,g,h,k,l) , \psi \right\rangle_{L^2} = \left\langle \Lambda_{\alpha,\beta,\gamma}(f,g,h,k,l) ,  M_\omega  \psi \right\rangle_{L^2}   = \mathscr{B}_{\alpha,\beta,\gamma}((f,g,h,k,l),  M_\omega  \psi) \\
	=  \int_{\Omega} (f \cdot M_\omega v - g M_\omega q)- \int_{\Sigma_b} (k \cdot M_\omega v -  h M_\omega \psi)  + \frac{1}{\alpha} \int_{\Sigma_0} l \cdot M_\omega v' 
	=  \int_{\Omega} ( M_\omega f \cdot  v - M_\omega g  q)  \\ - \int_{\Sigma_b} ( M_\omega k \cdot  v -  M_\omega h  \psi) + \frac{1}{\alpha} \int_{\Sigma_0} M_\omega l \cdot v'  = \langle \Lambda_{\alpha,\beta,\gamma}(M_\omega f,M_\omega g,M_\omega k,M_\omega l, M_\omega h), \psi \rangle.
\end{multline}
Since this is true for all $\psi \in H^{s+1/2}(\Sigma_b ; \R)$, the desired conclusion follows.

\end{proof}
Next we prove the main theorem of this section, which describes the low-frequency behavior of the images of $\mathcal{Y}^s$ and $\mathcal{Z}^s$ under $\Lambda_{\alpha,\beta,\gamma}$. 

\begin{thm}\label{thm: regularity of Lambda}
Suppose $\R \ni \alpha > 0,  \R \ni s \ge 0$. The following hold.
\begin{enumerate}
	 \item If $(f,g,h,k,l) \in \mathcal{Y}^s$, then $\Lambda_{\alpha,\beta,\gamma}(f,g,h,k,l) \in \dot{H}^{-1}(\Sigma_b ; \R) \cap H^{s+3/2}(\Sigma_b;\R)$ and there exists a constant $c > 0$ for which 
	 \begin{align}\label{eq: lambda est}
		  \norm{\Lambda_{\alpha,\beta,\gamma}(f,g,h,k,l)}_{\dot{H}^{-1} \cap H^{s+3/2}} \le c \norm{(f,g,h,k,l)}_{\mathcal{Y}^s}. 
	 \end{align}
	 \item If $(f,g,h,k) \in \mathcal{Z}^s$, then $\Lambda_{\alpha,\beta,\gamma}(f,g,h,k,0) \in \dot{H}^{-1}(\Sigma_b ; \R) \cap H^{s+3/2}(\Sigma_b;\R)$ and there exists a constant $c > 0$ for which 
	 \begin{align}\label{eq: lambda est l = 0}
		  \norm{\Lambda_{\alpha,\beta,\gamma}(f,g,h,k,0)}_{\dot{H}^{-1} \cap H^{s+3/2}} \le c \norm{(f,g,h,k)}_{\mathcal{Z}^s}. 
	 \end{align}
	 \item Furthermore, there exists a constant $c > 0$ for which the estimate \eqref{eq: lambda est} holds for all $\alpha \in (0,1)$. In other word, the constant $c > 0$ can be chosen to be independent of $\alpha$ if $\alpha \in (0,1)$.
\end{enumerate}

\end{thm}

\begin{proof}
	We first note that the second item follows immediately from the first item. To prove the first item, we first note that by Proposition~\ref{prop:lambda commutes},
	\begin{multline}\label{eq: Lambda gamma split}
		\norm{\Lambda_{\alpha,\beta,\gamma}(f,g,h,k,l)}_{\dot{H}^{-1} \cap H^{s+3/2}} \lesssim \norm{M_{\mathbbm{1}_{B(0,1)}} \Lambda_{\alpha,\beta,\gamma}(f,g,h,k,l)}_{\dot{H}^{-1}} + \norm{M_{\mathbbm{1}_{B(0,1)^c}} \Lambda_{\alpha,\beta,\gamma}(f,g,h,k,l)}_{H^{s+3/2}} \\
		 = \norm{M_{\mathbbm{1}_{B(0,1)}} \Lambda_{\alpha,\beta,\gamma}(f,g,h,k,l)}_{\dot{H}^{-1}} + \norm{\Lambda_{\alpha,\beta,\gamma}(M_{\mathbbm{1}_{B(0,1)^c}}  f, M_{\mathbbm{1}_{B(0,1)^c}}  g, M_{\mathbbm{1}_{B(0,1)^c}}  k, M_{\mathbbm{1}_{B(0,1)^c}} l, M_{\mathbbm{1}_{B(0,1)^c}} h)}_{H^{s+3/2}} \\
		 \lesssim_\alpha \norm{M_{\mathbbm{1}_{B(0,1)}} \Lambda_{\alpha,\beta,\gamma}(f,g,h,k,l)}_{\dot{H}^{-1}} + \norm{M_{\mathbbm{1}_{B(0,1)^c}} f}_{H^s} + \norm{M_{\mathbbm{1}_{B(0,1)^c}} g}_{H^{s+1}}   \\ + \norm{M_{\mathbbm{1}_{B(0,1)^c}}k}_{H^{s+1/2}} + \norm{M_{\mathbbm{1}_{B(0,1)^c}} l}_{H^{s+1/2}} + \norm{M_{\mathbbm{1}_{B(0,1)^c}} h}_{H^{s+3/2}} \\
		 \lesssim_\alpha \norm{M_{\mathbbm{1}_{B(0,1)}} \Lambda_{\alpha,\beta,\gamma}(f,g,h,k,l)}_{\dot{H}^{-1}} + \norm{f}_{H^s} + \norm{g}_{H^{s+1}} + \norm{k}_{H^{s+1/2}} + \norm{l}_{H^{s+1/2}}  + \norm{h}_{H^{s+3/2}}. 
	\end{multline}
	To arrive at the desired estimate it then suffices to control $\norm{M_{\mathbbm{1}_{B(0,1)}} \Lambda_{\alpha,\beta,\gamma}(f,g,h,k,l)}_{\dot{H}^{-1}}$. We note that for any $\psi \in L^2(\Sigma; \R)$, we may let $(v,q) = \Phi_{\alpha, \beta^T, -\gamma}^{-1}(0,0,\psi e_n,0)$ and compute
   \begin{multline}
   \int_{B(0,1)} \mathscr{F}[\Lambda_{\alpha,\beta,\gamma}(f,g,h,k,l)](\xi) \cdot \mathscr{F}[\psi](\xi) \; d\xi 
   = \langle M_{\mathbbm{1}_{B(0,1)}} \Lambda_{\alpha,\beta,\gamma}(f,g,h,k,l) , \psi \rangle_{L^2} \\
   = \langle  \Lambda_{\alpha,\beta,\gamma}(f,g,h,k,l) , M_{\mathbbm{1}_{B(0,1)}} \psi \rangle_{L^2} 
   = \mathscr{B}_{\alpha,\beta,\gamma}((f,g,h,k,l), M_{\mathbbm{1}_{B(0,1)}} \psi) \\
   =  \int_{\Omega} (f \cdot M_{\mathbbm{1}_{B(0,1)}} v - g M_{\mathbbm{1}_{B(0,1)}}  q)- \int_{\Sigma_b} (k \cdot M_{\mathbbm{1}_{B(0,1)}}  v -  h M_{\mathbbm{1}_{B(0,1)}} \psi) + \frac{1}{\alpha} \int_{\Sigma_0} l \cdot M_{\mathbbm{1}_{B(0,1)}} v' \\
   =  \int_{\Omega} f \cdot M_{\mathbbm{1}_{B(0,1)}} v - g M_{\mathbbm{1}_{B(0,1)}} (  q - \psi)- \int_{\Sigma_b} k \cdot M_{\mathbbm{1}_{B(0,1)}}  v  + \int_{\Sigma_b}M_{\mathbbm{1}_{B(0,1)}} \psi \left(h - \int_0^b g \right) +  \frac{1}{\alpha}  \int_{\Sigma_0} l \cdot M_{\mathbbm{1}_{B(0,1)}} v'.
   \end{multline}
   Therefore,
   \begin{multline}
	\abs{ \int_{B(0,1)} \mathscr{F}[\Lambda_{\alpha,\beta,\gamma}(f,g,h,k,l)](\xi) \cdot \mathscr{F}[\psi](\xi) \; d\xi } \lesssim_\alpha (\norm{f}_{H^s} + \norm{k}_{H^{s+1/2}} + \norm{l}_{H^{s+1/2}}  ) \norm{M_{\mathbbm{1}_{B(0,1)} }v}_{L^2}  \\ + \norm{g}_{L^2} \norm{ M_{\mathbbm{1}_{B(0,1)}} (  q - \psi)}_{L^2}  + \norm{M_{\mathbbm{1}_{B(0,1)}} \psi}_{\dot{H}^1} \left[ h - \int_0^b g\right]_{\dot{H}^{-1}}, 
\end{multline}
where we have used the second estimate in \eqref{eq: V estimates} on $V_{\alpha,\beta}(\xi,0,-\gamma)$ to handle to integral involving $l$. 

By Theorem~\ref{thm: chi commutes with M_omega}, we have $(M_{\mathbbm{1}_{B(0,1)} }v,M_{\mathbbm{1}_{B(0,1)} }q) = \Phi_{\alpha, \beta^T, -\gamma}^{-1}(0,0,M_{\mathbbm{1}_{B(0,1)} } \psi e_n,0)$. Then by \eqref{eq: Vpsi est} we have the bound $\norm{M_{\mathbbm{1}_{B(0,1)} }v}_{L^2} \lesssim_\alpha \norm{M_{\mathbbm{1}_{B(0,1)} }\psi}_{\dot{H}^1}$, and by Plancherel's theorem and the second estimate on $Q_{\alpha,\beta}$ in \eqref{eq: Q estimates}, we have $\norm{ M_{\mathbbm{1}_{B(0,1)}} (  q - \psi)}_{L^2} \lesssim_\alpha \norm{M_{\mathbbm{1}_{B(0,1)} }\psi}_{\dot{H}^1}$. By combining the previous estimates and the divergence-trace estimate \eqref{eq: divtraceestimate} we then have 
\begin{align}
	\abs{ \int_{B(0,1)} \mathscr{F}[\Lambda_{\alpha,\beta,\gamma}(f,g,h,k,l)](\xi) \cdot \mathscr{F}[\psi](\xi) \; d\xi }  \lesssim_\alpha \norm{(f,g,h,k,l)}_{\mathcal{Y}^s} \norm{M_{\mathbbm{1}_{B(0,1)} }\psi}_{\dot{H}^1}. 
\end{align}
Thus by duality, 
   \begin{multline}
	\norm{M_{\mathbbm{1}_{B(0,1)}} \Lambda_{\alpha,\beta,\gamma}(f,g,h,k,l)}_{\dot{H}^{-1}} \\ = \sup \bigg\{ \abs{ \int_{B(0,1)}  \mathscr{F}[\Lambda_{\alpha,\beta,\gamma}(f,g,h,k,l)](\xi) \cdot \mathscr{F}[\psi](\xi) \; d\xi } \mid [M_{\mathbbm{1}_{B(0,1)}} \psi]_{\dot{H}^1} \le 1 \bigg\} \lesssim_\alpha \norm{(f,g,h,k,l)}_{\mathcal{Y}^s}.
   \end{multline}
   To prove the third item, we note that if $l = 0$ then $\alpha$ does not appear in \eqref{eq: Lambda gamma split} and by the second items of Theorem~\ref{thm:energyequiv} and Theorem~\ref{thm:VandQ}, the constants in the estimates above can be chosen to be uniform in $\alpha$. The third item then follows.
\end{proof}

\section{Linear analysis with $\eta$} \label{sec: analysis with eta}
In this section we would like to establish the $\R$-solvability of the $\gamma-$Stokes system with gravity capillary boundary conditions 
\begin{align} \label{eq:gammacapillary}
	\begin{cases}
		\diverge S(p,u) - \gamma \p_1 u + (\nabla' \eta, 0)= f & \text{in} \; \Omega \\
	\diverge u = g, & \text{in} \; \Omega \\
	u_n + \gamma \p_1 \eta= h, & \text{on} \; \Sigma_b \\
	S(p,u)e_n + \sigma \Delta' \eta e_n = k, & \text{on} \; \Sigma_b \\
	(\alpha S(p,u)e_n)' + \beta u' = l , & \text{on} \; \Sigma_{0}\\
	u_n= 0, & \text{on} \; \Sigma_{0}.
\end{cases}
\end{align}

\subsection{Preliminaries}
First, we introduce the container space for the free surface function $\eta$. 

\begin{defn} 
Let $0 \le s \in \R$.  We define the specialized anisotropic Sobolev space  $X^s(\R^{d})$ to consist of $f\in\mathscr{S}'(\R^{d};\R)$ such that $\hat{f}\in L^1_{\loc}(\R^{d};\C)$ and 
\begin{equation}
 \norm{f}_{X^s}^2 := \int_{B(0,1)} \frac{\xi_1^2 + \abs{\xi}^4}{\abs{\xi}^2} \abs{\hat{f}(\xi)}^2 d\xi + \int_{B(0,1)^c} (1+\abs{\xi}^2)^{s}  \abs{\hat{f}(\xi)}^2 d\xi < \infty.
\end{equation}
\end{defn}

The following proposition summarizes the important properties of this space.

\begin{thm}\label{thm: X^s}
	Suppose $\R \ni s \ge 0$ and $d \ge 1$. The following hold. 
	\begin{enumerate}
		\item $X^s(\R^d)$ is a separable Hilbert space, and if $t \in \R$ and $s < t$, then we have the continuous inclusion $X^t(\R^d) \hookrightarrow X^s(\R^d)$. 
		
		\item If $d = 1$, we have $H^s(\R^d) = X^s(\R^d)$ and $\norm{\cdot}_{H^s}$ and $\norm{\cdot}_{X^s}$ are equivalent norms.  For $d \ge 2$, we have the continuous inclusion $H^s(\R^d) \hookrightarrow X^s(\R^d)$.
		
		\item If $s \ge 1$, then $\norm{\nabla f}_{H^{s-1}} \lesssim \norm{f}_{X^s}$ for $f \in X^s(\R^d)$. In particular, the map $\nabla: X^s(\R^d) \to H^{s-1}(\R^d; \R^d)$ is continuous. 

		\item  For every $f \in X^s(\R^{d})$ and $t > 0$, we can write $f=f_{l,t}+f_{h,t}$, where $f_{l,t}=\mathscr{F}^{-1} [\mathbbm{1}_{B(0,t)}\mathscr{F}[f] ] \in C^\infty_0(\R^{d} )$ and $f_{h,t}=\mathscr{F}^{-1} [\mathbbm{1}_{\R^d\setminus B(0,t)}\mathscr{F}[f] ]\in H^s(\R^{d} )$. Furthermore, we have the estimates
		\begin{align}\label{eq: low high Xs estimates}
			 \norm{f_{l,t}}_{C^k_b} = \sum_{\abs{\alpha} \le k} \norm{\p^\alpha f_{l,R}}_{L^\infty} \lesssim \norm{f_{l,t}}_{X^s} \; \text{for each} \; k \in \N \; \text{and} \; \norm{f_{h,t}}_{H^s} \lesssim \norm{f_{h,t}}_{X^s}.     
		\end{align}
		
		\item If $k \in \N$ and $s > k + d/2$, then we have the continuous inclusion $X^s(\R^{d}) \hookrightarrow C^k_0(\R^d; \R)$.
		
		\item  If $s > d/2$, then for any $f \in X^s(\R^{d}), g \in H^s(\R^{d})$ we have $fg \in H^s(\R^{d} )$;  moreover, $\norm{fg}_{H^s} \lesssim \norm{f}_{X^s} \norm{g}_{H^s}$  for all $f \in X^s(\R^{d})$ and $g \in H^s(\R^{d})$.
		
		\item  If $s \ge 1$, then $[\p_1 \eta]_{\dot{H}^{-1}} \lesssim  \norm{f}_{X^s}$ for all $f \in X^s(\R^{d})$. In particular, the map $\p_1 : X^{s}(\R^{d}) \to \dot{H}^{-1}(\R^{d}) \cap H^{s-1}(\R^{d} )$ is continuous and injective.
	\end{enumerate}
\end{thm}
\begin{proof}
All of these except the separability assertion from the first item are proved in Proposition 5.3 and Theorems 5.6 in \cite{leonitice}.  Separability follows from the calculations leading up to equation (B.1.20) in the proof for the second item of Proposition B.2 in \cite{noahtice2}. 
\end{proof}

Next, we introduce the container space for the solution tuple $(u,p,\eta)$.

\begin{defn} \label{defn:Xsforsol}
	For $\R \ni s \ge 0$.
	\begin{enumerate}
		 \item We define the separable Hilbert space 
		 \begin{align}\label{eq: solution space X^s}
			  \mathcal{X}^s = 
				 \{ (u,p,\eta) \in \Hzerostan{s+2}(\Omega ;\R^n) \times H^{s+1}(\Omega; \R) \times X^{s+5/2}(\R^{n-1}; \R)  \}
		 \end{align}
	   endowed with the norm $\norm{(u,p,\eta)}_{\mathcal{X}^s}^2 = \norm{u}_{\Hzerostan{s+2}}^2 + \norm{p}_{H^{s+1}}^2 + \norm{\eta}_{X^{s+5/2}}^2$.
	   
	   \item We define the separable Hilbert space 
		 \begin{align}\label{eq: solution space X^s with alpha}
			  \mathcal{X}_\alpha^s = 
				 \{ (u,p,\eta) \in \alphaHzeros{s+2}(\Omega ;\R^n) \times H^{s+1}(\Omega; \R) \times X^{s+5/2}(\R^{n-1}; \R)  \}
		 \end{align}
	   endowed with the norm $\norm{(u,p,\eta)}_{\mathcal{X}_\alpha^s}^2 = \norm{u}_{\alphaHzeros{s+2}}^2 + \norm{p}_{H^{s+1}}^2 + \norm{\eta}_{X^{s+5/2}}^2$.
	\end{enumerate}

	\end{defn}

Next, we record an embedding result for $\mathcal{X}^s$ and $\mathcal{X}_\alpha^s$.
\begin{prop}\label{prop:embed X}
	Suppose $\R \ni s \ge 0$ and $\mathcal{X}^s,\mathcal{X}_\alpha^s $ is the Banach space in Definition~\ref{defn:Xsforsol}. If $s > n/2$, then we have the continuous inclusion 
		\begin{align}
			 \mathcal{X}^s, \mathcal{X}_\alpha^s \subseteq C_b^{s + 1 - \lf n/2 \rf}(\Omega ; \R^n) \times C_b^{s - \lf n/2\rf}(\Omega ; \R)\times C_0^{s + 1 -\lf (n-1)/2 \rf}(\R^{n-1}; \R).
		\end{align}
		Moreover, if $(u,p,\eta) \in \mathcal{X}^s$ or $(u,p,\eta) \in \mathcal{X}_\alpha^s$, then 
		\begin{align} 
			\lim_{\abs{x'} \to \infty} \p^\alpha u(x) &= 0 \; \text{for all} \; \alpha \in \N^n \; \text{such that} \; \abs{\alpha} \le s + 1 - \lf n/2 \rf \\
			\lim_{\abs{x'} \to \infty} \p^\alpha p(x) &= 0 \; \text{for all} \; \alpha \in \N^n \; \text{such that} \; \abs{\alpha} \le s - \lf n/2 \rf.
		\end{align}
	\end{prop}

\begin{proof}
This follows from Proposition 6.3 in \cite{leonitice} and the continuous injections $\Hzerostan{s+2}(\Omega;\R^n),\alphaHzeros{s+2}(\Omega;\R^n) \hookrightarrow H^{s+2}(\Omega; \R^n)$. 
\end{proof}

Next we study the linear maps $\Upsilon_{\alpha, \beta, \gamma, \sigma} : \mathcal{X}^s \to \mathcal{Y}^s$ and $\mathfrak{T}_{\alpha, \beta, \gamma, \sigma} : \mathcal{X}_\alpha^s \to \mathcal{Z}^s$ defined via 
\begin{multline} \label{eq:upsilon}
	\Upsilon_{\alpha, \beta, \gamma, \sigma}(u,p,\eta) = (\diverge S(p,u) - \gamma \p_1 u + (\nabla' \eta , 0), \diverge u, u_n \rvert_{\Sigma_b} + \gamma \p_1 \eta, S(p,u) e_n \rvert_{\Sigma_b} + \sigma \Delta'\eta  e_n, [\alpha S(p,u) e_n + \beta u]'),
\end{multline}
and 
\begin{align} \label{eq: mathfrak T}
	\mathfrak{T}_{\alpha, \beta, \gamma, \sigma}(u,p,\eta) = (\diverge S(p,u) - \gamma \p_1 u + (\nabla' \eta , 0), \diverge u, u_n \rvert_{\Sigma_b} + \gamma \p_1 \eta, S(p,u) e_n \rvert_{\Sigma_b} + \sigma \Delta'\eta  e_n),
\end{align}
which are the solution operators corresponding to the system \eqref{eq:gammacapillary} with generic $l \in H^{1/2}(\Sigma_0; \R^{n-1})$ and with $l =0$, respectively. The next result shows that these maps are well-defined, bounded, and also injective. 

\begin{prop}\label{prop:Upsiloninjective}
Suppose $\R \ni \alpha > 0, \beta \in \R^{n \times n}$ is positive definite, $\gamma \in \R \setminus \{0\}$, $\R \ni \sigma \ge 0$, and $\R \ni s \ge 0$. The following the hold.
\begin{enumerate}
	 \item The linear map $\Upsilon_{\alpha, \beta, \gamma, \sigma}: \mathcal{X}^s \to \mathcal{Y}^s$ defined in \eqref{eq:upsilon} is well-defined, continuous, and injective.
	 \item The linear map $\mathfrak{T}_{\alpha, \beta, \gamma, \sigma}: \mathcal{X}_\alpha^s \to \mathcal{Z}^s$ defined in \eqref{eq: mathfrak T} is well-defined, continuous, and injective.
	 \item Furthermore, there exists a constant $c > 0$ for which $\norm{\mathfrak{T}_{\alpha, \beta, \gamma, \sigma}}_{\mathcal{L}(\mathcal{X}_\alpha^s ; \mathcal{Z}^s)} \le c$ for all $\alpha >0$.
\end{enumerate} 
\end{prop}

\begin{proof}
	To prove the first and second items, we first note that by Proposition 3.13 in \cite{tice} and standard trace theory, the maps $\Upsilon_{\alpha, \beta, \gamma, \sigma}$ and $\mathfrak{T}_{\alpha, \beta, \gamma, \sigma}$ are well-defined and continuous. To show that $\Upsilon_{\alpha, \beta, \gamma, \sigma}$ is injective, we suppose $(u, p, \eta) \in \mathcal{X}^s$ and  $\Upsilon_{\alpha, \beta, \gamma, \sigma} \left( u, p, \eta \right) = 0$. We note that if $\tilde{p} = p - \eta$, then $\nabla \tilde{p} = \nabla p - (\nabla' \eta,0)$ and $\tilde{p}I = p I - \eta I$. Therefore $\Upsilon_{\alpha, \beta, \gamma, \sigma} (u,p,\eta) =0$ if and only if $(u,\tilde{p},\eta)$ satisfies 
		\begin{align}\label{eq:upsilon=0}
			\begin{cases}
				\diverge S(\tilde{p},u) - \gamma \p_1 u = 0, & \text{in} \; \Omega \\
				\diverge u = 0, & \text{in} \; \Omega\\
				S(\tilde{p},u) e_n = (\eta - \sigma \Delta' \eta) e_n, & \text{on} \; \Sigma_b\\
				u_n + \gamma \p_1 \eta = 0, & \text{on} \; \Sigma_b \\
				\left[\alpha S(p,u)e_n + \beta u\right]' = 0 , & \text{on} \; \Sigma_{0}\\
				u_n= 0, & \text{on} \; \Sigma_{0}.
			\end{cases}
		\end{align}
		We note that by Tonelli's theorem, Parseval's theorem, and the fifth item of Theorem~\ref{thm: X^s} we have $\hat{u}(\xi, \cdot) \in H^{s}((0,b);\C^n)$ and $\widehat{\tilde{p}}(\xi,\cdot) \in H^1((0,b);\C)$, for a.e. $\xi \in \R^{n-1}$. By the second item in Theorem~\ref{thm: X^s}, $\hat{\eta} \in L^1(\R^{n-1}; \R) + L^2(\R^{n-1}, (1+\abs{\xi}^2)^{(s+5/2)/2} d\xi; \R)$. Thus, we may apply the horizontal Fourier transform to \eqref{eq:upsilon=0} to deduce for  a.e. $\xi \in \R^{n-1}$, $w = \hat{u}(\xi,\cdot), q = \widehat{\tilde{p}}(\xi,\cdot)$ satisfies
		\begin{align} \label{eq:odesys}
		\begin{cases}
		\left( - \p_n^2 + 4\pi^2 \abs{\xi}^2 \right) w' + 2\pi i \xi q - 2\pi i \xi_1 \gamma w' = 0, & \text{in} \; (0, b) \\
		\left( - \p_n^2 + 4\pi^2 \abs{\xi}^2 \right) w_n + \p_n q - 2\pi i \xi_1 \gamma w_n = 0, & \text{in} \; (0, b)  \\
		2\pi i \xi \cdot w' + \p_n w_n = 0, & \text{in} \; (0, b)  \\
		-\p_n w' - 2\pi i \xi w_n = 0,  & \text{for} \; x_n = b \\
		q - 2 \p_n w_n = (1+4\pi^2 \abs{\xi}^2 \sigma) \hat{\eta}, & \text{for} \; x_n = b \\
		w_n + 2\pi i \xi_1 \gamma \hat{\eta} = 0,  & \text{for} \; x_n = b \\
		[( \p_n  - \frac{1}{\alpha} \beta) w] ' = 0,  &\text{for} \; x_n = 0 \\
		w_n = 0, &\text{for} \; x_n = 0.
		\end{cases}
		\end{align}
		For a.e. $\xi \in \R^{n-1}$, by the first three equations in \eqref{eq:odesys} we have 
	\begin{multline}\label{ode 1}
		 2\pi i \xi_1 \gamma w' = \left( - \p_n^2 + 4\pi^2 \abs{\xi}^2 \right) w' + 2\pi i \xi q - 2\pi i \xi (2\pi i \xi \cdot w' + \p_n w_n) \\
		 = 2 \pi i \xi q - (2\pi i \xi \otimes w' + w' \otimes 2\pi i \xi) 2\pi i \xi - \p_n (\p_n w' + 2\pi i \xi w_n)
	\end{multline}
	and 
	\begin{multline}\label{ode 2}
		 2 \pi i \xi_1 \gamma w_n = \left( - \p_n^2 + 4\pi^2 \abs{\xi}^2 \right) w_n + \p_n q - \p_n (2\pi i \xi \cdot w' + \p_n w_n) \\ 
		 = - 2\pi i \xi \cdot (\p_n w' + 2\pi i \xi w_n) + \p_n (q - 2 \p_n w_n).
	\end{multline}
	Using \eqref{ode 1}, \eqref{ode 2}, integration by parts and the boundary conditions in \eqref{eq:odesys}, for a.e. $\xi \in \R^{n-1}$ we have 
	\begin{multline}\label{ode 3}
		 \int_0^b 2 \pi i \xi_1 \gamma w' \cdot \overline{w'} \; dx_n = \int_0^b -q \overline{2 \pi i \xi \cdot w'} + (2 \pi i \xi \otimes w' + w' \otimes w\pi i \xi) : \overline{v' \otimes 2\pi i \xi} - \p_n (\p_n w' + 2\pi i \xi w_n) \cdot \overline{w'} \; dx_n \\
		 = \int_0^b -q \overline{2 \pi i \xi \cdot w'} + (2 \pi i \xi \otimes w' + w' \otimes w\pi i \xi) : \overline{v' \otimes 2\pi i \xi} +  (\p_n w' + 2\pi i \xi w_n) \cdot \overline{\p_n w'} \; dx_n + \frac{1}{\alpha} \beta w(0) \cdot \overline{w(0)}
	\end{multline}
	and 
	\begin{multline}\label{ode 4}
		\int_0^b 2 \pi i \xi_1 \gamma w_n \overline{w_n} \; dx_n = \int_0^b (\p_n w' + 2\pi i \xi w_n) \cdot \overline{2\pi i \xi w_n} + \p_n (q - 2 \p_n w_n) \overline{w_n} \; dx \\ = \int_0^b (\p_n w' + 2\pi i \xi w_n) \cdot \overline{2\pi i \xi w_n} -  (q - 2 \p_n w_n) \overline{\p_nw_n} \; dx + (1+4\pi^2 \abs{\xi}^2 \sigma) \hat{\eta}\overline{w_n(b)}.
	\end{multline}
	We also note that by exploiting the symmetry of $2\pi i \xi \otimes w' + w' \otimes 2 \pi i \xi$, we can write 
	\begin{align}\label{ode 5}
	(2\pi i \xi \otimes w' + w' \otimes 2 \pi i \xi) : \overline{v' \otimes 2\pi i \xi}	 = \frac{1}{2} (2\pi i \xi \otimes w' + w' \otimes 2 \pi i \xi) : \overline{2 \pi i \xi \otimes w' + w' \otimes 2 \pi i \xi}.
	\end{align}
	Upon rearranging \eqref{ode 3}, \eqref{ode 4}, and \eqref{ode 5}, and the third to last equation $w_n + 2\pi i \xi_1 \gamma \hat{\eta} = 0$, we can deduce that  
\begin{multline}
		\int_0^b - \gamma 2\pi i \xi_1 \abs{w}^2 + 2 \abs{\p_n w_n}^2 + \abs{\p_n w' + 2\pi i \xi w_n}^2 + \frac{1}{2} \abs{2 \pi i \xi \otimes w' + w' \otimes 2\pi i \xi}^2 \; dx_n +  \frac{1}{\alpha} \beta w(0) \cdot \overline{w(0)}\\ = \int_0^b - \gamma 2\pi i \xi_1 w \cdot \overline{w} + 2 \p_n w_n \overline{\p_n w_n}  + (\p_n w' + 2\pi i \xi w_n) \cdot \overline{\p_n w'+ 2\pi i \xi w_n}   +(2\pi i \xi \otimes w' + w' \otimes 2 \pi i \xi) : \overline{v' \otimes 2\pi i \xi}\; dx_n \\= -(1+4\pi^2\abs{\xi}^2 \sigma)\hat{\eta}(\xi) \overline{w_n(\xi,b)} = - 2\pi i \xi_1 \gamma(1+4\pi^2\abs{\xi}^2 \sigma) \abs{\eta(\xi)}^2. 
\end{multline}	
			By taking the real part of this expression and applying the coercivity condition \eqref{eq: beta coercive}, we see that we must have for a.e. $\xi \in \R^{n-1}$, $\p_n w_n \equiv 0$ in $(0,b)$, $\p_n w' + 2\pi i \xi w_n \equiv 0$ in $(0,b)$, and $w(0) = 0$. This implies that $w_n \equiv 0$ in $[0,b]$, which in turn implies that we must have $w \equiv 0$ in $[0,b]$. Then by the first equation, we must have $q \equiv 0$. By the third to last equation, we find that $\eta \equiv 0$. From this we find that $(u,p,\eta) = (0,0,0)$, so we can conclude that $\Upsilon_{\alpha, \beta, \gamma, \sigma}$ is injective. The same argument shows that $\mathfrak{T}_{\alpha, \beta, \gamma, \sigma}$ is injective.
	The last item follows from the observation that $\alpha$ does not appear on the right hand side of \eqref{eq: mathfrak T}.
\end{proof}
Next we show that $\Upsilon_{\alpha, \beta, \gamma, \sigma}$ and $\mathfrak{T}_{\alpha, \beta, \gamma, \sigma}$ surjective. To do so we must construct the free surface function $\eta$ from a given data tuple $(f,g,h,k,l) \in \mathcal{Y}^s$ or $(f,g,h,k) \in \mathcal{Z}^s$ in the case when $l = 0$. We record this set of constructions in the next subsection. 

\subsection{Construction of the free surface function and the isomorphism associated to \eqref{eq:gammacapillary}}

\begin{lem} \label{lem:etaconstruct1}
	Suppose $\R \ni \alpha > 0$, $\beta \in \R^{n \times n}$ is positive definite, $\gamma \in \R \setminus \{0\}$, $\sigma > 0$, $\N \ni n \ge 2$, $\R \ni s \ge 0$, and let $\mathcal{Y}^s, \mathcal{Z}^s$ be the Banach spaces defined in Definition~\ref{defn:dataY}. The following hold.

	\begin{enumerate}
		 \item For every $(f,g,h,k,l) \in \mathcal{Y}^s$, there exists an $\eta_\alpha \in X^{s+\frac{5}{2}}(\R^{n-1};\R)$ for which the modified data tuple
		 \begin{multline}\label{eq:moddata}
				 (f - (\nabla' \eta_\alpha, 0), g,h - \gamma \p_1 \eta_\alpha,  k - \sigma \Delta'\eta_\alpha e_n, l )
				  \in H^{s} (\Omega;\R^n) \times H^{s+1}(\Omega)  \\ \times H^{s+\frac{3}{2}}(\Sigma_b; \R^n) \times H^{s+\frac{1}{2}}(\Sigma_b; \R^n) \times H^{s+\frac{1}{2}}(\Sigma_0; \R)  
		 \end{multline}
		 belongs to the range of $\Upsilon_{\alpha, \beta, \gamma, \sigma}$ defined in \eqref{eq:upsilon} and there exists a constant $C > 0$ for which 
		 \begin{align}\label{eq: eta X bound}
			  \norm{\eta_\alpha}_{X^{s+\frac{5}{2}}} \le C \norm{(f,g,h,k,l)}_{\mathcal{Y}^s}. 
		 \end{align}
		 \item For every $(f,g,h,k) \in \mathcal{Z}^s$, there exists an $\eta_\alpha \in X^{s+\frac{5}{2}}(\R^{n-1};\R)$ for which the modified data tuple
		 \begin{align}\label{eq:moddata alpha}
				 (f - (\nabla' \eta_\alpha, 0), g,h - \gamma \p_1 \eta,  k - \sigma \Delta'\eta_\alpha e_n)
				  \in H^{s} (\Omega;\R^n) \times H^{s+1}(\Omega) \times H^{s+\frac{3}{2}}(\Sigma_b; \R^n) \times H^{s+\frac{1}{2}}(\Sigma_b; \R^n)
		 \end{align}
		 belongs to the range of $\mathfrak{T}_{\alpha, \beta, \gamma, \sigma}$ defined in \eqref{eq: mathfrak T} and there exists a constant $C > 0$ for which 
		 \begin{align}\label{eq: eta X bound alpha}
			  \norm{\eta_\alpha}_{X^{s+\frac{5}{2}}} \le C \norm{(f,g,h,k)}_{\mathcal{Z}^s}. 
		 \end{align}
		 \item Furthermore, there exists a constant $C > 0$ for which \eqref{eq: eta X bound} holds for all $\alpha \in (0,1)$. In other words, the constant $C > 0$ can be chosen to be independent of $\alpha$ if $\alpha \in (0,1)$.
	\end{enumerate}

	\end{lem}
	\begin{proof}
		We proceed to prove the first item. Given $(f,g,h,k,l) \in \mathcal{Y}^s$, we propose to define $\eta_\alpha \in X^{s+\frac{5}{2}}(\R^{n-1}; \R)$ via $\hat{\eta}_\alpha =
				\rho_{\alpha,\beta,\gamma}^{-1}\mathscr{F}\{\Lambda_{\alpha,\beta,\gamma}(f,g,h,k,l)\}$,
		 where the operator $\Lambda_{\alpha,\beta,\gamma}$ is defined in \eqref{eq: Lambda gamma} and $\rho_{\alpha,\beta,\gamma}$ is defined in \eqref{eq:rho}. 
		
		Note that $\hat{\eta}_\alpha = \overline{\hat{\eta}_\alpha}$, so $\eta_\alpha$ is real-valued. Furthermore, by using Lemma~\ref{lem:rho} and the continuity of the operator $\Lambda_{\alpha,\beta,\gamma}$ established in Theorem~\ref{thm: regularity of Lambda} we have the estimate 
		\begin{multline}\label{eq: eta construction est}
				 \int_{\R^{n-1}} \left(  \frac{\xi_1^2 + \abs{\xi}^4}{\abs{\xi}^2} \mathbbm{1}_{B(0,1)}(\xi)  + (1+\abs{\xi}^2)^{s+\frac{5}{2}}\mathbbm{1}_{B(0,1)^c}(\xi)  \right)  \abs{\hat{\eta_\alpha}(\xi)}^2  \; d\xi  \\
				  \lesssim_\alpha \int_{\R^{n-1}} \max \{ \abs{\xi}^{-2}, \abs{\xi}^{2s+3} \} \abs{\mathscr{F}[\Lambda_{\alpha,\beta,\gamma}(f,g,h,k,l) ](\xi)}^2 d\xi   \lesssim_\alpha \norm{(f,g,h,k,l)}_{\mathcal{Y}^s}^2.
		\end{multline}
		This shows that if we define $\eta_\alpha = (\hat{\eta}_\alpha)^\vee$, $\eta_\alpha$ is a well-defined real-valued tempered distribution that belongs to $X^{s+\frac{5}{2}}(\R^{n-1})$. 
		
		Next we show that the modified data given in \eqref{eq:moddata} belongs to the range of $\Upsilon_{\alpha, \beta, \gamma, \sigma}$. To show this we invoke Theorem~\ref{thm: overdeterminediso} and show that it belongs to $\overleftarrow{\ker} \mathscr{B}_{\alpha,
		\gamma}$. For any $\psi \in H^{s+1/2}(\Sigma_b ; \R)$, by Plancherel's theorem we have  
		\begin{multline}
			 \left\langle  \Lambda_{\alpha,\beta,\gamma}(f - (\nabla' \eta_\alpha, 0),g, h - \gamma \p_1 h, k - \sigma \Delta' \eta_\alpha e_n,l) , \psi \right \rangle_{L^2} \\ = \int_{\R^{n-1}} \mathscr{F}[\Lambda_{\alpha,\beta,\gamma}(f - (\nabla' \eta_\alpha, 0),g,h - \gamma \p_1 \eta_\alpha, k - \sigma \Delta' \eta_\alpha e_n,l)](\xi) \overline{ \mathscr{F}[\psi](\xi)}
			 \\ =  \int_{\R^{n-1}} \mathscr{F}[\Lambda_{\alpha,\beta,\gamma}(f ,g,h,k,l )](\xi) \overline{ \mathscr{F}[\psi](\xi)} + \int_{\R^{n-1}} \mathscr{F}[\Lambda_{\alpha,\beta,\gamma}( - (\nabla' \eta_\alpha, 0),0,0 - \gamma \p_1 \eta_\alpha, - \sigma \Delta' \eta_\alpha e_n,0)](\xi) \overline{ \mathscr{F}[\psi](\xi)}.
		\end{multline}
		Furthermore, by letting $(v,q) = \Phi_{\alpha, \beta^T, -\gamma}^{-1}(0, 0, \psi e_n, 0)$ we have 
		\begin{multline}
			\int_{\R^{n-1}} \mathscr{F}[\Lambda_{\alpha,\beta,\gamma}(f,g,h,k,l)](\xi) \overline{ \mathscr{F}[\psi](\xi)} = \int_{\R^{n-1}}  \rho_{\alpha,\beta,\gamma}(\xi) \hat{\eta}_\alpha(\xi) \overline{\mathscr{F}[\psi](\xi)}\; d\xi   \\
			  =  \int_{\R^{n-1}}   \overline{m_\alpha(\xi,-\gamma)} \hat{\eta}_\alpha(\xi)  \overline{\mathscr{F}[\psi](\xi)}\; d\xi +  \int_{\R^{n-1}}  4\pi^2 \abs{\xi}^2 \sigma \overline{m_\alpha(\xi,-\gamma)}  \hat{\eta}_\alpha(\xi) \overline{\mathscr{F}[\psi](\xi)} \; d\xi  + \int_{\R^{n-1}} 2 \pi i \gamma \xi_1  \hat{\eta}_\alpha(\xi) \overline{\mathscr{F}[\psi](\xi)}
			  \\
			 =\int_{\R^{n-1}}   \overline{m_\alpha(\xi,-\gamma)} \hat{\eta}_\alpha(\xi)  \overline{\mathscr{F}[\psi](\xi)}\; d\xi  + \int_0^b  \sigma \Delta' \eta_\alpha(\xi) e_n \cdot \overline{v} \; dx_n +  \int_{\R^{n-1}} \mathscr{F}[\Lambda_{\alpha,\beta,\gamma}(0,0,\gamma \p_1 \eta_\alpha,0,0)](\xi) \overline{\mathscr{F}[\psi](\xi)} \\
			 = \int_{\R^{n-1}}   \overline{m_\alpha(\xi,-\gamma)} \hat{\eta}_\alpha(\xi)  \overline{\mathscr{F}[\psi](\xi)}\; d\xi  +  \int_{\R^{n-1}} \mathscr{F}[\Lambda_{\alpha,\beta,\gamma}(0,0,0,\sigma \Delta'\eta_\alpha e_n,0)](\xi) \overline{\mathscr{F}[\psi](\xi)} \\  +  \int_{\R^{n-1}} \mathscr{F}[\Lambda_{\alpha,\beta,\gamma}(0,0,\gamma \p_1 \eta_\alpha, 0,0)](\xi) \overline{\mathscr{F}[\psi](\xi)}.
		\end{multline}
		By the second and last equations in \eqref{eq:adjoint}, we have
		\begin{align}
			 \overline{m_\alpha(\xi,-\gamma)} = \int_0^b \p_n \overline{V_n(\xi,x_n,-\gamma)}\; d\xi = \int_0^b 2 \pi i \xi \cdot \overline{V'(\xi,x_n,-\gamma)} d\xi, 
		\end{align}
		therefore 
		\begin{multline}
			\int_{\R^{n-1}}   \overline{m_\alpha(\xi,-\gamma)} \hat{\eta}_\alpha(\xi) \overline{\mathscr{F}[\psi](\xi)} \; d\xi = \int_{\R^{n-1}}  \int_0^b  (2 \pi i \xi, 0 )  \hat{\eta}(\xi) \cdot \overline{\mathscr{F}[v(\cdot, x_n)](\xi)} \; \; dx_n d\xi \\
			= \int_{\R^{n-1}}  \mathscr{F}[\Lambda_{\alpha,\beta,\gamma}((\nabla' \eta_\alpha, 0),0,0,0,0)](\xi) \overline{\mathscr{F}[\psi](\xi)}.
		\end{multline}
		Thus upon rearranging, we have 
		\begin{align}
			\left\langle  \Lambda_{\alpha,\beta,\gamma}(f - (\nabla' \eta_\alpha, 0),g, h - \gamma \p_1 h, k - \sigma \Delta' \eta_\alpha e_n,l) , \psi \right \rangle_{L^2}  = 0,
		\end{align}
		and the first item follows immediately. 

		The second item follows similarly from the first item, where given $(f,g,h,k) \in \mathcal{Z}^s$ we propose to define $\eta_\alpha \in X^{s+\frac{5}{2}}(\R^{n-1}; \R)$ via $\hat{\eta}_\alpha = \rho_{\alpha,\beta,\gamma}^{-1}\mathscr{F}\{\Lambda_{\alpha,\beta,\gamma}(f,g,h,k,0)\}$. For the last item, we note that by the last items of Lemma~\ref{lem:rho}, and Theorem~\ref{thm: regularity of Lambda}, the constants appearing on the right hand side of \eqref{eq: eta construction est} can be chosen to be independent of $\alpha$ if $\alpha \in (0,1)$. The third item then follows.
	\end{proof}
	
	For the special case of $n =2$, we can also construct the free surface function $\eta$ in the case without surface tension.
	
	\begin{lem}\label{lem:etaconstruct2}
		Suppose $\gamma \in \R \setminus \{0\}$,   $\sigma = 0$ and $n = 2$, $s \ge 0$, and let $\mathcal{Y}^s, \mathcal{Z}^s$ be the Banach space defined in Definition~\ref{defn:dataY}. The following hold.
		\begin{enumerate}
			 \item For every $(f,g,h,k, l) \in \mathcal{Y}^s$, there exists an $\eta \in H^{s+\frac{5}{2}}(\R^{n-1};\R)$ for which the modified data tuple
			 \begin{align}
				 (f - \p_1 \eta e_1 , g, h - \gamma \p_1 \eta, k + \eta e_2, l)
				  \in H^{s} (\Omega;\R^2) \times H^{s+1}(\Omega; \R)  \times H^{s+\frac{3}{2}}(\Sigma_b; \R)  \times H^{s+\frac{1}{2}}(\Sigma_b; \R^2) \times H^{s+\frac{1}{2}}(\Sigma_0; \R) 
		 \end{align}
			 belongs to the range of $\Upsilon_{\alpha, \beta, \gamma, \sigma}$ defined in \eqref{eq:upsilon}. Moreover, there exists a constant $C > 0$ for which 
			 $\norm{\eta}_{H^{s+\frac{5}{2}}} \le C \norm{(f,g,h,k,l)}_{\mathcal{Y}^s}$.
			 
			 \item For every $(f,g,h,k) \in \mathcal{Z}^s$, there exists an $\eta \in H^{s+\frac{5}{2}}(\R^{n-1};\R)$ for which the modified data tuple
			 \begin{align}
				 (f - \p_1 \eta e_1 , g, h - \gamma \p_1 \eta, k + \eta e_2)
				  \in H^{s} (\Omega;\R^2) \times H^{s+1}(\Omega ; \R)  \times H^{s+\frac{3}{2}}(\Sigma_b; \R)  \times H^{s+\frac{1}{2}}(\Sigma_b; \R^2) 
		 \end{align}
			 belongs to the range of $\mathfrak{T}_{\alpha, \beta, \gamma, \sigma}$ defined in \eqref{eq: mathfrak T}. Moreover, there exists a constant $C > 0$ for which 
			 \begin{align}\label{eq: eta bound sigma = 0}
				  \norm{\eta}_{H^{s+\frac{5}{2}}} \le C \norm{(f,g,h,k)}_{\mathcal{Z}^s}. 
			 \end{align}
			 \item Furthermore, there exists a constant $C > 0$ for which \eqref{eq: eta bound sigma = 0} holds for all $\alpha \in (0,1)$. In other words, the constant $C > 0$ can be chosen to be independent of $\alpha$ if $\alpha \in (0,1)$.
		\end{enumerate}

	\end{lem}
	\begin{proof}
	To prove the first item, we note that by Theorem~\ref{thm: X^s}, in dimension $n =2$ the specialized space $X^s(\R^{n-1} ; \R)$ is the standard Sobolev space $H^s(\R^{n-1};\R)$. So given $(f,g,h,k,l) \in \mathcal{Y}^s$, we similarly define $\eta \in H^{s+5/2}(\R^{n-1} ; \R)$ via $\eta_\alpha = (\hat{\eta})^\vee$ where  $\hat{\eta}_\alpha =
		   \rho_{\alpha,\beta,\gamma}^{-1}\mathscr{F}\{\Lambda_{\alpha,\beta,\gamma}(f,g,h,k,l)\}$.
	Lemma~\ref{lem:rho} and Theorem~\ref{thm: regularity of Lambda} imply that
	\begin{multline}
		 \norm{\eta_\alpha}_{H^{s+5/2}}^2 = \int_{\R^{n-1}} (1+\abs{\xi}^2)^{s+ 5/2} \abs{\hat{\eta}(\xi)}^2 \; d\xi \\ 
		 \lesssim \int_{\R^{n-1}} (1+\abs{\xi}^2)^{s+5/2} \abs{\xi}^{-2} \abs{\mathscr{F}\{\Lambda_{\alpha,\beta,\gamma}(f,g,h,k,l)\}(\xi)}^2 \lesssim_\alpha \norm{(f,g,h,k)}_{\mathcal{Y}^s}^2.  
	\end{multline}
	This shows that $\eta_\alpha = (\hat{\eta}_\alpha)^\vee \in H^{s+5/2}(\R^{n-1})$. To conclude the first item we follow the same calculations as the previous lemma to show that the modified data tuple belongs to the range of $\Upsilon_{\alpha, \beta, \gamma, 0}$. The second and third items follow from a similar set of arguments presented in the proof of Lemma~\ref{lem:etaconstruct1}.
	\end{proof}

Now we are ready to prove that $\Upsilon_{\alpha, \beta, \gamma, \sigma}: \mathcal{X}^s \to \mathcal{Y}^s$ and $\mathfrak{T}_{\alpha, \beta, \gamma, \sigma}: \mathcal{X}_\alpha^s \to \mathcal{Z}^s$ are isomorphisms when $\sigma > 0$ and $n \ge 2$, and when $\sigma = 0$ and $n = 2$.

\begin{thm}\label{thm:upsiloniso}
Suppose $\R \ni \alpha > 0$, $\beta \in \R^{n \times n}$ is positive definite, $\gamma \in \R \setminus \{0\}$, and  $s \ge 0$. The following hold.
\begin{enumerate}
	 \item   If $\sigma > 0$ and $n \ge 2$, then the bounded linear maps $\Upsilon_{\alpha, \beta, \gamma, \sigma} : \mathcal{X}^s \to \mathcal{Y}^s$ defined in \eqref{eq:upsilon} and $\mathfrak{T}_{\alpha, \beta, \gamma, \sigma}: \mathcal{X}_\alpha^s \to \mathcal{Z}^s$ defined in \eqref{eq: mathfrak T} are isomorphisms.
	 \item   If $\sigma = 0$ and $n = 2$, then the bounded linear maps $\Upsilon_{\alpha, \beta, \gamma, \sigma} : \mathcal{X}^s \to \mathcal{Y}^s$ defined in \eqref{eq:upsilon} and $\mathfrak{T}_{\alpha, \beta, \gamma, \sigma}: \mathcal{X}_\alpha^s \to \mathcal{Z}^s$ defined in \eqref{eq: mathfrak T} are isomorphisms.
	 \item  If $\sigma > 0$ and $n \ge 2$, then there exists a constant $C > 0$ for which 
	 \begin{align}
		 \sup_{\alpha \in (0,1)} \left( \norm{\mathfrak{T}_{\alpha, \beta, \gamma, \sigma}}_{\mathcal{L}(\mathcal{X}_\alpha^s ; \mathcal{Z}^s)} + \norm{\mathfrak{T}^{-1}_{\alpha, \beta, \gamma, \sigma}}_{\mathcal{L}( \mathcal{Z}^s; \mathcal{X}_\alpha^s)}\right) \le C.
	 \end{align}
	 If $\sigma = 0$ and $n = 2$, then there exists a constant $c > 0$ for which 
	 \begin{align}
		 \sup_{\alpha \in (0,1)} \left( \norm{\mathfrak{T}_{\alpha, \beta, \gamma, 0}}_{\mathcal{L}(\mathcal{X}_\alpha^s ; \mathcal{Z}^s)} + \norm{\mathfrak{T}^{-1}_{\alpha, \beta, \gamma, 0}}_{\mathcal{L}( \mathcal{Z}^s ; \mathcal{X}_\alpha^s)} \right) \le c.
	 \end{align}
\end{enumerate}

\end{thm}
\begin{proof}
To prove the first item, by Proposition~\ref{prop:Upsiloninjective}, it suffices to show that $\Upsilon_{\alpha, \beta, \gamma, \sigma}$ and $\mathfrak{T}_{\alpha, \beta, \gamma, \sigma}$ are surjective. To prove that $\Upsilon_{\alpha, \beta, \gamma, \sigma}$ is surjective, we suppose $(f,g,h,k,l) \in \mathcal{Y}^s$ and define the free surface function $\eta \in X^{s+5/2}(\R^{n-1}; \R)$ by the construction in Lemma~\ref{lem:etaconstruct1}. By Theorem~\ref{thm: overdeterminediso}, there exists $(u,p) \in \Hzerostan{s+2}(\Omega ; \R^n) \times H^{s+1}(\Omega ; \R)$ such that $\Psi_{\alpha,\beta,\gamma} (u,p) = (\Phi_{\alpha, \beta, \gamma}(u,p),u_n \rvert_{\Sigma_b}) = (f - (\nabla' \eta, 0), g, k - \sigma \Delta'\eta e_n, l, h - \gamma \p_1 \eta)$. Therefore, we find that $\Upsilon_{\alpha, \beta, \gamma, \sigma} (u,p, \eta) = (f,g,h,k,l)$. This shows that $\Upsilon_{\alpha, \beta, \gamma, \sigma}$ is surjective, and it follows that $\Upsilon_{\alpha, \beta, \gamma, \sigma}$ is an isomorphism. The surjectivity of $\mathfrak{T}_{\alpha, \beta, \gamma, \sigma}$ follows from a similar set of arguments.To prove the second item we follow the same argument as above, using Lemma~\ref{lem:etaconstruct2} in place of Lemma~\ref{lem:etaconstruct1}, $\Upsilon_{\alpha, \beta, \gamma,0}$ in place of $\Upsilon_{\alpha, \beta, \gamma, \sigma}$, and $\mathfrak{T}_{\alpha, \beta, \gamma,0}$ in place of $\mathfrak{T}_{\alpha, \beta, \gamma, \sigma}$

The third item follows the last item of Proposition~\ref{prop:Upsiloninjective}, the $\alpha$-independent estimate \eqref{eq: regularity solution est w/o alpha} recorded in Proposition~\ref{prop:ind of alpha} and the last items in Lemma~\ref{lem:etaconstruct1} and Lemma~\ref{lem:etaconstruct2}.
\end{proof}

\section{Nonlinear analysis} \label{sec: nonlinear}

\subsection{Preliminaries}\label{sec: product est}

We begin by discussion some assumptions about the slip map $A$.  We set $\beta = DA(0) \in \R^{n \times n}$ and note since $A$ is smooth we have $A(w) = A(0) + \beta w + O(\abs{w}^2)$, so by \eqref{eq: A monotone} and a simple scaling argument, we have 
\begin{align}\label{eq: beta coercive}
	\beta w \cdot w \ge \theta' \abs{w}^2 > 0, \; \forall w \in \R^n \setminus \{0\},
\end{align} 
for some positive constant $\theta' > 0$. We also note that if $w = u + i v$ for $u,v \in \R^n$, then 
\begin{align}\label{eq: beta coercive C}
	 \Re( \beta w \cdot \overline{w}) = \beta u \cdot u + \beta v \cdot v \ge \theta' \abs{w}^2 > 0, \forall w \in \C^n \setminus \{0\}.
\end{align}

Next, we record a set of results on the smoothness of various maps defined in terms of $\eta$ that we will use in the subsequent analysis.

\begin{thm}\label{thm: flattening maps est}
Let $\N \ni n \ge 2$, $\R \ni s > n/2$, and $V$ be a real finite dimensional inner product space. 

\begin{enumerate} 
	\item Suppose $\varphi \in C^\infty_b(\R ; \R)$. Then for $0 \le r \le s$, $f \in H^r(\R^n; V)$, $\eta \in X^s(\R^{n-1}; \R)$ and $\varphi \eta f: \R^n \to V$ defined via $(\varphi \eta f)(x) = \varphi(x_n) \eta(x') f(x)$, we have $\varphi \eta f \in H^r(\R^n; V)$ and $\norm{\varphi \eta f}_{H^r} \lesssim \norm{\eta}_{X^s} \norm{f}_{H^r}$. 
	
	\item Let $\varphi \in C^\infty_b(\R ; \R)$ be such that $\varphi \ge 0$. Then there exists $r_1 > 0$ depending on $n,b,s, \varphi$ such that the maps $\Gamma_1, \Gamma_2: B_{X^s}(0,r_1) \times H^s(\Omega; V) \to H^s(\Omega; V)$ given by $\Gamma_1 (\eta,f) = \frac{f}{1+\eta \varphi}$ and $\Gamma_2(\eta, f) = \frac{\eta f }{1+\eta\varphi}$
	are well-defined and smooth. 
	\item There exists a constant $r_2 > 0$ depending on $d,s$ such that the map $\Gamma : B_{H^s}(0,r_2) \to H^{s}(\R^n ; \R^n)$ given by $\Gamma(f) = f/\sqrt{1 + \abs{f}^2}$
	is well-defined and smooth. 
\end{enumerate}
\end{thm}

\begin{proof}
	We first note that the first item follows from Theorem 5.13 in \cite{leonitice}, the third item follows from Theorem A.14 in \cite{leonitice}, so it suffices to only prove the second item. 
	
	To prove the second item, we first note that since $\Gamma_2(\eta, f) = \eta \Gamma_1(\eta,f)$, if $\Gamma_1$ is well-defined and smooth and then so is $\Gamma_2$ by the first item. Therefore it suffices to show that $\Gamma_1$ is well-defined and smooth. By the eighth item of Theorem~\ref{thm: X^s}, $\norm{\eta}_{C^0_b} \lesssim \norm{\eta}_{X^s}$, and by the ninth item of Theorem~\ref{thm: X^s} and an induction argument, we have $\norm{f \eta^k}_{X^s} \lesssim \norm{f}_{H^s} \norm{\eta}_{X^s}^k$ for all $k \ge 1, f \in H^s(\Omega; V)$ and  $\eta \in X^s(\R^{n-1}; \R)$. The first aforementioned estimate implies that there exists a constant $r > 0$ such that for $\eta \in B_{X^s}(0,r)$ we have $\sum_{k=0}^\infty \norm{\eta \varphi}_{C^0_b}^k \lesssim \sum_{k=0}^{\infty} \norm{\eta}_{X^s}^k < \infty$, and the second aforementioned estimate implies that $\sum_{k=1}^\infty \norm{f \eta^k}_{H^s} \lesssim \norm{f}_{H^s}  \sum_{k=1}^\infty  \norm{\eta}_{X^s}^k < \infty$. This shows that the series $\sum_{k=0}^\infty (-1)^k (\eta \varphi)^k$
converges uniformly to $\frac{1}{1+\eta\varphi}$ in $\Omega$, and the series $\sum_{k=1}^\infty (-1)^k f \eta^k$ converges in $H^s(\Omega;\R)$. Now we note that $\Gamma_1(\eta,f) = \frac{f}{1+\eta \varphi} = f + \sum_{k=1}^\infty (-1)^k f (\eta \varphi)^k \in H^s(\Omega; \R)$, and therefore the map $\Gamma_1$ is well-defined. To show that $\Gamma_1$ is smooth, we consider the map $T: X^s(\R^{n-1};\R) \to \mathcal{L}(H^s(\Omega;V))$ defined via $T(\eta)f = \varphi \eta  f$. By the first item of Theorem~\ref{thm: flattening maps est}, the map $T$ is bounded. Furthermore, in the unital Banach algebra $\mathcal{L}(H^s(\Omega;V))$, the power series $F(L) = \sum_{k=0}^\infty L^k$ converges and defines a smooth function in the unit ball $B_{\mathcal{L}(H^s(\Omega; V))}(0,1)$, thus the composition $F \circ T : X^s(\R^{n-1}; \R) \to  \mathcal{L}(H^s(\Omega;V))$ defines a smooth function. Since $\Gamma(f,g) = F(T(\eta)) f$, we may deduce that there exists a constant $r_1 > 0$ for which $\Gamma_1$ is smooth on $B_{X^s}(0,r_1) \times H^s(\Omega)$. 
\end{proof}

Now we can synthesize the aforementioned results to show that all the nonlinear maps appearing in \eqref{eq: main flattened} are well-defined and $C^2$. 

\begin{thm}\label{thm:smoothnessnonlinear}
	Suppose $ n \ge 2$ and $\sigma > 0$, or $n=2$ and $\sigma=0$. Let $\N \ni s \ge 1 + \tfloor{n/2}$. The following hold. 
	
	\begin{enumerate}
		 \item For any $\delta, M > 0$, define the open set $U^s_{\delta,M}$ of $\mathcal{X}^s$ via
		 \begin{align} \label{eq: ball for solution}
			 U^s_{\delta,M} = \{ (u, p ,\eta) \in \mathcal{X}^s \mid \norm{u}_{H^{s+2}} + \norm{p}_{H^{s+1}} < M,  \norm{\eta}_{X^{s+\frac{5}{2}}} < \delta\}.
		 \end{align}
		 Consider the Hilbert space 
		 \begin{align}\label{eq: E^s}
			 \mathcal{E}^s = \R \times \R \times H^{s+3}(\R^n ; \R^{n \times n}_{\sym}) \times H^{s+\frac{1}{2}}(\R^{n-1}  ; \R^{n\times n}_{\sym}) \times H^{s+2}(\R^n  ; \R^n) \times H^s(\R^{n-1}  ; \R^n).
		 \end{align}
		 Let $\mathfrak{F}$ be as defined in \eqref{eq:flattening}, $\mathcal{J}, \mathcal{A},  \mathcal{H}$ be as defined in \eqref{eq:JandK}, \eqref{eq:A}, and the $\mathcal{A}$-dependent operators be defined as in Section \ref{sec:flatten}. We define the solution operator $\Xi : \mathcal{E}^s \times U^s_{\delta,M} \to \mathcal{Y}^s$ associated to \eqref{eq: main flattened}  via 
	 \begin{multline} \label{eq: Xi}
		 \Xi(\alpha, \gamma,  \mathcal{T}, T, \mathfrak{f}, f, u,p, \eta) = 
		 (\diverge_\mathcal{A} S_\mathcal{A}(p,u) + (u- \gamma e_1)\cdot \nabla_\mathcal{A} u  + u \cdot \nabla_{\mathcal{A}} u - \mathfrak{f} \circ \mathfrak{F} - L_{\Omega_b} f, \mathcal{J}\diverge_{\mathcal{A}}u, \\ u \cdot \mathcal{N} + \gamma\p_1 \eta, S_{\mathcal{A}}(p,u)\mathcal{N} - (\sigma \mathcal{H}(\eta)I + \mathcal{T} \circ \mathfrak{F} + S_b T\rvert_{\Sigma_b}) \mathcal{N}  , [\alpha S_{\mathcal{A}}(p,u)\nu - A(u)]') 
	 \end{multline}
	 where
	 \begin{align}\label{eq: L_Omega and S_b}
		 L_{\Omega_b} f(x) = f(x') \text{ and } S_b T(x',b) = T(x').
	 \end{align} 
	 Then there exists a $\delta > 0$ for which $\Xi$ is well-defined and belongs to $C^2_b(\mathcal{E}^s \times U^s_{\delta,M}; \mathcal{Y}^s)$. Furthermore, we have the estimate 
	 \begin{align}
		  \sup_{\alpha > 0} \norm{\Xi(\alpha,\cdot) \rvert_{\mathcal{E}^s \times U^s_{\delta,M}}}_{C^2_b} < \infty.
	 \end{align}
	 \item Similarly, for any $\delta, M > 0$, define the open set $U^s_{\alpha, \delta,M}$ of $\mathcal{X}^s$ via
	 \begin{align} 
		 U^s_{\alpha,\delta,M} = \{ (u, p ,\eta) \in \mathcal{X}_\alpha^s \mid \norm{u}_{H^{s+2}} + \norm{p}_{H^{s+1}} < M,  \norm{\eta}_{X^{s+\frac{5}{2}}} < \delta\}.
	 \end{align}
	 Consider the Hilbert space $\mathcal{E}^s$ defined via \eqref{eq: E^s}. We define the solution operator $\mathfrak{X} : \mathcal{E}^s \times U^s_{\alpha, \delta,M} \to \mathcal{Z}^s$ associated to \eqref{eq: main flattened} with $A(\cdot) = \beta \cdot$ where $\beta \in \R^{n \times n}$ satisfies \eqref{eq: beta coercive} via 
 \begin{multline} \label{eq: Xi alpha}
	 \mathfrak{X}(\alpha, \gamma,  \mathcal{T}, T, \mathfrak{f}, f, u,p, \eta) = 
	 (\diverge_\mathcal{A} S_\mathcal{A}(p,u) + (u- \gamma e_1)\cdot \nabla_\mathcal{A} u  + u \cdot \nabla_{\mathcal{A}} u - \mathfrak{f} \circ \mathfrak{F} - L_{\Omega_b} f, \mathcal{J}\diverge_{\mathcal{A}}u, \\ u \cdot \mathcal{N} + \gamma\p_1 \eta, S_{\mathcal{A}}(p,u)\mathcal{N} - (\sigma \mathcal{H}(\eta)I + \mathcal{T} \circ \mathfrak{F} + S_b T\rvert_{\Sigma_b}) \mathcal{N}).
 \end{multline}
 Then there exists a $\delta > 0$ for which $\mathfrak{X}$ is well-defined and belongs to $C^2_b(\mathcal{E}^s \times U^s_{\alpha,\delta,M}; \mathcal{Z}^s)$. Furthermore, we have the estimate 
 \begin{align}\label{eq: mathfrak X norm}
	  \sup_{\alpha > 0 } \norm{\mathfrak{X} \rvert_{\mathcal{E}^s \times U^s_{\alpha,\delta,M}}}_{C^2_b} < \infty.
 \end{align}
\end{enumerate}

\end{thm}

\begin{proof}
We proceed to prove the first item. Let $\delta = \min \{r_1, r_2/c_1, \delta_*\}$, where $r_1, r_2$ are the radii from the second and third items of Theorem~\ref{thm: flattening maps est}, $c_1$ is the embedding constant from $X^s(\R^d) \to H^{s-1}(\R^d; \R^d)$ and $0 < \delta_* < 1$ is from Theorem~\ref{comp_C2}. We note that since $\varphi \in C^\infty_b(\R;\R)$, by the first and second items of Theorem~\ref{thm: flattening maps est}, the maps $\Gamma_1, \Gamma_2 : B_{X^r}(0,\delta) \times H^r(\Omega ; \R) \to H^r(\Omega ; \R)$ given by $\Gamma_1(\eta, f) = \frac{f \varphi}{1+\eta \varphi'}$ and $\Gamma_2(\eta,f) = \frac{\eta f \varphi}{1+\eta \varphi'}$ are well-defined and smooth for $r > n /2$. Utilizing this observation, the definition of the $\mathcal{A}$ and the $\mathcal{A}$-dependent operators in Section \ref{sec:flatten}, the fifth and ninth items of Theorem~\ref{thm: X^s}, the fact that $H^r(\R^d;\R)$ is an algebra for $r > d/2$, trace theory and the assumption that $A$ is smooth, the map 
\begin{multline}
	\R  \times \R \times U^s_{\delta,M}  \ni (\alpha, \gamma, T, u, p , \eta)  \mapsto \\
	(\diverge_\mathcal{A} S_\mathcal{A}(p,u) + (u- \gamma e_1)\cdot \nabla_\mathcal{A} u  + u \cdot \nabla_{\mathcal{A}} u, \mathcal{J}\diverge_{\mathcal{A}}u, u \cdot \mathcal{N} + \gamma\p_1 \eta, S_{\mathcal{A}}(p,u)\mathcal{N} \rvert_{\Sigma_b}, \alpha [S_{\mathcal{A}}(p,u)\nu]' \rvert_{\Sigma_0})   \\
	\in  H^s(\Omega ;\R^{n}) \times H^{s+1}(\Omega; \R )  \times  H^{s+3/2}(\Sigma_b ; \R)  \times H^{s+1/2}(\Sigma_b ; \R^{n})  \times  H^{s+1/2}(\Sigma_0 ; \R^{n-1})
\end{multline}
is well-defined and smooth. 

By the supercritical Sobolev embedding $H^{1+\tfloor{n/2}}(\Omega;\R^n) \hookrightarrow C^0_b(\Omega;\R^n)$, the map $A \in C^\infty(\R^n ; \R^n)$ agrees with the map $\tilde{A} = \psi A \in C^\infty_b(\R^n ;\R^n)$ on $B_{H^{s+2}(\Omega;\R^n)}(0,M)$ since $s+ 2 \ge 3 + \tfloor{n/2}$, where $\psi$ is a smooth cutoff function on $B_{\R^n}(0,r(M))$, $r(M)$ depends on $M$ and the embedding constant from $H^{1+\tfloor{n/2}}(\Omega;\R^n) \hookrightarrow C^0_b(\Omega;\R^n)$. Since $\tilde{A} \in C^\infty_b(\R^n ;\R^n)$ and $\tilde{A}(0) = 0$, by Theorem~\ref{thm: smooth Sobolev composition} we may then conclude that the map $U^s_{\delta,M} \ni (u,p,\eta) \mapsto A(u) \rvert_{\Sigma_0} \in H^{s+3/2}(\Sigma_0; \R)$ is well-defined and $C^2$. 

By the fifth item of Theorem~\ref{thm: X^s}, the third item of Theorem~\ref{thm: flattening maps est}, and the fact that $H^{s+1/2}(\R^{n-1}; \R)$ is an algebra, the map 
\begin{align}
	 B_{X^{s+5/2}}(0,\delta) \ni \eta \mapsto \sigma \mathcal{H}(\eta) I \mathcal{N} = \sigma \diverge' \left( \frac{\nabla' \eta}{\sqrt{1+ \abs{\nabla' \eta}^2}} \right) I_{n \times n} (-\nabla'\eta, 1)  \in H^{s+1/2}(\R^{n-1} ; \R) 
\end{align}
is well-defined and smooth. 

By Theorem 7.3 and Lemma A.10 in \cite{leonitice}, the map 
\begin{align}
	 H^{s+1/2}(\R^{n-1} ; \R^{n\times n}_{\sym}) \times H^s(\Omega ; \R^n) \ni  (T, f) \mapsto ( S_b T, L_{\Omega_b} f ) \in H^{s+1/2}(\Sigma_b ; \R^n) \times H^s(\Omega; \R^n) 
\end{align}
is well-defined and smooth. 

By Theorem~\ref{comp_C2}, we may conclude that the map 
\begin{multline}
	H^{s+3}(\R^n ; \R^{n\times n}_{\sym}) \times H^{s+2}(\R^n  ; \R^n) \times B_{X^{s+5/2}(\R^{n-1}; \R)}(0,\delta) \ni  (\mathcal{T},\mathfrak{f}, \eta) \mapsto (\mathfrak{f}\circ \mathfrak{F}, \mathcal{T} \circ \mathfrak{F} \rvert_{\Sigma_b}) \\
	\in H^s(\R^{n} ;\R^{n}) \times H^{s+1/2}(\R^{n-1}; \R^n)
\end{multline}
is well-defined and $C^2$. 

Finally, following the same calculations as Theorem 7.3 in \cite{leonitice}, we find that 
\begin{align}
	\R \times \mathcal{U}^s_\delta \ni (\gamma, u, p, \eta)  \mapsto u \cdot \mathcal{N} + \gamma\p_1 \eta - \int_0^b \mathcal{J} \diverge u (\cdot,x_n) \; dx_n \in H^{s+3/2}(\R^{n-1}; \R) \cap \dot{H}^{-1}(\R^{n-1}; \R) 
\end{align}
is well-defined and smooth. Combining the aforementioned results then shows that the map $\Xi: \mathcal{E}^s \times \mathcal{U}^s_{\delta,M} \to \mathcal{Y}^s$ is well-defined and $C^2$.

Next we note that by the form of the matrix $\mathcal{A}: \Omega \to \R^{n \times n}$ defined in Section \ref{sec:flatten}, the nonlinear terms in the map $\Xi$ are either products between standard Sobolev functions or products between specialized Sobolev and standard Sobolev functions. The same is true for $D\Xi$ and $D^2 \Xi$. Then by utilizing this observation and the ninth item of Theorem~\ref{thm: X^s}, we may conclude that the restriction of the solution map $\Xi \rvert_{ \mathcal{E}^s \times U^s_{\delta,M}} : \mathcal{E}^s \times U^s_{\delta,M} \to \mathcal{Y}^s$ is $C^2_b(\mathcal{E}^s \times U^s_{\delta,M}; \mathcal{Y}^s)$. Furthermore, since we assume that $\alpha \in (0,1)$ and $\alpha$ only appears in the linear terms of the last component of $\Xi$, we may conclude that the $C^2_b$ norm of $\Xi \rvert_{ \mathcal{E}^s \times U^s_{\delta,M}}$ is independent of $\alpha$.

The second item and in particular \eqref{eq: mathfrak X norm} follows from a similar set of arguments as above and the observation that $\alpha$ does not appear on the right hand side of \eqref{eq: Xi alpha}.
\end{proof}

\subsection{Solvability of the flattened system \eqref{eq: main flattened}}\label{sec: main flattened}
Now we are ready to construct solutions to \eqref{eq: main flattened} by using the implicit function theorem.

\begin{proof}[Proof of Theorem~\ref{thm:main1}]
We first consider the case with surface tension, $\sigma > 0$ and $n \ge 2$. Let $\delta$ be the minimum of the $\delta_1 > 0$ from Theorem~\ref{thm:smoothnessnonlinear}, $\delta_* > 0$ from the third item of Theorem~\ref{thm: flattening maps est}, and $\delta_A > 0$ in \eqref{eq: beta coercive}. We fix $M > 0$ and consider the open subset $U^s_{\delta,M}$ of $\mathcal{X}^s$ defined via \eqref{eq: ball for solution}. Using Proposition~\ref{prop:embed X} and standard Sobolev embedding, any open subset of $U^s_{\delta,M}$ containing $(0,0,0)$ satisfies the first assertion of the theorem. This proves the first item.

To prove the remaining items, we consider the Hilbert space $\mathcal{E}^s$ defined in \eqref{eq: E^s} and the solution map $\Xi : \mathcal{E}^s \times U^s_{\delta,M} \to \mathcal{Y}^s$ defined in \eqref{eq: Xi}. By Theorem~\ref{thm:smoothnessnonlinear}, the map $\Xi$ is well-defined and $C^2$. By the product structure of $\mathcal{E}^s \times U^s_{\delta,M}$, we can define $D_1 \Xi : \mathcal{E}^s \times U^s_{\delta,M} \to \mathcal{L}(\mathcal{E}^s ; \mathcal{Y}^s )$ and $D_2 \Xi : \mathcal{E}^s \times U^s_{\delta,M} \to \mathcal{L}(\mathcal{X}^s; \mathcal{Y}^s )$ to be the derivatives of $\Xi$ with respect to $\mathcal{E}^s$ and $U^s_{\delta,M}$, respectively. Note that by the second item of Theorem~\ref{thm: flattening maps est}, we have $D_2\mathfrak{S}_b(0,0) = 0$ and $D_2\Lambda_\Omega(0,0) = 0$. Therefore, for any $\alpha \in \R, \gamma \in \R$, $\Xi(\alpha, \gamma,0,0,0,0, 0,0,0) = (0, 0, 0, 0, 0)$ and $D_2 \Xi (\alpha, \gamma, 0, 0,0,0,0, 0,0,0)(u,p,\eta) = \Upsilon_{\alpha, \beta, \gamma, \sigma}(u,p,\eta)$ where $\Upsilon_{\alpha, \beta, \gamma, \sigma}$ is defined in \eqref{eq:upsilon}. By Theorem~\ref{thm:upsiloniso}, for every $\alpha_* > 0$ and $\gamma_* \neq 0$ the map $D_2 \Xi (\alpha_*, \gamma_*, 0, 0,0,0,0, 0,0,0)$ is a linear isomorphism. Thus, by the implicit function theorem there exists an open sets $\mathcal{U}(\alpha_*, \gamma_*) \subseteq \mathcal{E}^s$ and $\mathcal{O}(\alpha_*, \gamma_*) \subseteq U^s_{\delta,M}$ such that $(\alpha_*, \gamma_*, 0, 0, 0, 0) \in \mathcal{U}(\alpha_*, \gamma_*)$, $(0,0,0) \in \mathcal{O}(\alpha_*, \gamma_*)$, and there exists a $C^1$ Lipschitz map $\varpi_{\alpha_*, \gamma_*} : \mathcal{U}(\alpha_*, \gamma_*) \to \mathcal{O}(\alpha_*, \gamma_*) \subseteq U^s_{\delta,M}$ such that $\Xi(\alpha, \gamma, \mathcal{T}, T, \mathfrak{f}, f, \varpi_{\alpha_*, \gamma_*}(\alpha, \gamma, \mathcal{T}, T, \mathfrak{f}, f) ) = (0, 0, 0, 0, 0)$ for all $(\alpha, \gamma, \mathcal{T}, T, \mathfrak{f}, f) \in \mathcal{U}(\alpha_*, \gamma_*)$. Moreover, $(u,p,\eta) = \varpi_{\alpha_*, \gamma_*}(\alpha, \gamma, \mathcal{T}, T, \mathfrak{f}, f)$ is the unique solution to $\Xi(\gamma, \mathcal{T}, T, \mathfrak{f}, f, u,p,\eta ) = (0, 0, 0, 0, 0)$ in $\mathcal{O}(\alpha_*, \gamma_*)$. 

Next, we define the open sets 
\begin{align}\label{eq: final open}
	 \mathcal{U}^s = \bigcup_{\alpha_* \in \R^+, \gamma_* \in \R \setminus \{0\}}\mathcal{U}(\alpha_*, \gamma_*) \subseteq \mathcal{E}^s \; \text{and} \; \mathcal{O}^s = \bigcup_{\alpha_* \in \R^+, \gamma_* \in \R \setminus \{0\}} \mathcal{O}(\alpha_*, \gamma_*) \subseteq U^s_{\delta,M}.
 \end{align}
 We note that by construction, $\R^+ \times (\R \setminus \{0\}) \times \{ 0\} \times \{ 0 \} \times \{ 0 \} \times \{0 \} \subset \mathcal{U}^s$. Furthermore, for every 
$(\alpha, \gamma,  \mathcal{T},T,\mathfrak{f},f) \in \mathcal{U}^s$, there exists an $\alpha_* > 0, \gamma_* \in \R \setminus \{0\}$ for which $(\alpha, \gamma,  \mathcal{T},T,\mathfrak{f},f) \in \mathcal{U}(\alpha_*, \gamma_*)$ and $(u,p,\eta) = \varpi_{\alpha_*, \gamma_*}(\alpha, \gamma,  \mathcal{T},T,\mathfrak{f},f) \in \mathcal{O}(\alpha_*, \gamma_*)$. By the observation above and the implicit function theorem, the map $\overline{\varpi}: \mathcal{U}^s \to \mathcal{O}^s$ defined via  $\overline{\varpi}(\alpha, \gamma,  \mathcal{T},T,\mathfrak{f},f) = \varpi_{\alpha_*, \gamma_*}(\alpha, \gamma,  \mathcal{T},T,\mathfrak{f},f)$, where $\alpha_* > 0, \gamma_* \in \R \setminus \{ 0 \}$ is such that $(\alpha, \gamma,  \mathcal{T},T,\mathfrak{f},f) \in \mathcal{U}(\alpha_*, \gamma_*)$, is well-defined, $C^1$, and locally Lipschitz. This proves the remaining items for $\sigma > 0$ and $n \ge 3$. 

To prove the remaining items in the case without surface tension and $n=2$, we argue along the same lines but use the second item of Theorem~\ref{thm:upsiloniso} instead of the first and use the isomorphism $\Upsilon_{\alpha, \beta, \gamma,0}$. 
\end{proof}

\subsection{The solutions to \eqref{eq: main unflattened} as $\alpha \to 0$} \label{sec: alpha to zero}

\begin{proof}[Proof of Theorem~\ref{thm: alpha to zero}]

We first note that Theorem~\ref{thm: X^s} shows that the space $Z = H^{s+2}(\Omega ; \R^n) \times H^{s+1}(\Omega; \R) \times X^{s+5/2}(\R^{n-1}; \R)$ is a separable Hilbert space, and therefore the ball $B_{Z} (0,M)$ is metrizable in the weak topology for any $M > 0$ (see Theorem 3.29 in \cite{brezis}).

Let $\delta > 0$ be the same as in the proof for Theorem~\ref{thm:main1} above and for a fixed $M > 0$ consider the solution map $\mathfrak{X} : \mathcal{E}^s \times U^s_{\alpha,\delta,M} \to \mathcal{Z}^s$ defined via \eqref{eq: Xi alpha}. We note that if $\alpha \in (0,1)$, the last item of Theorem~\ref{thm:upsiloniso} and the second item of Theorem~\ref{thm:smoothnessnonlinear} implies that $\mathfrak{X}$ satisfies the $\alpha$-independent estimate \eqref{eq: implicit ests}. Furthermore, the arguments presented above in the proof of Theorem~\ref{thm:main1} show that $\mathfrak{X}$ also satisfies the rest of the requirements of Theorem~\ref{thm: implicitFT}. Thus by applying Theorem~\ref{thm: implicitFT}, for every $\gamma_* \in \R \setminus \{0\}$, there exists an $\alpha_*$-independent open set $V(\gamma_*) \subseteq (\R \setminus \{0\}) \times H^{s+3}(\R^n ; \R^{n \times n}_{\sym}) \times H^{s+\frac{1}{2}}(\R^{n-1}  ; \R^{n\times n}_{\sym}) \times H^{s+2}(\R^n  ; \R^n) \times H^s(\R^{n-1}  ; \R^n)$ and an $\alpha_*$-independent constant $M > 0$ such that for every $\alpha_* \in (0,1)$, there exists an open set $\mathcal{O}(\alpha_*, \gamma_*) \subseteq U^s_{\alpha, \delta,M}$ such that $(\alpha, \gamma_*, 0, 0, 0, 0) \in (0,1) \times  V(\gamma_*)$, $(0,0,0) \in \mathcal{O}(\alpha_*, \gamma_*)$, and there exists a $C^1$ Lipschitz map $\varpi_{\alpha_*,\gamma_*} : (0,1) \times  V(\gamma_*) \to \mathcal{O}(\alpha_*, \gamma_*) \subseteq U^s_{\alpha,\delta,M}$ such that $\mathfrak{X}(\alpha, \gamma, \mathcal{T}, T, \mathfrak{f}, f, \varpi_{\alpha_*,\gamma_*}(\alpha, \gamma, \mathcal{T}, T, \mathfrak{f}, f) ) = (0, 0, 0, 0)$ for all $(\alpha, \gamma, \mathcal{T}, T, \mathfrak{f}, f) \in (0,1) \times  V(\gamma_*)$. Moreover, $(u_\alpha,p_\alpha,\eta_\alpha) = \varpi_{\alpha_*,\gamma_*}(\alpha, \gamma, \mathcal{T}, T, \mathfrak{f}, f)$ is the unique solution to $\mathfrak{X}(\gamma, \mathcal{T}, T, \mathfrak{f}, f, u,p,\eta ) = (0, 0, 0, 0)$ in $\mathcal{O}(\alpha_*, \gamma_*)$ and satisfies $\sup_{\alpha \in (0,1)} \norm{(u_\alpha,p_\alpha,\eta_\alpha)}_{\mathcal{X}_\alpha^s} \le M$.  

\begin{figure}[!ht]
	\includegraphics[scale=0.8]{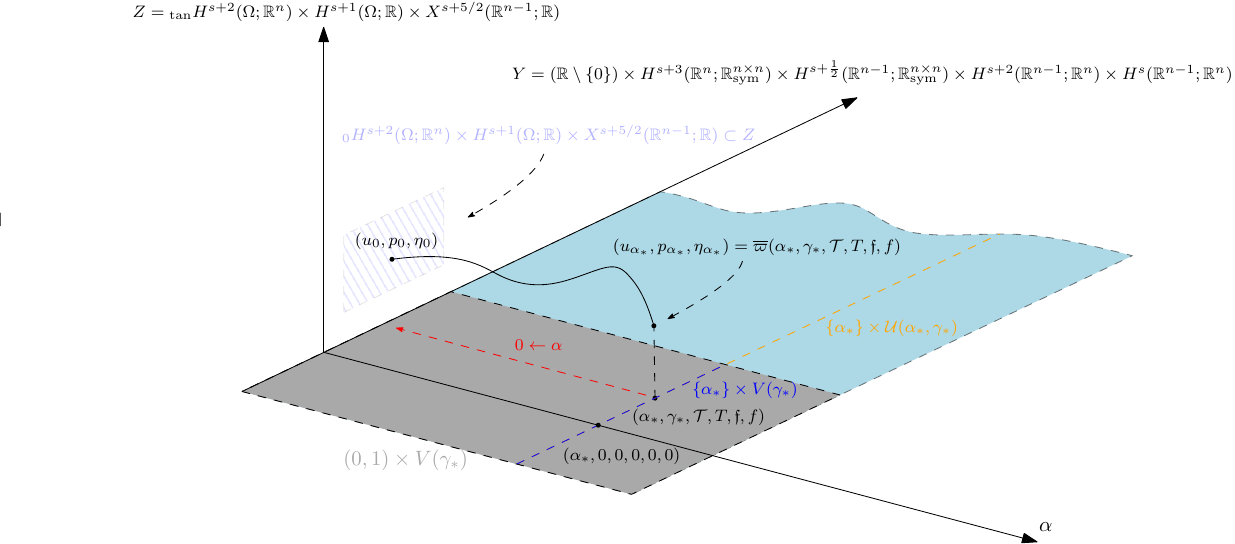}
	\caption{A toy picture of the schematics of the proof}
	\label{fig:alphatozero}
\end{figure}

We now fix $\gamma_* \in \R \setminus \{0\}$, $(\gamma_*, \mathcal{T},T,\mathfrak{f},f) \in V(\gamma_*)$ and consider the function $f: (0,1) \to B_Z(0,M)$ defined via $f(\alpha) = (u_\alpha, p_\alpha, \eta_\alpha) = \varpi_{\alpha_*,\gamma_*}(\alpha, \gamma, \mathcal{T}, T, \mathfrak{f}, f)$. Let $\{\alpha_j\}_{j=1}^\infty \subset (0,1)$ be any sequence such that $\alpha_j \to 0$ as $j \to \infty$ and let $\{\alpha_{j_k}\}_{k=1}^\infty \subset \{\alpha_j\}_{j=1}^\infty$ be any subsequence of the original sequence. We note that since $\sup_{k} \norm{f(\alpha_{j_k})}_{\mathcal{X}^s} < \infty$, there exists a further subsequence $\{ f(\alpha_{j_{k_l}}) \}_{l=1}^\infty$ such that $ f(\alpha_{j_{k_l}})$ converges weakly to $f_0 \in B_Z(0,M)$ in $Z$ as $l \to \infty$.

By the sixth item of Theorem~\ref{thm: X^s}, we may decompose any element $\eta_l \in X^{s+5/2}(\R^{n-1} ; \R)$ as $\eta = \eta_{\text{low}} + \eta_{\text{high}}$, where $\eta_{\text{low}} \in C^\infty_0(\R^{n-1}; \R)$ and $ \eta_{\text{high}} \in H^{s+5/2}(\R^{n-1};\R)$ satisfy the bounds \eqref{eq: low high Xs estimates}. Thus we may define for every $l \ge 0$,
\begin{multline}
	 g_l = \left(u_{\alpha_{j_{k_l}}}, p_{\alpha_{j_{k_l}}}, \left(\eta_{\alpha_{j_{k_l}}}\right)_{\text{high}} \right) \in H^{s+2}(\Omega; \R^n) \times H^{s+1}(\Omega ; \R) \times H^{s+5/2}(\R^{n-1} ; \R), \\  h_l = \left(\eta_{\alpha_{j_{k_l}}} \right)_{\text{low}} \in C^\infty_0(\R^{n-1} ; \R),
\end{multline} 
satisfying $ \sup_{l} ( \norm{g_l}_{H^{s+2}(\Omega) \times H^{s+1}(\Omega) \times H^{s+5/2}(\R^{n-1})} + \norm{h_l}_{C_b^k(\R^{n-1})}) < \infty$ for any $k \ge 1$. 

Now consider the nested sequence of compact sets $\{E_m\}_{m=1}^\infty$ defined via $E_m = [-m, m]^{n-1} \times [0,b] \subset \Omega$. On $E_1$, we consider the restriction of the sequence $\{ g_l \rvert_{E_1} \}_{l=1}^\infty$ and $\{ h_l \rvert_{E_1} \}_{l=1}^\infty$ and note that since the restriction operator is continuous, $ g_l \rvert_{E_1}$ belongs to $\mathscr{B}(E_1) := H^{s+2}(E_1; \R^n) \times H^{s+1}(E_1; \R) \times H^{s+5/2}([-1,1]^{n-1} ; \R)$ and $h_l \rvert_{E_1}$ belongs to $C^\infty([-1,-1]^{n-1} ; \R)$. Furthermore, we have $\sup_l (\norm{g_l \rvert_{E_1}}_{\mathscr{B}(E_1)} + \norm{h_l \rvert_{E_1}}_{C^k_b([-1,-1]^{n-1})} ) < \infty$ for any $k \ge 1$. Since $\mathscr{B}(E_1)$ is also a Hilbert space, we may conclude that up to passing to a subsequence, there exists $g_{1,l} \in \mathscr{B}(E_1)$ for which $\{g_l \rvert_{E_1}) \}_{l=1}^\infty$ converges weakly to $g_{1,l}$ in $ \mathscr{B}(E_1)$ as $l \to \infty$.

We next note that since for any $r \ge 0$, the identity operator $id: H^{s+r}(K; V) \to H^{s}(K; V)$ is compact for any compact Lipschitz domain $K \subseteq \R^{d}$ and any finite dimensional vector space $V$. Furthermore, for sufficiently small $\ve > 0$ we also have $H^{s+k-\ve}(E_1; V) \hookrightarrow C^k(E_1;V)$ by standard Sobolev embedding, therefore we may conclude that up to passing to a subsequence, $\{ g_l \rvert_{E_1}) \}_{l=1}^\infty$ converges strongly to $(g_{1,l})$ in $C_b^2(E_1;\R^n) \times C_b^1(E_1; \R) \times C_b^1 ([-1,1]^{n-1}; \R)$ as $l \to \infty$. By the Arzel\`{a}-Ascoli theorem, we may also conclude that up to passing to a subsequence, there exists an $h_{1,l} \in C^1([-1,1]^{n-1} ; \R)$ for which $h_l \rvert_{E_1}\to h_{1,l}$ strongly in $C^1_b([-1,1]^{n-1}; \R)$.

Now we consider the subsequences of the original sequences $\{g_l\}_{l=1}^\infty, \{h_l\}_{l=1}^\infty$ constructed in the previous step that converge strongly to $g_{1,l}$ and $h_{1,l}$ respectively. We note that we may repeat the same argument as above to obtain a further subsequence converging strongly to some $g_{2,l}$ and $h_{2,l}$ in $C_b^2(E_2;\R^n) \times C_b^1(E_2; \R) \times C_b^1 ([-2,2]^{n-1}; \R)$ and $C^1_b([-2,-2]^{n-1} ; \R)$, respectively. Furthermore, $g_{2,l}, h_{2,l}$ must coincide with $g_{1,l}, h_{1,l}$ respectively on $E_1$. 

Thus, by continuing this procedure ad infinitum and employing a standard diagonal argument, we may upon relabeling identify a subsequence  $\{(u_{\alpha_{j_{k_{l}}}}, p_{\alpha_{j_{k_{l}}}}, (\eta_{\alpha_{j_{k_{l}}}})_{\text{low}}, (\eta_{\alpha_{j_{k_{l}}}})_{\text{high}}) \}_{l=1}^\infty \subseteq \{(u_{\alpha_{j_{k}}}, p_{\alpha_{j_{k}}}, (\eta_{\alpha_{j_{k}}})_{\text{low}}, (\eta_{\alpha_{j_{k}}})_{\text{high}}) \}_{k=1}^\infty$ converging strongly to some $(u_0, p_0, (\eta_0)_{\text{low}}, (\eta_0)_{\text{high}})$ in $C_b^2(\Omega;\R^n) \times C_b^1(\Omega; \R) \times C_b^1 (\R^{n-1}; \R) \times C_b^1 (\R^{n-1}; \R) $. Since $\beta \in \R^{n \times n}$ is assumed to satisfy \eqref{eq: beta coercive}, we have 
\begin{multline}
	u \in \alphaHzeros{s+2}(\Omega; \R^n) \hookrightarrow \Hzerostan{s+2}(\Omega; \R^n)  \; \text{and} \; [\beta u]' = 0 \; \text{on} \; \Sigma_0 \\ \implies \beta u \cdot u = [\beta u]' \cdot u' + [\beta u]_n \cdot u_n = 0 \; \text{on} \; \Sigma_0 \implies u = 0 \; \text{on} \; \Sigma_0.
\end{multline}
Therefore as $\alpha \to 0$, by passing to the limit in we may conclude that $(u_0, p_0, \eta_0)$ (where $\eta_0 = (\eta_0)_{\text{low}} + (\eta_0)_{\text{high}}$) solve the incompressible Navier-Stokes system \eqref{eq: noslip} with the no-slip condition on $\Sigma_0$ classically. 

Finally, by invoking the uniqueness part of Theorem~\ref{thm: no-slip thm}, we may conclude that $(u_0, p_0, \eta_0) \in {}_0H^{s+2}(\Omega; \R^n) \times H^{s+1}(\Omega; \R) \times X^{s+5/2}(\R^{n-1} ; \R)$. Thus, every subsequence $\{f(\alpha_{j_{k}})\}_{k=1}^\infty= \{ (u_{\alpha_{j_{k}}}, p_{\alpha_{j_{k}}},  \eta_{\alpha_{j_{k}}})\}_{k=1}^\infty$ of the sequence $\{f(\alpha_j)\}_{j=1}^\infty = \{ (u_{\alpha_j}, p_{\alpha_j},\eta_{\alpha_j}) \}_{j=1}^\infty$ has a further subsequence converging weakly to $f_0 := (u_0, p_0, \eta_0)$ in $Z$. Since the ball $B_Z(0,M)$ is metrizable in the weak topology for any $M > 0$, we may then conclude that $(u_\alpha, p_\alpha, \eta_\alpha) \weakar (u_0, p_0, \eta_0)$ weakly in $H^{s+2}(\Omega; \R^n) \times H^{s+1}(\Omega; \R) \times X^{s+5/2}(\R^{n-1} ; \R)$ as $\alpha \to 0$. 
\end{proof}

\appendix

\section{Analysis tools}
In this section we record some tools utilized in our analysis.

\subsection{The incompressible Navier Stokes system with no-slip conditions}

In this subsection we record a result from \cite{leonitice} adapted to the flattening map \eqref{eq:flattening}. 

\begin{thm}\label{thm: no-slip thm}
Suppose that either $\sigma > 0$ and $n \ge 2$ or $\sigma = 0$ and $n= 0$. Assume that $\N \ni s \ge 1 + \tfloor{n/2}$, and let $\mathcal{X}^s$ be as defined by \eqref{eq: solution space X^s}, $L_{\Omega_b}$ and $S_b$ be as defined as in \eqref{eq: L_Omega and S_b}. Then there exists open sets
\begin{align} 
	\mathcal{V}^s \subset \R^+ \times (\R \setminus \{0\}) \times H^{s+2}(\R^{n} ; \R^{n \times n}_{\sym}) \times H^{s+\frac{1}{2}}(\R^{n-1} ; \R^{n\times n}_{\sym}) \times H^{s+1}(\R^{n-1}; \R^n) \times H^s(\R^{n-1} ; \R^n) 
\end{align}
and $\mathcal{O}^s \subset \mathcal{X}^s$ such that for each $(\gamma, \mathcal{T}, T, \mathfrak{f}, f)$, there exists a unique $(u,p,\eta) \in \mathcal{O}^s$ classically solving \eqref{eq: noslip} with the flattening map defined via \eqref{eq:flattening}.
\end{thm}
\begin{proof}
This essentially follows from the work in \cite{tice} for $\kappa = 0$ and the proof for the third item of Theorem 1.2 in the same paper, though we note that the flattening map $\mathfrak{F}$ defined via \eqref{eq:flattening} is slightly different from the one employed in \cite{tice}, which is given by $	 \mathfrak{G}_\eta (x', x_n) = x + \frac{x_n \eta(x')}{b} e_n.$  Though, since $\mathfrak{F}$ and $\mathfrak{G}_\eta$ are both diffeomorphisms for $\eta \in X^{s+5/2}(\R^{n-1};\R)$ such that $\norm{\eta}_{X^{s+5/2}}$ is sufficiently small, and both maps satisfy the $C^1$ $\omega$-lemma, the flattening map $\mathfrak{F}$ can be used in the arguments in \cite{tice} to arrive at the desired result.
\end{proof}



\subsection{Tangential Fourier multipliers}
In this subsection we record a few essential results concerning bounded translation invariant operators on Sobolev spaces and  tangential Fourier multipliers. Recall that the reflection operator $\delta_{-1}: \mathcal{F}(\R^{d_1};\C^{d_2}) \to  \mathcal{F}(\R^{d_1};\C^{d_2})$ is defined via $\delta_{-1} f(x) = f(-x)$.

The first proposition gives a characterization of bounded linear maps on Sobolev spaces that commute with tangential multipliers. 
\begin{prop}\label{prop: tangential m}
Let $\F \in \{\R , \C\}$, $s,t \in \R$, and $T \in \mathcal{L}(H^s(\R^d; \F) ; H^t(\R^d; \F))$. The following are equivalent. 
\begin{enumerate}
	 \item $T$ commutes with translation operators. 
	 \item There exists a measurable function $\omega: \R^d \to \C$ such that $\overline{\omega} = \delta_{-1} \omega$ if $\F = \R$, $Tf = \mathscr{F}^{-1}[\omega \mathscr{F}[f] ]$, and
	 \begin{align}\label{eq: esssup multiplier est}
		s_{\omega} = \esssup \{ (1+ \abs{\xi}^2)^{t-s} \abs{\omega(\xi)} : \xi \in \R^d \} < \infty.
	 \end{align}
	 Furthermore, we have the estimate $\norm{T}_{\mathcal{L}(H^s; H^t)} \le s_\omega \le 2 \norm{T}_{\mathcal{L} (H^s; H^t)}$.
\end{enumerate}
\end{prop}
\begin{proof}
This follows from Proposition A.10 in \cite{noahtice}.
\end{proof}

\begin{lem}\label{lem: bessel}
	Suppose $\N \ni d, k \ge 1$ and let $\F \in \{\R, \C\}$.   Let $s \ge 0, M > 0$ and $U = \R^d \times (0,b)$. Define the operator $\mathfrak{J}^{s}_M : H^{s}(U; \F^k) \to H^s(U; \F^k)$ via $\mathfrak{J}^{s}_M f (\cdot, x_n) = \mathscr{F}^{-1} [\chi_{B(0,M)}(\cdot) (1+ \abs{\cdot}^2)^{s/2} \mathscr{F}[f(\cdot,x_n )]]$. Then $\mathfrak{J}^{s}_M$ is well-defined, and for all $f \in H^s(U; \F^k)$ and $t \in \R$ such that $t \le s$, we have the $M$-independent estimate $\norm{\mathfrak{J}^{s}_M f}_{L^2(U)} \lesssim_{d,s,b} \norm{\mathfrak{J}_M^{s-t} f}_{H^t(U)}$.
\end{lem}
\begin{proof}
Using Corollary A.7 in \cite{leonitice}, we estimate
\begin{equation}
	\norm{\mathfrak{J}^{s}_M f}_{L^2(U)}^2  
	= \int_0^b \norm{\mathfrak{J}^{s-t}_Mf(\cdot,x_n)}_{H^s(\R^d)}^2 \; dx_n \lesssim_{d,s,b} \norm{\mathfrak{J}^{s-t}_Mf}_{H^s(U)}^2. 
\end{equation}
\end{proof}

We conclude this subsection by recalling some results on tangential multipliers from \cite{noahtice}.

\begin{lem} \label{defn: tangential m}
Suppose $\N \ni d, k \ge 1$, $\F \in \{ \R , \C \}$, and let $\omega \in L^\infty(\R^d; \C^{k \times k})$ be a Fourier multiplier such that if $\F = \R$, then $\overline{\omega} = \delta_{-1} \omega$. Let $U = \R^d \times (0,b)$ and $s \ge 0$.

\begin{enumerate}
	\item We define the tangential Fourier multiplier on $L^2(\R^d ; \K^k)$ as the operator $M_\omega: L^2(\R^d ; \K^k) \to L^2(\R^d ; \K^k)$ defined via $M_\omega f(\cdot) = \mathscr{F}^{-1}[\omega \mathscr{F}[f(\cdot)]]$.
	 \item We define the tangential Fourier multiplier on $H^s(U ; \F^k)$ as the operator $M_\omega: H^s(U ; \F^k) \to H^s(U; \F^k)$ defined via $M_\omega f(\cdot, x_n) = \mathscr{F}^{-1}[\omega \mathscr{F}[f(\cdot,x_n)]]$ for all $x_n \in (0,b)$. Then $M_\omega$ is well-defined and satisfies the estimate 
	 \begin{align}
		\norm{M_\omega f}_{H^s(U)} \lesssim_{d,s} \norm{\omega}_{L^\infty(U)} \norm{f}_{H^s(U)} \; \text{for all} \; f \in H^s(U; \F^k).
	 \end{align}
	 Furthermore, if $s > 1/2$ and $\Sigma \in \{\Sigma_b, \Sigma_0\}$ then
	 \begin{align}\label{eq: trace multiplier commute}
		  \Tr_\Sigma M_\omega f = M_\omega \Tr_\Sigma f \; \text{for all} \; f \in  H^s(U; \F^k).
	 \end{align}
	 \item We extend the notion of tangential Fourier multipliers to $( {}_0 H^1(U; \F^k))^{\overline{*}}$ by defining the operator $M_\omega: ( {}_0 H^1(U; \F^k))^{\overline{*}} \to ( {}_0 H^1(U; \F^k))^{\overline{*}}$ using the action of the anti-linear functional acting on test functions via 
	 \begin{align}
		 \langle M_\omega F, \varphi \rangle_{({}_0 H^1)^{\overline{*}}, {}_0 H^1} =  \langle  F, M_{\overline{\omega}}\varphi \rangle_{({}_0 H^1)^{\overline{*}}, {}_0 H^1} \; \text{for all} \; \varphi \in {}_0 H^1(U; \F^k), F \in ( {}_0 H^1(U; \F^k))^{\overline{*}}.
	 \end{align}
	 Then $M_\omega$ is well-defined and satisfies the estimate 
	 \begin{align}
		 \norm{M_\omega F}_{({}_0 H^1(U))^{\overline{*}}} \lesssim \norm{\omega}_{L^\infty(U)} \norm{F}_{({}_0 H^1(U))^{\overline{*}}} \; \text{for all} \; F \in ( {}_0 H^1(U; \F^k))^{\overline{*}}.
	 \end{align}
\end{enumerate}

\end{lem}
\begin{proof}
This follows from Lemma A.12 in \cite{noahtice}.
\end{proof}

\subsection{A parameter dependent implicit function theorem}\label{sec: parameter IFT}

In this subsection we aim to prove a variant of the implicit function theorem, for functions of the form $f_\alpha(\cdot) = f(\alpha,\cdot)$ where $\alpha \in \R$ and where the underlying spaces are allowed to vary with the parameter $\alpha$. First, we establish a variant of the inverse function theorem.
\begin{thm}\label{thm: inverseFT}
	Let $X,Y$ be Banach spaces and suppose $\{X_\alpha \}_{\alpha \in (0,1)} \subset X$ is a one-parameter family of closed subspaces of $X$. For any $\alpha \in (0,1)$, suppose $f_\alpha \in C^2(U_\alpha ; Y)$ for a non-empty open set $U_\alpha \subseteq X_\alpha$ containing $0$,  $f_\alpha( 0 ) = 0$, and $D f_\alpha(0) \in \mathcal{L}(X_\alpha; Y)$ is a linear homeomorphism. Furthermore, we suppose there exists constants $\ve > 0$, $C > 2$ such that $B_{X_\alpha}(0,\ve)  \subseteq U_\alpha$ and 
	\begin{align}\label{eq: inverse fn est}
		\sup_{\alpha \in (0,1)} \left( \norm{Df_\alpha(0)}_{\mathcal{L}(X_\alpha;Y)} + \norm{Df_\alpha(0)^{-1}}_{\mathcal{L}(Y;X_\alpha)}  + \sup_{z \in B_{X_\alpha}(0,\ve)} \norm{D^2 f_\alpha(z)}_{\mathcal{L}^2(X ; Y)} \right) \le C.
	\end{align}
	Then the following hold.

\begin{enumerate}
		 \item There exists a $\delta > 0$ such that for all $\alpha \in (0,1)$, there exists an open set $V_\alpha$ such that $
			  B_{X_\alpha} \left(0,  \frac{\delta}{3C^2}\right) \subset V_\alpha \subset B_X(0,\delta)$, $f_\alpha(V_\alpha) = B_Y\left(0, \frac{\delta}{2C} \right)$, and the restriction $f_\alpha \rvert_{V_\alpha}: V_\alpha \to f_\alpha(V_\alpha)$ is a bi-Lipschitz homeomorphism. 
		 \item There exists a constant $K > 0$ such that 
		 \begin{align}\label{eq: f C^0_b est alpha ind}
			 \sup_{\alpha \in (0,1)} \left( \norm{f_\alpha}_{C^0_b(V_\alpha;Y)}  + \norm{f_{\alpha}^{-1}}_{C^0_b (f_\alpha(V_\alpha); X_\alpha)} \right) \le K.
		 \end{align}
		 \item We have $f_\alpha \in C^1_b(V_\alpha ;Y)$ and $f_\alpha^{-1} \in C^1_b(f_\alpha(V_\alpha) ; X_\alpha)$. Furthermore, $Df_\alpha(x) \in \mathcal{L}(X_\alpha ; Y)$ is a linear homeomorphism for every $x \in V_\alpha$ and $D f_\alpha^{-1}(y) \in \mathcal{L}(Y; X_\alpha)$ is a linear homeomorphism for every $y \in f_\alpha(V_\alpha)$, and the two are related via 
		 \begin{align}\label{eq: inverse dervative}
			 D f_\alpha^{-1}(y) = ( D f_\alpha(f_\alpha^{-1}( y)))^{-1}
		 \end{align}
		 for every $y \in f_\alpha(V_\alpha)$. 
		 \item If $f \in C^k(U_\alpha; Y)$ for some $k \ge 2$, then $f \in C^k(V_\alpha; Y)$ and $f \in C^k(f_\alpha(V_\alpha) ; X_\alpha)$. 
	\end{enumerate}
	\end{thm}
\begin{figure}[!ht]
	\includegraphics[scale=0.8]{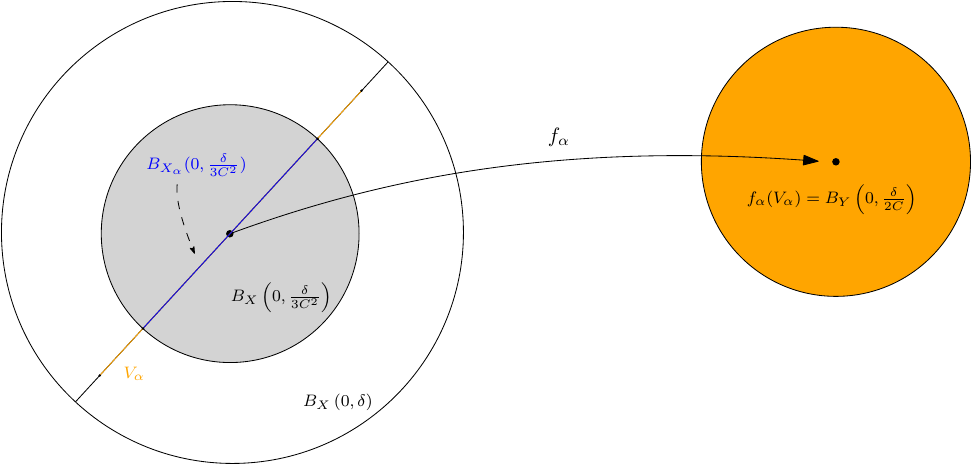}
	\caption{A depiction of the sets $V_\alpha,f_\alpha(V_\alpha)$ and the $\alpha$-independent ``core'' $B_X(0,\delta/(3C)^2)$}
\end{figure}
\begin{proof}
It suffices to only prove the first two items as the third and fourth items follow from the standard inverse function theorem applied to each $\alpha \in (0,1)$, see Theorem 2.5.2 in \cite{abraham}. 

For any $\alpha \in (0,1)$ we consider the function $F_\alpha: U_\alpha \to X$ defined via $F_\alpha(x) = x - Df_\alpha(0)^{-1} f_\alpha(x)$. Then for all $\alpha \in (0,1)$, $F_\alpha \in C^2(U_\alpha; X)$, $F_\alpha(0) = 0$, and $DF_\alpha(0) = 0$. Furthermore, by applying the mean value inequality over $B_{X_{\alpha}}(0,\ve)$ and \eqref{eq: inverse fn est} we may conclude that for all $x \in B_{X_{\alpha}}(0,\ve)$,
\begin{align}\label{eq: first derivative est F}
	\norm{DF_\alpha(x)}_{\mathcal{L}(X)} \le \sup_{t \in [0,1]} \norm{D^2 F_\alpha(tx)}_{\mathcal{L}^2(X_\alpha)} \norm{x}_{X_\alpha} \le \norm{Df_\alpha(0)^{-1}}_{\mathcal{L}(Y;X_\alpha)} \norm{D^2 f_\alpha(x)}_{\mathcal{L}^2(X_\alpha;Y)} \norm{x}_{X_\alpha} \le C^2 \norm{x}_{X_\alpha}.    
\end{align}
By \eqref{eq: inverse fn est}, we may  choose $\delta > 0$ sufficiently small and independent of $\alpha$ for which $\delta < (2C^2)^{-1}$ and $\norm{DF_\alpha(x)}_{\mathcal{L}(X)} \le \frac{1}{2}$ for all $x \in B_{X_\alpha}(0,\delta) \subseteq E_\alpha$. By the mean value inequality, we also have 
\begin{align}\label{eq: g one half est}
	\norm{F_\alpha(x) - F_\alpha(y)}_{X_\alpha} \le \norm{x-y}_{X_\alpha} \sup_{z \in B_{X_\alpha}(0,\delta)} \norm{D F_\alpha( z)}_{\mathcal{L}(X_\alpha)} \le \frac{1}{2} \norm{x-y}_{X_\alpha} \; \text{for all} \; x,y \in B_{X_\alpha}(0,\delta).
\end{align}
Fix $y \in B_Y(0,\delta(2C)^{-1})$ and define the function $h_\alpha :B_{X_\alpha}[0,\delta] \to B_{X_\alpha}(0,\delta) \subset B_{X_\alpha}[0,\delta]$ via $h_\alpha(x) = Df_\alpha(0)^{-1}(y + Df_\alpha(0) F_\alpha(x))$, where $B_{X_\alpha}[0,\delta]$ denotes the closed ball in $X_\alpha$ with radius $\delta$. To check that the map is well-defined, we note that since $F_\alpha(0) = 0$, by the writing $h_\alpha(x) = Df_\alpha(0)^{-1} y + F_\alpha(x)$ and using \eqref{eq: g one half est} we have $	\norm{h_\alpha(x)}_{X_\alpha} \le C\norm{y}_{Y} + \frac{1}{2} \norm{x}_{X_\alpha}   < \delta$ for all $x \in B_{X_\alpha}[0,\delta]$.  This shows that the map is well-defined. Next we note that since $h_\alpha$ and $F_\alpha$ differ by a constant, by the estimate on $DF_\alpha$ we also have $ \norm{ Dh_\alpha(x)}_{\mathcal{L}(X_\alpha)} \le \frac{1}{2}$ for all $x \in B_{X_\alpha}[0,\delta]$, which implies that $h_\alpha$ is a contraction on the complete metric space $B_{X_\alpha}[0,\delta]$. Therefore by the contraction mapping theorem, there exists a unique $x \in B_{X_\alpha}[0,\delta]$ for which $h_\alpha(x) = x$, but since $h_\alpha(B_{X_\alpha}[0,\delta]) \subseteq B_{X_\alpha}(0,\delta)$, we get the inclusion $x = h_\alpha(x) \in B_{X_\alpha}(0,\delta)$. Since $h_\alpha(x) = x$ is equivalent to $f_\alpha(x) = y$, we find that for every $y \in B_Y(0,\delta(2C)^{-1})$ there exists a unique $x \in B_{X_\alpha}(0,\delta)$ such that $f_\alpha(x) = y$. 

Now we define the set $V_\alpha = f_\alpha^{-1}(B_Y(0,\delta(2C)^{-1})) \cap B_X(0,\delta)$, which by the contraction mapping argument above is an open subset of $B_{X_\alpha}(0,\delta) \subset B_X(0,\delta)$. We first note that by \eqref{eq: g one half est}, we have
\begin{multline}\label{eq: f lip est 1}
	\norm{f_\alpha(x) - f_\alpha(y)}_{X_\alpha} \le \norm{Df_\alpha(0)}_{\mathcal{L}(X_\alpha;Y)} \left( \norm{x - y}_{X_\alpha} + \norm{F_\alpha(x) - F_\alpha(y)}_{X_\alpha}\right) \\ \le \frac{3}{2} C \norm{x-y}_{X_\alpha} \; \text{for all} \; x,y \in B_{X_\alpha}(0,\delta).
\end{multline}
In particular, since $f_\alpha(0) = 0$ we have $\norm{f_\alpha(x)}_Y \le 3C/2  \norm{x}_{X_\alpha}$ for all $x \in  B_X(0,\delta)$. This implies the inclusion $f_\alpha(B_{X_\alpha}(0,\delta (3C)^{-2})) \subseteq B_Y(0,\delta(2C)^{-1} )$, and subsequently $B_{X_\alpha}(0,\delta (3C)^{-2}) \subseteq V_\alpha$.

By the contradiction mapping argument above, the restriction $f_\alpha \rvert_{V_\alpha}: V_\alpha \to f_\alpha(V_\alpha) =B_Y(0,\delta(2C)^{-1})$ is invertible. Next we note that for all $x_1, x_2 \in V_\alpha$, we have 
\begin{multline}
	 \norm{x_1 - x_2}_{X_\alpha} \le \norm{F_\alpha(x_1) - F_\alpha(x_2)}_{X_\alpha} + \norm{Df_\alpha(0)^{-1}}_{\mathcal{L}(Y;X)} \norm{f_\alpha(x_1) - f(x_2)}_{X_\alpha} \\ 
	 \le \frac{1}{2} \norm{x_1 - x_2}_{X_\alpha} +  C \norm{f_\alpha(x_1) - f(x_2)}_{X_\alpha}.
\end{multline}
This then implies  
\begin{align}\label{eq: f lip est 2}
	 \norm{f_\alpha^{-1}(y_1) - f_\alpha^{-1}(y_2)}_{X_\alpha} = \norm{x_1 - x_2}_{X_\alpha} \le 2 C \norm{f(x_1) - f(x_2)}_{Y} = 2 C \norm{y_1 - y_2}_{Y},    
\end{align}
for all $y_1, y_2 \in f_\alpha(V_\alpha) =B_Y(0,\delta(2C)^{-1})$. From \eqref{eq: f lip est 1} and \eqref{eq: f lip est 2} we may conclude that $f_\alpha \rvert_{V_\alpha} : V_\alpha \to f_\alpha(V_\alpha)$ is a bi-Lipschitz homeomorphism and the estimate \eqref{eq: f C^0_b est alpha ind} holds.


\end{proof}

Now we are ready to prove a parameter dependent implicit function theorem.

\begin{thm}\label{thm: implicitFT}
Let $X,Y,Z$ be Banach spaces over $\F$ and let $\{ Y_\alpha \}_{\alpha \in (0,1)} \subset Y$ be a one-parameter family of closed subspaces of $Y$. We equip the Cartesian products $X \times Y, X \times Z$ with the $\infty$-norm defined via $\norm{(x,x')}_{X \times X'} = \max\{  \norm{x}_{X}, \norm{x'}_{X'}  \}$, and we equip the Cartesian products $X \times Y_\alpha$ with the norm inherited from $X \times Y$. 

For all $\alpha \in (0,1)$, we suppose $U_\alpha \subseteq X \times Y_\alpha$ is a non-empty open set containing $0$, $f_\alpha \in C^2(U_\alpha ; Z)$, $f_\alpha(0, 0) = 0$, $D_2f_\alpha(0, 0) \in \mathcal{L}(Y; Z)$ is a linear homeomorphism, and there exists a constant $C > 2$ and a non-empty open set $E_\alpha \subseteq U_\alpha$ containing 0 such that
\begin{align}\label{eq: implicit ests}
	\sup_{\alpha \in (0,1)} \left( \norm{Df_\alpha(0,0)}_{\mathcal{L}(X; Y)} + \norm{D f_\alpha(0,0)^{-1}}_{\mathcal{L}(Y ; X)}  + \sup_{(x,y) \in E_\alpha} \norm{D^2 f_\alpha(x,y)}_{\mathcal{L}(X ; Y)}  \right) \le C, 
\end{align}
Then there exists a $\delta_1 > 0$ such that for all $\alpha \in (0,1)$, there exists  $g_\alpha \in C^2_b(B_X(0,\delta_1)  ; Y_\alpha) \cap C_b^{0,1}(B_X(0,\delta_1) ; Y_\alpha)$ such that the following hold.
\begin{enumerate}
	 \item $g_\alpha(0) = 0$ and $(x,g_\alpha(x)) \in B_X(0,\delta_1) \times B_{Y_\alpha}(0,\delta_1) \subseteq U_\alpha$ for all $x \in B_X(0,\delta_1)$.
	 \item $f_\alpha(x,g_\alpha(x)) = 0$ for all $x \in B_X(0,\delta_1)$, and if $(x,y) \in B_X(0,\delta_1) \times B_{Y_\alpha}(0,\delta_1)$ satisfy $f_\alpha(x,y) = 0$, then $y = g_\alpha(x)$. Furthermore, there exists a constant $M > 0$ for which
	 \begin{align}\label{eq: uniform solution bound}
		\sup_{\alpha \in (0,1)}  \sup_{x \in B_X(0,\delta_1)} \norm{g_\alpha(x)}_{Y} \le M.
	 \end{align}
	
\end{enumerate}
\end{thm}
\begin{proof}

For any $\alpha \in (0,1)$ consider the function $F_\alpha: U_\alpha \to  X \times Z$ defined via $F_\alpha(x, y) = (x, f_\alpha(x,y))$. Then $F_\alpha \in C^2(U_\alpha; X \times Z)$ and $DF_\alpha \in \mathcal{L}( X \times Y_\alpha ; X \times Z)$ may be represented in matrix form by  
\begin{align}
	 DF_\alpha(x,y) = \begin{pmatrix}
		 I_{X} & 0_{Y_\alpha}  \\
		 D_1f_\alpha(x,y) & D_2 f_\alpha(x,y)
	 \end{pmatrix}.
\end{align}
Since $D_2f_\alpha(0,0)$ is a linear homeomorphism, we readily conclude that $DF_\alpha(0,0)$ is also a linear homeomorphism. Thus, we may apply the standard inverse function theorem to conclude that $F_\alpha$ is a local $C^2$-diffeomorphism around 0. Note that $DF_\alpha$ is then locally invertible with 
\begin{align}\label{eq: DF inverse}
	 (DF_\alpha(x,y))^{-1} = \begin{pmatrix}
		 I_X & 0_{Y_\alpha}  \\
		 - (D_2 f_\alpha(x,y))^{-1} D_1 f_\alpha(x,y) & (D_2 f_\alpha(x,y))^{-1}
	 \end{pmatrix}
\end{align}
for all $(x,y)$ in a sufficiently small neighborhood of $(0,0)$. Combining the expression \eqref{eq: DF inverse} with \eqref{eq: implicit ests}, we may then conclude that $F_\alpha$ also satisfies the estimate \eqref{eq: inverse fn est} and the rest of the hypothesis of Theorem~\ref{thm: inverseFT}. Thus, by Theorem~\ref{thm: inverseFT} there exists $\delta_1, \delta_2, \delta_3 > 0$ such that for all $\alpha \in (0,1)$, there exists an open set $V_\alpha$ such that we have $(0,0) \in B_X(0,\delta_1) \times B_{Y_\alpha}(0,\delta_1) \subseteq V_\alpha \subseteq B_X(0,\delta_2) \times B_{Y_\alpha}(0,\delta_2) \subseteq U_\alpha$ and $F_\alpha \rvert_{V_\alpha}: V_\alpha \to F_\alpha(V_\alpha) = B_X(0,\delta_3) \times B_Z(0,\delta_3)$ is a $C^2_b$-diffeomorphism and a bi-Lipschitz homeomorphism. Furthermore, the $C^0_b(V_\alpha; X \times Z)$ norm of $F_\alpha$ and the $C^0_b(F_\alpha(V_\alpha); X\times Y_\alpha)$ norm of $F^{-1}_\alpha$ are independent of $\alpha$. 

\begin{figure}[!ht]
	\includegraphics[scale=0.8]{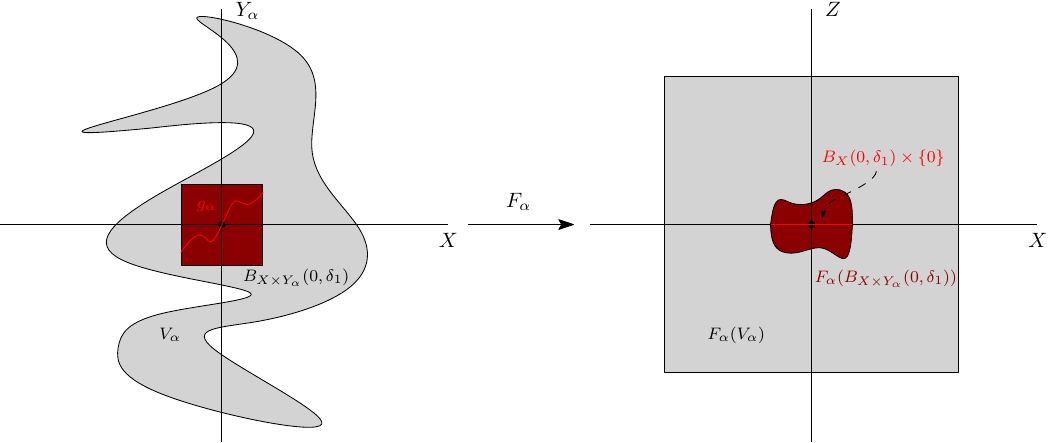}
	\caption{A toy picture of the sets $V_\alpha, F_\alpha(V_\alpha)$ and the $\alpha$-independent ``core'' $B_{X}(0,\delta_1)$}
\end{figure}

Now we propose to define the function $g_\alpha \in C^2_b(B_X(0,\delta_1)  ; Y_\alpha) \cap C_b^{0,1}(B_X(0,\delta_1) ; Y_\alpha)$ by $g_\alpha(\cdot) = G_\alpha(\cdot, 0)$, where the function $G_\alpha: F_\alpha(V_\alpha) \to Y_\alpha$ is defined via $G_\alpha = \pi_2 \circ F_\alpha^{-1} \in C^2_b(F_\alpha(V_\alpha); Y_\alpha) \cap C_b^{0,1} (F_\alpha(V_\alpha); Y_\alpha)$.  To prove the first item, we note that since $(0,0) \in V_\alpha$ and $F_\alpha(0,0) = (0, f_\alpha(0,0)) = (0,0)$, this immediately implies that $g_\alpha(0) = G_\alpha(0,0) = \pi_2 \circ F_\alpha^{-1}(0,0) = 0$. Next we note that by the definition of $F_\alpha$ we have $(x,0) \in F_\alpha(B_X(0,\delta_1) \times B_{Y_\alpha}(0,\delta_1))$ for all $x \in B_X(0,\delta_1)$. This implies that $g_\alpha(x) = G_\alpha(x,0) \in \pi_2(B_X(0,\delta_1) \times B_{Y_\alpha}(0,\delta_1)) =  B_{Y_\alpha}(0,\delta_1)$ for all $x \in B_X(0,\delta_1)$.

To prove the second item, we note that by construction we have $F_\alpha^{-1}(x,0) = (x,G_\alpha(x,0))$ for all $x \in B_X(0,\delta_1)$. Therefore
\begin{align}
	(x, f_\alpha(x,g_\alpha(x))) = (x, f_\alpha(x,G_\alpha(x,0))) = F_\alpha(x, G_\alpha(x,0)) = F_\alpha \circ F_\alpha^{-1} (x,0) = (x,0) \; \text{for all} \; x \in B_X(0,\delta_1),
\end{align}
which implies that $f_\alpha(x,g_\alpha(x)) = 0$ for all $ x\in B_X(0,\delta_1)$. Moreover, if $f_\alpha(x,y) = 0$ for $x \in B_X(0,\delta_1)$ and $y \in B_{Y_\alpha}(0,\delta_1)$, then $F_\alpha(x,y) = (x, f_\alpha(x,y)) = (x,0)$ and so $(x,y) = F_\alpha^{-1}(x,0)$. This in turn implies that $y = G_\alpha(x,0) = g_\alpha(x)$. 

Furthermore, since the $C^0_b(F_\alpha(V_\alpha); X\times Y_\alpha)$ norm of $F^{-1}_\alpha$ is independent of $\alpha$ and $Y_\alpha$ is a closed subspace of $Y$, we may conclude that there exists a positive constant $M$ for which \eqref{eq: uniform solution bound} holds.
\end{proof}

\subsection{Smoothness of composition operators between Sobolev spaces}

In this subsection we record composition results involving the flattening map $\mathfrak{F}$ defined via \eqref{eq:flattening}. 

\begin{thm}\label{comp_C2}
	Let $\N \ni k \ge 1 + \lfloor n/2 \rfloor$, $d \ge 1$, and $m \in \{0,1,2\}$. Let $\varphi \in C^\infty_b(\R;\R)$ be as in \eqref{eq:flattening}, and for every $\eta \in X^{k+1/2}(\R ; \R)$ define the map $\mathfrak{F}: \R^n \to \R^n$ via $\mathfrak{F}(x) = x + \varphi(x_n)\eta(x') e_n$. Then there exists a $0 < \delta_* < 1$ for which the map $
		\Lambda : H^{k+m}(\R^n;\R^d) \times B_{X^{k+1/2}(\R^{n-1}; \R)}(0,\delta)
		\to H^k(\R^n; \R^d)$
	defined via $\Lambda(f,\eta) = f\circ \mathfrak{F}$ is well-defined and the following hold.
	\begin{enumerate}
		\item For all $m \in \{0,1,2\}$, $\Lambda$ is continuous.
		\item If $m=1$, then $\Lambda$ is $C^1$ and satisfies $D \Lambda(f,\eta) (f_1,\eta_1) = (\p_n f \circ \mathfrak{F})\varphi \eta_1 + f_1 \circ \mathfrak{F}$.
		\item If $m=2$, then  $\Lambda$ is $C^2$ and satisfies $D^2 \Lambda(f,\eta) [(f_1,\eta_1),(f_2,\eta_2)] =  (D^2 f \circ \mathfrak{F}) (\varphi \eta_1 e_n, \varphi \eta_2 e_n) + (\p_n f_1\circ \mathfrak{F}) \varphi \eta_2 + (\p_n f_2 \circ \mathfrak{F}) \varphi \eta_1$.
	\end{enumerate}
\end{thm}
\begin{proof}
	The first and second items follow from Theorem 5.20 in \cite{leonitice}, and a close examination of the proof therein shows that the argument can be extended prove the third item. 

\end{proof}

Next we prove a variant of Theorem~\ref{comp_C2} for compositions between $C^k_b$ functions and Sobolev functions.
 
\begin{thm}\label{thm: smooth Sobolev composition}
Let $\Omega \subseteq \R^n$ be an extension domain and assume $\N \ni k \ge 2 + \tfloor{n/2}, m \in \{0,1,2\}$. Let $f \in C^{k+1+ m}_b(\R^n;\R^n)$ and assume $f(0) = 0$ if $\Omega$ has infinite measure. Then the map $\Lambda_f: H^k(\Omega; \R^n) \to H^{k}(\Omega;\R^n)$ defined via $\Lambda_f(u) = f \circ u$ is well-defined and $C^m$. 
\end{thm}
\begin{proof}
We prove this in four steps.

\textbf{Step 1}: A multiplier estimate. We first prove via finite induction the statement $\mathbb{P}_j$ for all $0 \le j \le k$, where $\mathbb{P}_j$ denotes the proposition that for all $f \in C^{j+1}_b(\R^n; \R^n)$ and $u,v \in H^k(\Omega;\R^n)$, we have the a priori estimate 
\begin{align}\label{eq: H^k product est}
	\norm{\Lambda_f(u) v}_{H^j} \lesssim  \norm{f}_{C^{j}_b} \langle \norm{u}_{H^k} \rangle^{j} \norm{v}_{H^j}  \; \text{for all \;} u, v\in H^k(\Omega;\R^n),
\end{align}
where $\langle \cdot \rangle = (1 + \abs{\cdot}^2)^{1/2}$ denotes the Japanese bracket. 

In the case when $j = 0$, by the supercritical Sobolev embedding $H^{1+\tfloor{n/2}}(\Omega;\R^n) \hookrightarrow C^0_b(\Omega;\R^n)$ and the assumption that $f \in C^1_b(\R^n; \R^n)$, the estimate \eqref{eq: H^k product est} is satisfied trivially. Thus $\mathbb{P}_0$ holds.

Next we proceed inductively and suppose that $\mathbb{P}_l$ holds for all $0 \le l \le j \le k-1$ and consider the case $j+1 \le k$. We first note that for any $1 \le p \le n$ and $u,v \in H^k(\Omega;\R^n)$ we have 
\begin{equation}\label{eq: product derivative}
	\partial_p \Lambda_f(u) = \sum_{q=1}^n \Lambda_{\partial_q f}(u) \partial_p (u)_q, \; \p_p (\Lambda_f(u) v) = \Lambda_f(u) \p_p v + \sum_{q=1}^n \Lambda_{\partial_q f}(u) \partial_p (u)_q v.
\end{equation}
Then by the supercritical Sobolev embedding, applying the estimate \eqref{eq: product derivative} from the induction hypothesis on $f \in C_b^{j+2}(\R^n ; \R^n)$, $\p_q f \in C^{j+1}_b(\R^n; \R^n)$ and the standard Sobolev product estimate with the fact that $k-1 \ge 1 + \tfloor{n/2}$, we have
\begin{multline}
	\norm{\Lambda_f(u)v}_{H^{j+1}}   \lesssim \norm{\Lambda_f(u) v}_{H^{0}} + \sum_{p=1}^n \tnorm{\partial_p (\Lambda_f(u)v)}_{H^{j}} \lesssim \norm{\Lambda_f(u) v}_{H^{0}} + \sum_{p=1}^n \norm{\Lambda_f(u) \p_p v}_{H^j} \\  + \sum_{p,q=1}^n \norm{\Lambda_{\p_q f}(u) \p_p (u)_q v}_{H^j} 
   \lesssim \norm{f}_{C^0_b} \norm{v}_{H^0} + \norm{f}_{C^j_b} \langle \norm{u}_{H^k} \rangle^j \norm{v}_{H^{j+1}}    + \sum_{p,q=1}^n  \norm{f}_{C^{j+1}_b} \langle \norm{u}_{H^k}\rangle^j \norm{\p_p(u)_q v}_{H^j} \\
   \lesssim    \norm{f}_{C^0_b} \norm{v}_{H^0} + \norm{f}_{C^j_b} \langle\norm{u}_{H^k}\rangle^j \norm{v}_{H^{j+1}}    + \norm{f}_{C^{j+1}_b} \langle\norm{u}_{H^k}\rangle^j \norm{u}_{H^k} \norm{v}_{H^j} \lesssim \norm{f}_{C^{j+1}_b}\langle\norm{u}_{H^k}\rangle^{j+1} \norm{v}_{H_{j+1}},   
\end{multline}
which shows $\mathbb{P}_{j+1}$ holds. This completes the induction argument. 

\textbf{Step 2}: A difference estimate. Next we use the multiplier estimate from the previous step to prove the statement $\Q_j$ for $0 \le j \le k$, where $\mathbb{Q}_j$ denotes the proposition that for all $f \in C^{j+1}_b(\R^n ; \R^n)$ and $u,v \in H^k(\Omega; \R^n)$, the difference $\Lambda_f(u) - \Lambda_f(v) \in H^j(\Omega; \R^n)$ and satisfies 
\begin{align}\label{eq: Lambda continuity}
	\norm{\Lambda_f(u) - \Lambda_f(v)}_{H^j} \to 0 \; \text{if} \; v \to u \; \text{in} \; H^k(\Omega;\R^n). 
\end{align}
To prove $\mathbb{Q}_j$ for each admissible $j$ we proceed by finite induction again. 

In the case when $j=0$, we note that by applying the mean value inequality we may  deduce that 
\begin{align}\label{eq: continuity j=0}
	 \norm{\Lambda_f(u) - \Lambda_f(v)}_{H^0} \le \norm{f}_{C^1_b} \norm{u-v}_{H^0} \le \norm{f}_{C^1_b} \norm{u-v}_{H^k}.
\end{align}
Thus $\Lambda_f(u) - \Lambda_f(v) \in H^0(\Omega;\R^n)$ and $ \norm{\Lambda_f(u) - \Lambda_f(v)}_{H^0} \to 0$ as $v \to u \in H^k(\Omega;\R^n)$. Thus $\mathbb{Q}_0$ holds. 

Next we suppose that $\mathbb{Q}_l$ holds for all $0 \le l \le j \le k-1$ and consider the case $j+1 \le k$. We note that \eqref{eq: product derivative} implies that for all $1 \le p \le n$, 
\begin{multline}
	\partial_p(\Lambda_f(u) - \Lambda_f(v) )
	= \sum_{q=1}^n \Lambda_{\partial_q f}(u) \partial_p (u)_q - \Lambda_{\partial_q f}(v) \partial_p (v)_q \\ 
	= \sum_{q=1}^n \left( \Lambda_{\partial_q f}(u) - \Lambda_{\partial_q f}(v) \right) \partial_p (u)_q + \sum_{q=1}^n  \Lambda_{\partial_q f}(v) \partial_p (u-v)_q \; \text{for all} \; u,v \in H^k(\Omega;\R^n),
\end{multline}
thus by \eqref{eq: continuity j=0}, the multiplier estimate \eqref{eq: H^k product est}, the induction hypothesis and basic product estimates, we have
\begin{multline}\label{eq: difference estimate}
	\norm{\Lambda_f(u) - \Lambda_f(v)}_{H^{j+1}} =  \norm{\Lambda_f(u) - \Lambda_f(v)}_{H^{0}} + \sum_{p=1}^n \norm{\partial_p(\Lambda_f(u) - \Lambda_f(v) )}_{H^j} \\
	\lesssim  \norm{\Lambda_f(u) - \Lambda_f(v)}_{H^{0}} + 
	\sum_{p,q=1}^n \norm{\left( \Lambda_{\partial_q f}(u) - \Lambda_{\partial_q f}(v) \right) \partial_p (u)_q}_{H^j}+ \sum_{p,q=1}^n  \norm{\Lambda_{\partial_q f}(v) \partial_p (u-v)_q}_{H^j} \\
	\lesssim  \norm{\Lambda_f(u) - \Lambda_f(v)}_{H^{0}} +  \norm{u}_{H^k} \sum_{q=1}^n \norm{\Lambda_{\p_q f}(u) -\Lambda_{\p_q f}(v) }_{H^j} 
	+ \sum_{q=1}^n \norm{\p_q f}_{C^{j}_b} \langle \norm{v}_{H^k} \rangle^{j} \norm{u-v}_{H^{j+1}}.
\end{multline}
This shows that $\Lambda_f(u) - \Lambda_f(v) \in H^{j+1}(\Omega;\R^n)$ and $\norm{\Lambda_f(u) - \Lambda_f(v)}_{H^{j+1}} \to 0$ if $v \to u \in H^{j+1}(\Omega;\R^n)$.  
Thus $\mathbb{Q}_{j+1}$ holds and the induction argument is complete. 

\textbf{Step 3}: Well-definedness and continuity. Next we utilize the result from the previous step to show that $\Lambda_f: H^k(\Omega;\R^n) \to  H^k(\Omega;\R^n)$ is well-defined and continuous. We note that in the case when $\Omega$ has infinite measure, using the additional assumption $f(0) = 0$ we may apply the mean value inequality to deduce that 
\begin{align}
	\abs{\Lambda_f(u)} = \abs{f(u) - f(0)} \le \norm{Df}_{C^0_b} \abs{u} \; \text{for all} \; u \in H^k(\Omega;\R^n),
\end{align}
which in turn implies that 
\begin{align}\label{eq: comp H^0 est}
	\norm{\Lambda_f(u)}_{H^0} \lesssim  \norm{f}_{C^1_b}  \norm{u}_{H^0}  \lesssim \norm{f}_{C^1_b} \norm{u}_{H^k}  \; \text{for all} \; u \in H^k(\Omega;\R^n).
\end{align}
In the case when $\Omega$ has finite measure, the estimate \eqref{eq: comp H^0 est} also holds. In either case, we may use $v = 0 \in H^k(\Omega;\R^n)$ in the estimate \eqref{eq: difference estimate} for $j+1 = k$ from the previous step and \eqref{eq: comp H^0 est} to immediate deduce that
\begin{multline}
	\norm{\Lambda_f(u)}_{H^k} \lesssim \norm{\Lambda_f(u)}_{H^{0}} +  \norm{u}_{H^k} \sum_{q=1}^n \norm{\Lambda_{\p_q f}(u) -\Lambda_{\p_q f}(0) }_{H^{k-1}} 
	+ \sum_{q=1}^n \norm{\p_q f}_{C^{j}_b}  \norm{u}_{H^{k}} \\ 
	\lesssim \norm{f}_{C^1_b} \norm{u}_{H^k} +   \norm{u}_{H^k} \sum_{q=1}^n \norm{\Lambda_{\p_q f}(u) -\Lambda_{\p_q f}(0) }_{H^{k-1}} 
	+ \norm{f}_{C^{k}_b}  \norm{u}_{H^{k}} \; \text{for all} \; u \in H^k(\Omega;\R^n).
\end{multline}
Since $\p_q f \in C^k(\R^n;\R^n)$, from $\Q_{k-1}$ we know that $\Lambda_{\p_q f}(u) -\Lambda_{\p_q f}(0) \in H^{k-1}$ for all $1 \le q \le n$, therefore $\Lambda_f(u) \in H^k(\Omega;\R^n)$ for all $u \in H^k(\Omega;\R^n)$. From $\mathbb{Q}_k$ we also have $\norm{\Lambda_f(u) - \Lambda_f(v)}_{H^k} \to 0$ if $v \to u$ in $H^k(\Omega;\R^n)$, thus we may conclude that the map $\Lambda_f: H^k(\Omega;\R^n) \to  H^k(\Omega;\R^n)$ is well-defined and continuous.

\textbf{Step 4:} Continuous differentiability. 
To conclude the proof we show that $\Lambda_f$ is $C^m$ given $f \in C^{k+1+m}_b(\R^n ; \R^n)$ for $m \in \{1,2\}$. In the case when $m=1$, we note that by the fundamental theorem of calculus for all $u,v \in H^k(\Omega;\R^n)$ we have 
\begin{align}
	 \Lambda_f (u + v) - \Lambda_f(u) - \sum_{q=1}^n \Lambda_{\p_q f}(u) (v)_q =  \underbrace{\sum_{q=1}^n (v)_q \int_0^1 \Lambda_{\p_q f}(u+tv) - \Lambda_{\p_q f}(u) \; dt}_{ := \mathcal{R}_1},
\end{align}
and thus by the fact that $H^k(\Omega;\R^n)$ is an algebra and applying the statement $\mathbb{Q}_k$ from the induction argument above to $\p_q f \in C_b^{k+1}(\R^n ; \R^n)$, we have 
\begin{align}
\frac{\norm{\mathcal{R}_1}_{H^k}}{\norm{v}_{H^k}} 
\lesssim \sum_{q=1}^n \int_0^1 \norm{\Lambda_{\p_q f} (u+tv) - \Lambda_{\p_q f}(u)}_{H^k} \; dt \to 0 \; \text{as} \; \norm{v}_{H^k} \to 0. 
\end{align}
Thus we may conclude that $\Lambda_f$ is differentiable when $m=1$ and 
\begin{align}\label{eq: iterate step}
	 D \Lambda_f (u) (v) = \sum_{q=1}^n \Lambda_{\p_q f}(u) (v)_q. 
\end{align}
Since $D \Lambda_f(u)$ is in terms $\Lambda_{\p_q f}$ which satisfies \eqref{eq: Lambda continuity}, we may then conclude that $D\Lambda_f$ is continuously differentiable. 

To conclude in the case of $m=2$, we note that by \eqref{eq: iterate step} and the fundamental theorem of calculus again we have 
\begin{align}
	 D\Lambda_f(u+w)(v) - D \Lambda_f(u)(v) - \sum_{p,q=1}^n \Lambda_{\p_p \p_q f} (u)(v)_p(w)_q = \underbrace{\sum_{p,q=1}^n  (v)_p(w)_q \int_0^1 \Lambda_{\p_p \p_q f}(u+tw) - \Lambda_{\p_p \p_q f} (u) \; dt}_{:= \mathcal{R}_2}.
\end{align}
Using the fact that $H^k(\Omega;\R^n)$ is an algebra and applying the statement $\mathbb{Q}_k$ from the induction argument above to $\p_p \p_q f \in C_b^{k+1}(\R^n ; \R^n)$, we then have 
\begin{align}
 \frac{\norm{\mathcal{R}_2}_{\mathcal{L}(H^k)}}{\norm{w}_{H^k} } \lesssim \sum_{p,q=1}^n \int_0^1 \norm{\Lambda_{\p_p \p_q f}(u+tw) - \Lambda_{\p_p \p_q f} (u)}_{H^k} \; dt \to 0 \; \text{as} \; \norm{w}_{H^k} \to 0.
\end{align}
This shows that $\Lambda_f$ is twice-differentiable with
\begin{align}
	 D^2 \Lambda_f (u)(v,w) = \sum_{p,q=1}^n \Lambda_{\p_p \p_q f}(u)(v)_p (w)_q.
\end{align}
Since $\Lambda_{\p_p \p_q f}$ satisfies \eqref{eq: Lambda continuity}, we may then conclude that $\Lambda_f$ is $C^2$ when $m=2$.

\end{proof}

\bibliographystyle{abbrv}
\bibliography{bib.bib}

\end{document}